\documentclass[11pt]{article}
\usepackage{amsfonts, amsmath, amssymb, amscd, amsthm, color, graphicx, mathrsfs, mathabx, wasysym, setspace, mdwlist, calc,float}
\usepackage{setspace}
\usepackage{hyperref}
\usepackage{tikz-cd} 
\usepackage{hyperref}
\usepackage{tocloft}
\usepackage{xcolor}

\usepackage{hyperref}
\hypersetup{linktocpage}

\hypersetup{colorlinks,
    linkcolor={red!50!black},
    citecolor={blue!80!black},
    urlcolor={blue!80!black}}
\usepackage{float}

\newcommand{\NN}{\mathbb{N}}

\newcommand{\ZZ}{\mathbb{Z}}
\newcommand{\RR}{\mathbb{R}}

\newcommand{\N}{\mathcal{N}}
\newcommand{\T}{\mathcal{T}}

\newtheorem{thm}{Theorem}[section]
\newtheorem{cor}[thm]{Corollary}
\newtheorem{lem}[thm]{Lemma}
\newtheorem{prop}[thm]{Proposition}

\theoremstyle{definition}
\newtheorem{defn}[thm]{Definition}
\newtheorem{conv}[thm]{Convention}

\theoremstyle{remark}
\newtheorem{rem}[thm]{Remark}

\newcommand{\stab}{\mathrm{Stab}}
\newcommand{\lk}{\mathrm{Lk}}
\newcommand{\st}{\mathrm{St}}
\newcommand{\elk}{\mathrm{Elk}}
\newcommand{\supp}{\mathrm{supp}}
\newcommand{\diam}{\mathrm{diam}}
\newcommand{\asdim}{\mathrm{asdim}}
\newcommand{\girth}{\mathrm{girth}}
\newcommand{\leaf}{\mathrm{leaf}}
\newcommand{\geo}{\mathrm{Geo}}

\renewcommand{\N}{\mathcal{N}}
\newcommand{\C}{\mathscr{C}}

\newcommand{\la}{\langle}
\newcommand{\ra}{\rangle}

\renewcommand{\H}{\mathcal{H}}

\newcommand{\A}{\mathscr{A}}

\newcommand{\G}{\mathcal{G}}

\newcommand{\act}{\curvearrowright}

\newcommand{\hGammae}{\widehat{\Gamma}^e}

\newcommand\blfootnote[1]{%
  \begingroup
  \renewcommand\thefootnote{}\footnote{#1}%
  \addtocounter{footnote}{-1}%
  \endgroup
}

\begin{document}

\title{Infinite graph product of groups I: Geometry of the extension graph}

\author{Koichi Oyakawa}
\date{}

\maketitle

\vspace{-3mm}

\begin{abstract}
We introduce the extension graph of graph product of groups and study its geometry. This enables us to study properties of graph product by exploiting large scale geometry of its defining graph. In particular, we show that the extension graph is isomorphic to the crossing graph of a canonical quasi-median graph and exhibits the same phenomenon about asymptotic dimension as quasi-trees of metric spaces studied by Bestvina-Bromberg-Fujiwara. As an application of the extension graph, we prove relative hyperbolicity of graph wreath product. This provides a new construction of relatively hyperbolic groups.
\end{abstract}

\tableofcontents

\section{Introduction}
\label{sec:Introduction}
\blfootnote{\textbf{MSC 2020} Primary: 20F65. Secondary: 20F67, 51F30.}
\blfootnote{\textbf{Key words and phrases}: graph product, asymptotic dimension, relatively hyperbolic groups.}

Given a simplicial graph $\Gamma$ and a collection of groups $\G = \{G_v\}_{v \in V(\Gamma)}$ assigned to each vertex of $\Gamma$, the graph product $\Gamma\G$ is a group obtained by taking quotient of the free product $\ast_{v \in V(\Gamma)} G_v$ by setting that group elements of two adjacent vertices commute (see Definition \ref{def:graph product of groups}). Graph product is generalization of both direct product (when $\Gamma$ is a complete graph) and free product (when $\Gamma$ has no edge). Also, famous classes such as right angled Artin groups (when $|\Gamma|<\infty$ and $G_v\equiv \ZZ$) and right angled Coxeter groups (when $|\Gamma|<\infty$ and $G_v \equiv \ZZ/2\ZZ$) are obtained as graph product. Therefore, properties of graph product have attracted interest of many people and a lot of research has been done.

However, most results so far concern only the case where the defining graph $\Gamma$ is finite, or the results just hold for any defining graph. To the best of my knowledge, there is no systematic study about how the geometry of a defining graph affects properties of graph product of groups and the purpose of this paper is to investigate this direction of research. In order to capture the geometry of a defining graph, it is not effective to consider graph product as amalgamated free product or to consider Cayley graphs, which are standard methods so far to study graph product. Instead, we achieve it by studying a graph $\Gamma^e$ that remembers the geometry of the defining graph $\Gamma$ and also admits an action by the graph product. This graph $\Gamma^e$, which we call the extension graph (see Definition \ref{def:extension graph}), was introduced by various people (see \cite{Gen17, CRKdlNG24, EH24}) and generalizes the extension graph constructed for right angled Artin groups (RAAGs) by Kim and Koberda (see \cite{KK13} and \cite{KK14}).

However, there is a crucial difference in the general case of graph product from that of RAAGs. That is, the extension graph $\Gamma^e$ in general case cannot be obtained by repeatedly taking the double of a graph along a star unlike the case of RAAGs. More precisely, \cite[Lemma 22]{KK13}, which is a key lemma in \cite{KK13}, is no longer true in general (see Lemma \ref{lem:doubling is not true}). Therefore, we develop a new way to study the extension graph of general graph product, which we do in Section \ref{sec:Geometry of the extension graph}. For a technical reason, we have to assume that the girth of a defining graph is more than 20 in many of our results. Although this condition on the girth may well not be optimal, it is often essential in our applications. Also, we can increase the girth of a graph easily by taking its barycentric subdivision and assigning vertex groups to the new vertices. For various groups with features of non-positive curvature, we can also make a Cayley graph have arbitrarily large girth by choosing a nice finite generating set (see \cite{Ak05,Yam11,Nak14,Ak25}). 

Moreover, it turns out that our extension graph $\Gamma^e$ and its coned-off graph $\hGammae$ (see Definition \ref{def:coned-off extension graph}) are isomorphic to the crossing graph and the contact graph respectively of a canonical quasi-median graph associated to graph product (see Section \ref{subsec:Connection to the crossing and contact graphs}). This generalizes the observation of Kim and Koberda in the case of the extension graph of RAAGs (see \cite[Section 7]{KK14}, \cite[Section 8.5]{Gen17}). Therefore, this paper provides a new perspective to study the crossing graph and the contact graph by using the geometry of a defining graph $\Gamma$. (Another way to study these graphs is to use the geometry of a quasi-median Cayley graph of graph product groups (see \cite{Gen17, Val21}).)

One feature of the Kim-Koberda's extension graph of right angled Artin groups was that it is quasi-isometric to a tree when the defining graph is connected (see \cite[Lemma 26 (7)]{KK13}). This is no longer true for general graph product once we allow infinite defining graphs. It is because the defining graph $\Gamma$, which we assume is connected throughout this paper, isometrically embeds into the extension graph $\Gamma^e$ (see Corollary \ref{cor:Gamma is embedded into Gammae}). (Here, vertex groups are not very relevant. For example, the extension graph for infinite countable vertex groups is isomorphic to the extension graph of a RAAG with the same defining graph (see \cite[Fact 8.25]{Gen17}).) However, this fact for RAAGs means that when the asymptotic dimension of a defining graph is 0, the asymptotic dimension of the extension graph is at most 1. We generalize this as follows.
\begin{thm}\label{thm:intro asymptotic dimension}
     Suppose that $\Gamma$ is a connected simplicial graph with $\girth(\Gamma) > 20$ and that $\{G_v\}_{v \in V(\Gamma)}$ is a collection of non-trivial groups. If $\asdim(\Gamma) \le n$ with $n \in \NN \cup \{0\}$, then $\asdim(\Gamma^e) \le n+1$.
\end{thm}
Intuitively, Theorem \ref{thm:intro asymptotic dimension} shows that the extension graph, or the crossing graph, of graph product of groups is obtained by pasting copies of the defining graph in a `tree-like' way. Moreover, Theorem \ref{thm:intro asymptotic dimension} can also be considered as an analogue of \cite[Theorem 4.24]{BBF15}, where they studied the asymptotic dimension of quasi-trees of metric spaces.

Among other geometric properties, we focus on the case where a defining graph is hyperbolic in this paper. It turns out that the extension graph of graph product exhibits similar properties as the curve complex of a surface or the coned-off Cayley graph of a relatively hyperbolic group under natural conditions. More precisely, we show Theorem \ref{thm:intro extension graph is tight} below (see Definition \ref{def:fine graph} for fine graphs and uniformly fine graphs). Examples of fine (resp. uniformly fine) hyperbolic graphs include trees, which don't need to be locally finite, and locally finite (resp. uniformly locally finite) hyperbolic graphs. In particular, Theorem \ref{thm:intro extension graph is tight} can be applied to the case of finite defining graphs as well, because finite connected graphs are uniformly fine and hyperbolic.
\begin{thm}\label{thm:intro extension graph is tight}
    Suppose that $\Gamma$ is a connected simplicial graph with $\girth(\Gamma)>20$ and that $\G = \{G_v\}_{v \in V(\Gamma)}$ is a collection of non-trivial groups. Then, the following hold.
    \begin{itemize}
        \item[(1)]
        $\Gamma$ is hyperbolic if and only if $\Gamma^e$ is hyperbolic.
        \item[(2)]
        If $\Gamma$ is uniformly fine and hyperbolic, then $\Gamma^e$ is tight in the sense of Bowditch.
        \item[(3)]
        If $\Gamma$ is fine and $\{G_v\}_{v\in V(\Gamma)}$ is a collection of non-trivial finite groups, then $\Gamma^e$ is fine.
    \end{itemize}
\end{thm}
Tightness was first introduced by Masur and Minsky in \cite{MM00} in the context of the curve graph of surfaces. Bowditch later introduced the notion of tightness that is more abstract than Masur-Minsky and used it to show acylindricity of the action of the mapping class group of a surface on the curve graph (see \cite{Bow08}). This notion was also used to show that the curve complex has finite asymptotic
dimension (see \cite{BelF08}) and Yu's Property A (see \cite{Ki08}). Theorem \ref{thm:intro extension graph is tight} (2) is new even in the case of RAAGs. It is also worth mentioning that Theorem \ref{thm:intro extension graph is tight} (3) provides a new construction of fine graphs, while tress and coned-off Cayley graphs of relatively hyperbolic groups have been main examples of fine hyperbolic graphs so far.

In this paper, we present one application of Theorem \ref{thm:intro extension graph is tight} to graph wreath product. In the forthcoming paper, we present more applications to analytic properties of graph product. Given a group $G$ acting on a simplicial graph $\Gamma$ and another group $H$, we can assign vertex groups $\G = \{G_v\}_{v \in V(\Gamma)}$ by setting $G_v = H$ for every $v \in V(\Gamma)$, and the group $G$ acts on the graph product $\Gamma\G$ as group automorphisms by permuting vertex groups according to the action $G \act \Gamma$. This action $G \to \mathrm{Aut}(\Gamma\G)$ defines the semi-direct product $\Gamma\G \rtimes G$, which was called graph wreath product in \cite{KM16}. Interestingly, we can consider this construction as interpolation of wreath product and free product. Indeed, given groups $G$ and $H$, the group $G$ acts on the set $G$ by left multiplication. If $\Gamma$ is the complete graph with $V(\Gamma)=G$, then we have $\Gamma\G \rtimes G = H \wr G$. If $\Gamma$ is the graph with $V(\Gamma)=G$ having no edge, then we have $\Gamma\G \rtimes G = H \ast G \, (= \ast_{g \in G}gHg^{-1} \rtimes G)$.

In \cite{KM16}, Kropholler and Martino studied homotopical finiteness
 conditions, i.e. type $\mathrm{F}_n$, of graph wreath product. In Theorem \ref{thm:intro semidirect product is relatively hyperbolic} and Corollary \ref{cor:intro semidirect product becomes hyperbolic} below, we prove geometric properties of graph wreath product, which provide a new construction of relatively hyperbolic groups and hyperbolic groups.
\begin{thm}\label{thm:intro semidirect product is relatively hyperbolic}
    Suppose that $\Gamma$ is a fine hyperbolic graph with $\girth(\Gamma)>20$ and that a finitely generated group $G$ acts on $\Gamma$ satisfying the following two conditions.
    \begin{itemize}
        \item[(1)]
        $E(\Gamma) / G$ is finite and for any $e \in E(\Gamma)$, $\stab_G(e) \,(=\stab_G(e_-) \cap \stab_G(e_+))$ is finite.
        \item[(2)]
        For any $v \in V(\Gamma)$, $\stab_G(v)$ is finitely generated.
    \end{itemize}
    Let $H$ be a finite group and define $\G=\{G_v\}_{v \in V(\Gamma)}$ by $G_v=H$ for any $v \in V(\Gamma)$. Then, there exists a finite set $F \subset V(\Gamma)$ such that $\Gamma\G \rtimes G$ is hyperbolic relative to the collection $\{\, \la \stab_G(v), G_w \mid w \in \st_\Gamma(v) \ra \,\}_{v \in F}$ of subgroups.
\end{thm}
\begin{cor}\label{cor:intro semidirect product becomes hyperbolic}
    Suppose that $\Gamma$ is a locally finite hyperbolic graph with $\girth(\Gamma)>20$ and that a group $G$ acts on $\Gamma$ properly and cocompactly. Let $H$ be a finite group and define $\G=\{G_v\}_{v \in V(\Gamma)}$ by $G_v=H$ for any $v \in V(\Gamma)$. Then, $\Gamma\G \rtimes G$ is hyperbolic.
\end{cor}
In Corollary \ref{cor:intro semidirect product becomes hyperbolic}, both of the conditions $|H|<\infty$ and $\girth(\Gamma)>20$ are essential, because no hyperbolic group contains $\ZZ^2$ as a subgroup while $\Gamma\G$ can contain $\ZZ^2$ if either $|H|=\infty$ or $\girth(\Gamma)=4$.

It seems interesting to investigate properties of hyperbolic groups constructed in Corollary \ref{cor:intro semidirect product becomes hyperbolic}. For example, studying residual finiteness of these hyperbolic groups might shed light on the long-standing open problem asking whether every hyperbolic group is residually finite, because non residually finite groups can be easily constructed by wreath product. Indeed, wreath product of two groups $H \wr G$ is residually finite if and only if $H$ and $G$ are residually finite, and either $H$ is abelian or $G$ is finite (see \cite[Theorem 3.2]{Gru57}). Even if the hyperbolic groups in Corollary \ref{cor:intro semidirect product becomes hyperbolic} are residually finite, it is still interesting to investigate where this transition occurs as graph wreath product changes from wreath product to free product.

This paper is organized as follows. In Section \ref{sec:Preliminaries}, we explain preliminary definitions and known results that are necessary in this paper. In Section \ref{sec:Geometry of the extension graph}, we define the extension graph of graph product and study its geometry. In Section \ref{sec:Asymptotic dimension of the extension graph}, we study asymptotic dimension of the extension graph and prove Theorem \ref{thm:intro asymptotic dimension} by introducing an auxiliary graph that we call the coned-off extension graph. In Section \ref{sec:Hyperbolicity, tightness, and fineness of the extension graph}, we discuss how properties of a defining graph carry over to the extension graph and prove Theorem \ref{thm:intro extension graph is tight} by using results in Section \ref{sec:Geometry of the extension graph}. In Section \ref{sec:Relative hyperbolicity of semi-direct product}, we discusses the application to graph wreath product and prove Theorem \ref{thm:intro semidirect product is relatively hyperbolic} and Corollary \ref{cor:intro semidirect product becomes hyperbolic}.

\par

\vspace{2mm}

\noindent\textbf{Acknowledgment.}
I would like to thank Anthony Genevois and Denis Osin for helpful discussions. I would also like to thank Peter Kropholler for helpful comments on an earlier draft. I would also like to thank an anonymous referee for suggesting Corollary \ref{cor:stabilizer of two far vertices is trivial} (1) and for helpful comments, which improved exposition of the paper.

\section{Preliminaries}
\label{sec:Preliminaries}

We start with preparing necessary notations about graphs, metric spaces, and group actions. Throughout this paper, we assume that graphs are simplicial (i.e. having no loops nor multiple edges) and a group acts on a graph as graph automorphisms unless otherwise stated.

\begin{defn}\label{def:concepts in graph theory}
    A \emph{graph} $X$ is the pair of a set $V(X)$ and a subset $E(X) \subset V(X)\times V(X)$ satisfying $\forall\, x \in V(X),\, (x,x) \notin E(X)$ and $\forall\, (x,y) \in V(X)^2,\, (x,y) \in E(X) \Leftrightarrow (y,x) \in E(X)$. An element of $V(X)$ is called a \emph{vertex} and an element of $E(X)$ is called an \emph{edge}. For an edge $e=(x,y) \in E(X)$, we denote $x$ by $e_-$ and $y$ by $e_+$, that is, we have $e=(e_-,e_+)$. For a vertex $x\in V(X)$, we define $\lk_X(x), \st_X(x) \subset V(X)$ and $\elk_X(x) \subset E(X)$ by
    \begin{align*}
        \lk_X(x)&=\{y\in V(X) \mid (x,y) \in E(X)\}, \\
        \st_X(x)&=\{x\} \cup \lk_X(x), \\
        \elk_X(x) &= \{ (x,y) \in E(X) \mid y \in \lk_X(x) \}.
    \end{align*}
    We define $\leaf(X)$ by $\leaf(X) = \{x \in V(X) \mid |\lk_X(v)|\le 1 \}$. A \emph{path} $p$ in $X$ is a sequence $p=(p_0,\cdots,p_n)$ of vertices, where $n \in \NN\cup\{0\}$ and $p_i \in V(X)$, such that $(p_i,p_{i+1}) \in E(X)$ for any $i \ge 0$. Given a path $p=(p_0,\cdots,p_n)$ in $X$,
    \begin{itemize}
        \item[-]
        the \emph{length} $|p| \in \NN\cup\{0\}$ of $p$ is defined by $|p|=n$,
        \item[-]
        we denote $p_0$ (the \emph{initial vertex}) by $p_-$ and $p_n$ (the \emph{terminal vertex}) by $p_+$,
        \item[-]
        we define $V(p)=\{p_i \mid 0 \le i \le n\}$ and $E(p) = \{(p_{i-1},p_i),(p_i,p_{i-1}) \mid 1 \le i \le n\}$,
        \item[-]
        a \emph{subsequence of} $V(p)$ is $(p_{i_0},\cdots,p_{i_m})$ with $m \in \NN\cup\{0\}$ such that $i_0 \le \cdots \le i_m$,
        \item[-]
        we say that $p$ \emph{has backtracking at} $p_i$ if $p_{i-1} = p_{i+1}$.
    \end{itemize}
    A \emph{loop} $p$ in $X$ is a path such that $p_-=p_+$. A \emph{circuit} $p=(p_0,\cdots,p_n)$ in $X$ is a loop with $|p|>2$ and without self-intersection except $p_-=p_+$ i.e. $p_i\neq p_j$ for any $i,j$ with $0 \le i <j <n$. For $e \in E(X)$ and $n\in \NN$, we define $\C_X(e,n)$ to be the set of all circuits in $X$ that contain $e$ and have length at most $n$. The \emph{girth} of $X$ $\girth(X) \in \NN$ is defined by $\girth(X)=\min\{|p| \mid \text{$p$ is a circuit in $X$}\}$ if there exists a circuit in $X$. If there is no circuit in $X$, then we define $\girth(X)=\infty$ for convenience. A graph is called \emph{connected} if for any $x,y \in V(X)$, there exists a path $p$ such that $p_-=x$ and $p_+=y$. When a graph $X$ is connected, $X$ becomes a geodesic metric space by a graph metric $d_X$ (i.e. every edge has length 1), hence we also denote this metric space by $X$. Given $L \subset V(X)$, the \emph{induced subgraph on} $L$ is defined by the vertex set $L$ and the edge set $E(X)\cap L^2$.
\end{defn}

\begin{defn}
    Let $X$ be a connected graph. A path $p$ in $X$ is called \emph{geodesic} if $|p|$ is the smallest among all paths from $p_-$ to $p_+$. For $x,y \in V(X)$, we denote by $\geo_X(x,y)$ the set of all geodesic paths in $X$ from $x$ to $y$. For a path $p=(p_0, \cdots, p_n)$ without self-intersection and $i,j$ with $0 \le i \le j \le n$, we denote the subpath $(p_i, \cdots, p_j)$ of $p$ by $p_{[p_i, p_j]}$.
\end{defn}

\begin{rem}
    Since we consider only simplicial graphs throughout this paper, the girth of a graph is always at least 3.
\end{rem}

\begin{rem}\label{rem:vertex with 0 link}
    When a graph $X$ is connected, $x\in V(X)$ satisfies $|\lk_\Gamma(x)| = 0$ if and only if $V(X)=\{x\}$.
\end{rem}

\begin{defn}
    Let $(X,d_X)$ be a metric space. For a subset $A \subset X$, the \emph{diameter} of $A$ $\diam_X(A) \in [0,\infty]$ is defined by $\diam_X(A) = \sup_{x,y \in A}d_X(x,y)$. For $A \subset X$ and $r \in \RR_{>0}$, we define $\N_X(A,r)\subset X$ by $\N_X(A,r) = \{y \in X \mid \exists \, x\in A, d_X(x,y) \le r\}$. When $A$ is a singleton i.e. $A=\{x\}$ with $x \in X$, we denote $\N_X(\{x\},r)$ by $\N(x,r)$ for brevity, that is, $\N(x,r)=\{y \in X \mid d_X(x,y) \le r\}$. For two subsets $A,B \subset X$, we define $d_X(A,B) \in \RR_{\ge0}$ by $d_X(A,B)=\inf_{x\in A, y\in B}d_X(x,y)$.
\end{defn}

\begin{rem}\label{rem:diam when V(Gamma)setminus I(Gamma) is finite}
    If a connected graph $\Gamma$ satisfies $|V(X)\setminus \leaf(X)|<\infty$, then we have $\diam_X(X) < \infty$. Indeed, when $\diam_X(X) \ge 2$, for any $v \in \leaf(X)$ and $w \in \lk_X(v)$, we have $w \notin \leaf(X)$ by $\diam_X(X) \ge 2$. This implies $\diam_X(X) \le \diam_X(V(X)\setminus \leaf(X)) + 2 < \infty$.
\end{rem}

\subsection{Graph products of groups}

Readers are referred to \cite[Definition 3.5]{Gre} for details of graph product of groups.

\begin{defn}
    Let $G$ be a group. For $g,h \in G$, we define $[g,h] \in G$, by $[g,h] = ghg^{-1}h^{-1}$. For subsets $A, B\subset G$, we define $A B, [A,B] \subset G$ by $A B=\{gh \in G \mid g \in A, h \in B\}$ and $[A, B] = \{[g,h]\in G \mid g \in A, h \in B\}$. For a subset $A \subset G$, we denote by $\la A \ra$ the subgroup of $G$ generated by $A$ and also by $\la\!\la A \ra\!\ra$ the normal subgroup of $G$ generated by $A$, that is, $\la\!\la A \ra\!\ra = \la \bigcup_{g \in G}gAg^{-1} \ra$.
\end{defn}

\begin{defn}\label{def:graph product of groups}
    Let $\Gamma$ be a simplicial graph and $\G=\{G_v\}_{v\in V(\Gamma)}$ be a collection of groups. The \emph{graph product} $\Gamma\G$ is defined by
    \begin{align*}
        \Gamma\G
        =
        \ast_{v\in V(\Gamma)} \, G_v ~ / ~ \la\!\la\, \{ [g_v, g_w] \mid (v,w)\in E(\Gamma), g_v\in G_v, g_w \in G_w \}\, \ra\!\ra.
    \end{align*}
\end{defn}

\begin{rem}\label{rem:vertex subgroup}
    For any $v \in V(\Gamma)$, the group $G_v$ is a subgroup of $\Gamma\G$. We often identify $G_v$ as a subgroup of $\Gamma\G$. Also, for any $v,w \in V(\Gamma)$ with $v\neq w$, we have $G_v \cap G_w = \{1\}$.
\end{rem}

\begin{defn}\label{def:reduced form of graph product}
    Let $\Gamma$ be a simplicial graph and $\G=\{G_v\}_{v\in V(\Gamma)}$ be a collection of groups. Given $g \in \Gamma\G$, a geodesic word of $g$ in the generating set $\bigsqcup_{v\in V(\Gamma)}(G_v\setminus\{1\})$ is called a \emph{normal form of} $g$. We denote the word length of $g$ by $\|g\|$ (i.e. $\|g\|=|g|_{\bigsqcup_{v\in V(\Gamma)}(G_v\setminus\{1\})}$) and call $\|g\|$ the \emph{syllable length} of $g$. Given a normal form $g=g_1\cdots g_n$ of $g$,
    \begin{itemize}
        \item [-] 
        each letter $g_i \in \bigsqcup_{v\in V(\Gamma)}(G_v\setminus\{1\})$ is called a \emph{syllable},
        \item [-]
        and we refer to the process of obtaining the new normal form $g=g_1\cdots g_{i+1}g_i \cdots g_n$, where $1\le i <n$, $g_i \in G_{v_i}$, $g_{i+1} \in G_{v_{i+1}}$, and $(v_i,v_{i+1}) \in E(\Gamma)$, as \emph{syllable shuffling}.
    \end{itemize}
    For $h_1,\cdots,h_n \in \Gamma\G$, we say that the decomposition $h_1\cdots h_n$ is \emph{reduced} if we have $\|h_1\cdots h_n\| = \|h_1\| + \cdots + \|h_n\|$.
\end{defn}

\begin{conv}\label{conv:normal form}
    Throughout this paper, when we say that $g=g_1\cdots g_n$ is a normal form of $g \in \Gamma\G$, we assume that $g_1\cdots g_n$ is a geodesic word in $\bigsqcup_{v\in V(\Gamma)}(G_v\setminus\{1\})$ satisfying $g_i \in \bigsqcup_{v\in V(\Gamma)}(G_v\setminus\{1\})$ for each $i$, even if we don't mention it for brevity.
\end{conv}

\begin{rem}\label{rem:reduced decomposition}
    The decomposition $g=h_1\cdots h_n$ is reduced if and only if for any normal form $h_i=s_{i,1}\cdots s_{i,\|g_i\|}$ of each $g_i$, the word $g=(s_{1,1}\cdots s_{1,\|g_1\|}) \cdots (s_{n,1}\cdots s_{n,\|g_n\|})$ is a normal form of $g$.
\end{rem}

Theorem \ref{thm:normal form theorem} below follows by the same proof as \cite[Theorem 3.9]{Gre}, although the underlying graph $\Gamma$ is assumed to be finite in \cite[Theorem 3.9]{Gre}. That is, we don't need to assume that $\Gamma$ is finite.

\begin{thm}[Normal form theorem]\label{thm:normal form theorem}
    Let $\Gamma$ be a simplicial graph and $\G=\{G_v\}_{v\in V(\Gamma)}$ be a collection of groups. For any $g \in \Gamma\G$ with $g \neq 1$, $g=g_1\cdots g_n$ is a normal form of $g$ if and only if for any pair $(i,j)$ with $1 \le i < j \le n$ satisfying $v_i=v_j$, there exists $k$ with $i<k<j$ such that $(v_k,v_i) \notin E(\Gamma)$. Also, if $g=g_1\cdots g_n$ and $g=h_1\cdots h_m$ are normal forms of $g$, then $n=m$ and we obtain one from the other by finite steps of syllable shuffling.
\end{thm}

\begin{defn}\label{def:support of g}
    Let $\Gamma$ be a simplicial graph and $\G=\{G_v\}_{v\in V(\Gamma)}$ be a collection of groups. Let $g=g_1\cdots g_n$ be a normal form of $g \in \Gamma\G \setminus \{1\}$. For each syllable $g_i$, there exists a unique vertex $v_i \in V(\Gamma)$ with $g \in G_{v_i} \setminus \{1\}$. We define $\supp(g) \subset V(\Gamma)$ by $\supp(g)=\{v_i \mid 1 \le i\le n\}$ and call $\supp(g)$ the \emph{support of} $g$. We define the support of $1\in \Gamma\G$ by $\supp(1)=\emptyset$.
\end{defn}

\begin{rem}
    The support of $g \in \Gamma\G$ is well-defined by Remark \ref{rem:vertex subgroup} and Theorem \ref{thm:normal form theorem}.
\end{rem}

\begin{rem}\label{rem:notation of supp}
    When $g \in G_v\setminus\{1\}$ with $v \in V(\Gamma)$, we often consider the singleton $\supp(g) \,(=\{v\})$ as an element of $V(\Gamma)$ (and denote $\supp(g) \in \st_\Gamma(v)$ for example) by abuse of notation.
\end{rem}

\begin{rem}
    When $g=g_1\cdots g_n$ is a normal form of $g \in \Gamma\G \setminus \{1\}$, we have $\supp(g_i) \neq \supp(g_{i+1})$ for any $i$ since the word $g_1\cdots g_n$ is geodesic.
\end{rem}

Lemma \ref{lem:subsequence of reduced sequence} below easy follows from minimality of the length of a geodesic word.

\begin{lem}\label{lem:subsequence of reduced sequence}
    Let $\Gamma$ be a simplicial graph and $\G=\{G_v\}_{v\in V(\Gamma)}$ be a collection of groups. Let $g=g_1\cdots g_n$ be a normal form of $g \in \Gamma\G$, then any word obtained from $g_1\cdots g_n$ by finite steps of syllable shuffling is a normal form of $g$. Also, for any $i,j$ with $1\le i \le j \le n$, the subword $g_i\cdots g_j$ is a normal form.
\end{lem}


\subsection{The crossing graph and the contact graph of a quasi-median graph}\label{subsec:The crossing graph and the contact graph of a quasi-median graph}

In this section, we review notions on a quasi-median graph (see \cite{Gen17} for details). Readers are referred to \cite[Definition 2.1]{Gen17} and \cite[Definition 2.1]{Val21} for the definition of a quasi-median graph. These definitions are equivalent by \cite[Theorem 1]{BMW94}. A \emph{square} in a graph $X$ is an induced subgraph in $X$ isomorphic to a circuit of length $4$.

\begin{defn}
    Let $X$ be a quasi-median graph. A \emph{hyperplane} is an equivalence class of edges of $X$, where two edges $e$ and $f$ of $X$ are defined to be equivalent if there exists a sequence of edges $e = e_0, \cdots, e_n = f$ of $X$ such that, for every $1 \le i \le n-1$, $e_i$ and $e_{i+1}$ are either two sides of a triangle or opposite sides of a square. We denote by $\H(X)$ the set of all hyperplanes of $X$. For a hyperplane $J$ of $X$, we denote by $N'(J)$ the set of the endpoints of all the edges in $J$. A \emph{carrier} of $J \in \H(X)$, denoted by $N(J)$, is the induced subgraph on $N'(J)$. The \emph{crossing graph} of $X$, denoted by $\Delta X$, is a graph with $V(\Delta X) = \H(X)$, where two distinct vertices $J_1, J_2 \in \H(X)$ are adjacent if there exists a square in $X$ that contains both an edge in $J_1$ and an edge in $J_2$. The \emph{contact graph} of $X$, denoted by $\mathcal{C} X$, is a graph with $V(\mathcal{C} X) = \H(X)$, where two distinct vertices $J_1, J_2 \in \H(X)$ are adjacent if $N'(J_1) \cap N'(J_2) \neq \emptyset$.
\end{defn}

\begin{thm}\label{thm:contact graph is quasitree}{\rm \cite[Theorem A]{Val21}}
    For any connected quasi-median graph $X$, the contact graph $\mathcal{C}X$ of $X$ is quasi-isometric to a simplicial tree.
\end{thm}

Proposition \ref{prop:results on quasi-median graph} follows from \cite[Proposition 8.2, Corollary 8.10, Lemma 8.12]{Gen17}.

\begin{prop}\label{prop:results on quasi-median graph}
    Let $\Gamma$ be a simplicial graph and $\G=\{G_v\}_{v\in V(\Gamma)}$ be a collection of non-trivial groups. Let $X$ be the Cayley graph of $\Gamma\G$ with respect to the generating set $\bigsqcup_{v\in V(\Gamma)}(G_v\setminus\{1\})$. Then, the following holds.
    \begin{itemize}
        \item[(1)]
        $X$ is a quasi-median graph.
        \item[(2)]
        For each $v \in V(\Gamma)$, let $J_v \in \H(X)$ be the unique hyperplane with $G_v \subset N'(J_v)$. Then, $\H(X) = \{ gJ_v \mid g \in \Gamma\G, v \in V(\Gamma)\}$.
        \item[(3)]
        For any $v \in V(\Gamma)$, $N'(J_v) = \la G_w \mid w \in \st_\Gamma(v) \ra$.
        \item[(4)]
        If two distinct hyperplanes $gJ_v, hJ_w \in \H(X)$, where $g,h \in \Gamma\G$ and $v,w \in V(\Gamma)$, are adjacent in $\Delta X$, then $(v,w) \in E(\Gamma)$.
    \end{itemize}
\end{prop}

\subsection{Asymptotic dimension of metric spaces}\label{subsec:Asymptotic dimension of metric spaces}

Readers are referred to \cite{BD08} for details of asymptotic dimension.

\begin{defn}\cite[Theorem 19 (3)]{BD08}\label{def:asymptotic dimension}
    A metric space $(X,d_X)$ is said to have \emph{asymptotic dimension at most $n \in \NN\cup\{0\}$} (and denoted by $\asdim (X) \le n$) if for any $r\in\RR_{>0}$, there exists a set $\mathcal{U}$ in $2^X$ such that
    \begin{align*}
        X = \bigcup_{U 
        \,\in\, \mathcal{U}} U,~~~~ \sup_{U \,\in\, \mathcal{U}} \diam_X(U) < \infty, {\rm~~~~and~~~~} \sup_{x \,\in\, X} |\{U \in \mathcal{U} \mid \N_X(x,r)\cap U \neq \emptyset \}| \le n+1.
    \end{align*}
\end{defn}

Definition \ref{def:uniformly asdim} below comes from \cite[p.10]{BD08}, although we consider an arbitrary family of metric spaces, not necessarily subsets of a common metric space.

\begin{defn}\label{def:uniformly asdim}
    Let $n \in \NN\cup\{0\}$. A family of metric spaces $(X_\alpha, d_\alpha)_{\alpha \in \A}$ is said to \emph{satisfy $\asdim \le n$ uniformly}, if for any $r \in \RR_{>0}$, there exist $\mathcal{U}^0_\alpha, \cdots, \mathcal{U}^n_\alpha \subset 2^{X_\alpha}$ for each $\alpha \in \A$ satisfying the following three conditions.
    \begin{itemize}
        \item[(1)]
        $X_\alpha = \bigcup_{U \,\in\, \mathcal{U}^0_\alpha \cup \cdots \cup \mathcal{U}^n_\alpha}U$ for any $\alpha\in\A$.
        \item[(2)]
        $\inf\{d_{X_\alpha}(U,V) \mid U,V \in \mathcal{U}^i_\alpha, U\neq V \}>r$ for any $i\in \{0,\cdots,n\}$ and $\alpha \in \A$.
        \item[(3)] 
        $\sup\{\diam_{X_\alpha}(U) \mid \alpha\in\A, i\in \{0,\cdots,n\}, U \in \mathcal{U}^i_\alpha\} < \infty$.
    \end{itemize}
\end{defn}

Theorem \ref{thm:union theorem} is a variant of the usual Union Theorem for a family of metric spaces, but can be proved in the same way as \cite[Theorem 25]{BD08} (see \cite[Section 4]{BD08}).

\begin{thm}[Union Theorem]\label{thm:union theorem}
    Let $(X_j, d_j)_{j \in J}$ be a family of metric spaces. Suppose that for each $i\in J$, there exists $\A_j \subset 2^{X_j}$ satisfying the following two conditions.
    \begin{itemize}
        \item[(1)]
        $X_j = \bigcup_{U \in \A_j} U$ for any $j\in J$ and the family $\bigcup_{j \in J}\A_j$ satisfies $\asdim \le n$ uniformly.
        \item[(2)]
        For any $r \in \RR_{>0}$, there exists $Y_{r,j} \subset X_j$ for each $j \in J$ such that the family $(Y_{r,j})_{j \in J}$ satisfies $\asdim \le n$ uniformly and $\inf\{d_j(U \setminus Y_{r,j}, U' \setminus Y_{r,j}) \mid U, U' \in \A_j, U\neq U' \} > r$ for any $j \in J$.
    \end{itemize}
    Then, the family $(X_j)_{j \in J}$ satisfies $\asdim \le n$ uniformly.
\end{thm}

Theorem \ref{thm:Hurewicz Theorem} below is \cite[Theorem 1]{BD06}.

\begin{thm}[Bell-Dranishnikov's Hurewicz Theorem]\label{thm:Hurewicz Theorem}
Let $X$ be a geodesic metric space and $Y$ be a metric space. Let $f \colon X \to Y$ be a Lipschitz map. Suppose that for every $R \in \RR_{>0}$, the family $(f^{-1}(\N_Y(y,R)))_{y \,\in Y}$ satisfies $\asdim \le n$ uniformly. Then $\asdim (X) \le \asdim (Y) + n$.
\end{thm}

\subsection{Hyperbolic graphs, tightness, and fineness}
\label{subsec:Hyperbolic graphs, fineness, and tightness}

In this section, we review hyperbolic spaces, tightness in the sense of Bowditch, and fineness. Readers are referred to \cite{BH99} for details of hyperbolic spaces.

\begin{defn}\label{def:gromov product}
    Let $(X,d_X)$ be a metric space. For $x,y,z\in X$, we define $(x,y)_z$ by
\begin{align}\label{eq:gromov product}
    (x,y)_z=\frac{1}{2}\left( d_X(x,z)+d_X(y,z)-d_X(x,y) \right).    
\end{align}
\end{defn}

\begin{prop} \label{prop:hyp sp}
    For any geodesic metric space $(X,d_X)$, the following conditions are equivalent.
    \item[(1)]
    There exists $\delta\in\NN$ satisfying the following property. Let $x,y,z\in X$, and let $p$ be a geodesic path from $z$ to $x$ and $q$ be a geodesic path from $z$ to $y$. If two points $a\in p$ and $b\in q$ satisfy $d_X(z,a)=d_X(z,b)\le (x,y)_z$, then we have $d_X(a,b) \le \delta$.
    \item[(2)] 
    There exists $\delta\in\NN$ such that for any $w,x,y,z \in X$, we have
    \[
    (x,z)_w \ge \min\{(x,y)_w, (y,z)_w\} - \delta.
    \]
\end{prop}

\begin{defn}\label{def:hyperbolic space}
    A geodesic metric space $X$ is called \emph{hyperbolic}, if $X$ satisfies the equivalent conditions (1) and (2) in Proposition \ref{prop:hyp sp}. We call a hyperbolic space $\delta$-\emph{hyperbolic} with $\delta \in \NN$, if $\delta$ satisfies both of (1) and (2) in Proposition \ref{prop:hyp sp}. A connected graph $X$ is called \emph{hyperbolic}, if the geodesic metric space $(X,d_X)$ is hyperbolic.
\end{defn}

For a $\delta$-hyperbolic graph $X$ with $\delta>0$, there exists some constant $\delta_0>0$,
depending only on $\delta$ such that if $c \in V(X)$ lies in some geodesic
from $a \in V(X)$ to $b \in V(X)$, then every geodesic from $a$ to $b$ passes through $\N_X(c,\delta_0)$.

\begin{defn}
    Let $X$ be a $\delta$-hyperbolic graph with $\delta>0$. We say that $X$ is \emph{tight in the sense of Bowditch} if for each $(a,b) \in V(X)\times V(X)$, there exists $\T(a,b) \subset \geo_X(a,b)$ satisfying the conditions (1) and (2) below, where $V_\T(a,b)=\bigcup \{ V(p) \mid p \in \T(a,b) \}$ and $V_\T(a,b\, ;r)=\bigcup\{ V_{\T}(a',b') \mid a'\in \N_X(a,r),\, b'\in \N_X(b,r) \}$.
    \begin{itemize}
        \item [(1)]
        $\exists\, P_0 \in\NN, \forall\, a,b \in V(X), \forall\, c \in V_\T(a,b),\, |V_{\T}(a,b) \cap \N_X(c,\delta_0)| \le P_0$.
        \item [(2)] 
        $\exists\, P_1,k_1 \in\NN, \forall\, r \in\NN, \forall\, a,b \in V(X), \text{for all $c \in V_\T(a,b)$ with $d_X(c,\{a,b\}) \ge r+k_1$}, 
        \newline
        |V_{\T}(a,b\, ;r) \cap \N_X(c,\delta_0)| \le P_1$.
    \end{itemize}
    When a group $G$ acts on $X$, we say that the family $\{\T(a,b)\}_{(a,b)\in V(X)\times V(X)}$ is $G$-\emph{equivariant} if $g\T(a,b) = \T(ga,gb)$ for any $g \in G$ and $(a, b) \in V(X)\times V(X)$.
\end{defn}

See \cite[Sectionn 8]{Bow12} for details of fine graphs and their connection with relatively hyperbolic groups.

\begin{defn}\label{def:fine graph}
    Let $X$ be a simplicial graph. The graph $X$ is called \emph{fine}, if $|\C_X(e,n)|<\infty$ for any $e \in E(X)$ and $n\in\NN$ (see Definition \ref{def:concepts in graph theory}). The graph $X$ is called \emph{uniformly fine}, if there exists a function $f \colon \NN \to \NN$ such that $|\C_X(e,n)| \le f(n)$ for any $e \in E(X)$ and $n\in\NN$.
\end{defn}

\subsection{Relatively hyperbolic groups}

See \cite[Appendix]{Osi06} for the definition of coned-off Cayley graphs.

\begin{defn}
    A finitely generated group $G$ is called \emph{hyperbolic relative to} a finite collection $\{H_i\}_{i=1}^n$ of subgroups of $G$, if for some (equivalently, any) finite generating set $X$ of $G$, the coned-off Cayley graph $\widehat{\Gamma}(G,X)$ of $G$ with respect to $\{H_i\}_{i=1}^n$ is hyperbolic and fine.
\end{defn}

\begin{defn}
    Let a group $G$ act on a set $X$. We denote by $X/G$ the quotient set of the orbit equivalence relation induced by the action $G \act X$. For $x \in X$, we define $\stab_G(x)\subset G$ by $\stab_G(x)=\{g \in G \mid gx=x\}$.
\end{defn}

The following equivalent condition of relative hyperbolicity follows from \cite[Theorem 7.10]{Bow12} and \cite[Theorem 6.1]{Dah03}. See also \cite{Dah03b} for the complete version of the proof and \cite{Osi06} for more equivalent conditions.

\begin{thm}\label{thm:relatively hyperbolic groups}
    Suppose that $G$ is a finitely generated group and $\{H_i\}_{i=1}^n$ is a finite collection of finitely generated subgroups of $G$. Then, $G$ is hyperbolic relative $\{H_i\}_{i=1}^n$ if and only if there exist a fine hyperbolic graph $X$ and an action $G \act X$ such that the following two conditions hold.
    \begin{itemize}
        \item[(1)]
        $E(\Gamma) / G$ is finite and for any $e \in E(\Gamma)$, $\stab_G(e) \,(=\stab_G(e_-) \cap \stab_G(e_+))$ is finite.
        \item[(2)]
        For any $x\in V(X)$, the group $\stab_{G}(x)$ is either finite or conjugate to $H_i$ for some $i\in \{1,\cdots,n\}$.
    \end{itemize}
\end{thm}

\section{Geometry of the extension graph}
\label{sec:Geometry of the extension graph}

In this section, we introduce the extension graph of graph product of groups and study its properties. In Section \ref{subsec:Definition and basic properties of the extension graph of graph products of groups}, we define the extension graph and study orbits and stabilizers of the action of graph product on the extension graph. Important results in Section \ref{subsec:Definition and basic properties of the extension graph of graph products of groups} are Corollary \ref{cor:Stab(v)}, Corollary \ref{cor:vertex orbits are disjoint}, and Lemma \ref{lem:edges of Gammae come from Gamma}, which are used implicitly throughout this paper. In Section \ref{subsec:Connection to the crossing and contact graphs}, we show that the extension graph is isomorphic to the crossing graph. In Section \ref{subsec:Admissible paths}, we introduce the notion of an admissible path and study its properties. This notion contains geodesic paths in the extension graph and is an analogue of a path without backtracking in a tree under the intuition that the extension graph is obtained by assembling copies of a defining graph in a `tree-like' way. The key result in Section \ref{subsec:Admissible paths} is Proposition \ref{prop:Gamma is embedded into Gammae}. In Section \ref{subsec:Classification of geodesic bigons and triangles}, we study geodesic bigons and triangles in the extension graph by using admissible paths. It turns out that we can prove similar classification results to Strebel's classification of geodesic bigons and triangles for small cancellation groups, which are Proposition \ref{prop:paths must penetrate long} and Proposition \ref{prop: classification of geodesic triangle}.

\subsection{Definition and basic properties of the extension graph of graph product}
\label{subsec:Definition and basic properties of the extension graph of graph products of groups}

In this section, we record basic properties of the extension graph. Some of the lemmas here are already available or can be easily deduced from the literature (e.g. \cite{AM15}).

\begin{defn}\label{def:extension graph}
Let $\Gamma$ be a simplicial graph and $\G=\{G_v\}_{v\in V(\Gamma)}$ be a collection of non-trivial groups. The \emph{extension graph} $\Gamma^e$ is defined as follows.
\begin{align*}
    V(\Gamma^e)
    &=
    \{gG_vg^{-1} \in 2^{\Gamma\G} \mid v\in V(\Gamma), g\in \Gamma\G \}, \\
    E(\Gamma^e)
    &=
    \{ (gG_vg^{-1},hG_wh^{-1}) \in V(\Gamma^e)^2 \mid \text{$gG_vg^{-1}\neq hG_wh^{-1}$ and $[gG_vg^{-1},hG_wh^{-1}]=\{1\}$} \}.
\end{align*}
\end{defn}

\begin{rem}
The group $\Gamma\G$ acts on $\Gamma^e$ by $(g,x)\mapsto gxg^{-1}$ for each $x\in V(\Gamma^e)$ and $g\in \Gamma\G$. Although the action of $\Gamma\G$ on $\Gamma^e$ is a right action in \cite{KK13}, my personal preference is a left action. For brevity, given $g \in \Gamma\G$ and $x \in V(\Gamma^e)$, we denote $gxg^{-1}$ by $g.x$ (i.e. $g.x=gxg^{-1}$). For each $x\in V(\Gamma^e)$, we define $\stab_{\Gamma\G}(x)$ by $\stab_{\Gamma\G}(x)=\{g\in \Gamma\G \mid g.x=x\}$.
\end{rem}

From Lemma \ref{lem:support of product} up to Lemma \ref{lem:leaf in Gamma is leaf in extension graph}, let $\Gamma$ be a simplicial graph and $\G=\{G_v\}_{v\in V(\Gamma)}$ be a collection of non-trivial groups.

We first study how the support of elements of graph product behaves under the conjugate action, which is from Lemma \ref{lem:support of product} up to Corollary \ref{cor:support for conjugation}. This is used to clarify orbits and stabilizers of the action of graph product on the extension graph.

\begin{lem}\label{lem:support of product}
    For any $g,h\in \Gamma\G$, we have $\supp(gh) \subset \supp(g)\cup\supp(h)$.
\end{lem}

\begin{proof}
    We will show the statement by induction on $\|g\|+\|h\|$. When $\|g\|+\|h\|=0$, we have $g=h=1$, hence $\supp(gh) \subset \supp(g)\cup\supp(h)$. Next, given $N\in\NN$, we assume that the statement is true for any $g,h \in\Gamma\G$ satisfying $\|g\|+\|h\|<N$. Let $g,h \in \Gamma\G$ satisfy $\|g\|+\|h\|=N$. Let $g=g_1\cdots g_n$ and $h=h_1\cdots h_m$ be normal forms of $g$ and $h$ respectively. Note $n+m=N$. If $\|gh\| = \|g\| + \|h\|$, then $gh=g_1\cdots g_n h_1 \cdots h_m$ is a normal form of $gh$. Hence, we have $\supp(gh)=\supp(g)\cup\supp(h)$. If $\|gh\| < \|g\| + \|h\|$, then by Theorem \ref{thm:normal form theorem} there exists $g_i$ and $h_j$ such that $\supp(g_i) = \supp(h_j)$ and $\{\supp(g_{i'}) \mid i<i'\}\cup\{\supp(h_{j'}) \mid j'<j\} \subset \lk_\Gamma(\supp(g_i))$. Define $g',h' \in \Gamma\G$ by $g'=g_1\cdots g_{i-1}(g_ih_j)g_{i+1} \cdots g_n$ and $h' = h_1 \cdots h_{j-1}h_{j+1} \cdots h_m$, then we have $gh=g'h'$. Since we have $\|g'\|+\|h'\| \le n + (m-1) = N-1$, by applying our assumption on induction to $g'$ and $h'$, we get $\supp(gh)=\supp(g'h')\subset \supp(g')\cup\supp(h') \subset \supp(g)\cup\supp(h)$.
\end{proof}

\begin{lem}\label{lem:conjugation of vertex group}
    Let $v\in V(\Gamma)$, $a\in G_v\setminus\{1\}$, and $g\in \Gamma\G$, then there exist $h_1,h_2,h_3 \in \Gamma\G$ satisfying the four conditions (i)-(iv): (i) $g=h_1h_2h_3$, (ii) $\supp(h_1)\cup\{v\}=\supp(gag^{-1})$, (iii) $\supp(h_2) \subset \lk_\Gamma(v)$, (iv) $\supp(h_3) \subset \{v\}$.
\end{lem}

\begin{proof}
    We will show the statement by induction on $\|g\|$. When $\|g\|=0$, we have $g=1$. Hence, $h_1,h_2,h_3$ defined by $h_1=h_2=h_3=1$ satisfy the statement. Next, given $n\in\NN$, we assume that the statement is true for any $v,a,g$ satisfying $\|g\|<n$ in addition. Let $g\in \Gamma\G$ satisfy $\|g\|=n$. Let $g=g_1\cdots g_n$ be a normal form of $g$. Note that $g^{-1}=g_n^{-1}\cdots g_1^{-1}$ is a normal of $g^{-1}$. If the decomposition $gag^{-1}$ is reduced, then $gag^{-1}=(g_1\cdots g_n) a (g_n^{-1}\cdots g_1^{-1})$ is a normal form of $gag^{-1}$. Hence, $h_1,h_2,h_3$ defined by $h_1=g$ and $h_2=h_3=1$ satisfy the statement. If the decomposition $gag^{-1}$ is not reduced, then by Theorem \ref{thm:normal form theorem}, there exists $i$ with $1\le i\le n$ that satisfies one of (1) or (2): (1) $\{v\}\cup\{\supp(g_j) \mid i<j\} \subset \lk_\Gamma(\supp(g_i))$, (2) $\supp(g_i)=v$ and $\{\supp(g_j) \mid i<j\} \subset \lk_\Gamma(v)$.
    
    In case (1), we have $gg_i^{-1}=g_1\cdots g_{i-1}g_{i+1}\cdots g_n$ and $gag^{-1} = gg_i^{-1}ag_ig^{-1}$. By $\|gg_i^{-1}\| = \|g_1\cdots g_{i-1}g_{i+1}\cdots g_n\| \le n-1$, we can apply our assumption of induction to $a$ and $gg_i^{-1}$ and see that there exists $h'_1,h'_2,h'_3 \in \Gamma\G$ such that $gg_i^{-1}=h'_1h'_2h'_3$, $\supp(h'_1)\cup\{v\}=\supp(gg_i^{-1}ag_ig^{-1})=\supp(gag^{-1})$, $\supp(h'_2) \subset \lk_\Gamma(v)$, and $\supp(h'_3) \subset \{v\}$. Since we have $g=h'_1h'_2g_ih'_3$ by $(\supp(g_i),v) \in E(\Gamma)$, we can check that $h_1,h_2,h_3$ defined by $h_1=h'_1$, $h_2=h'_2g_i$, $h_3=h'_3$ satisfy the statement.
    
    In case (2), we have $gg_i^{-1}=g_1\cdots g_{i-1}g_{i+1}\cdots g_n$ and $g_iag_i^{-1} \in G_v\setminus \{1\}$ by $a\neq 1$. By $\|gg_i^{-1}\| \le n-1$, we can apply our assumption of induction to $g_iag_i^{-1}$ and $gg_i^{-1}$ and see that there exists $h'_1,h'_2,h'_3 \in \Gamma\G$ such that $gg_i^{-1}=h'_1h'_2h'_3$, $\supp(h'_1)\cup\{v\}=\supp(gg_i^{-1} (g_iag_i^{-1}) (gg_i^{-1})^{-1})=\supp(gag^{-1})$, $\supp(h'_2) \subset \lk_\Gamma(v)$, and $\supp(h'_3) \subset \{v\}$. Since we have $g=h'_1h'_2h'_3g_i$ and $h'_3g_i \in G_v$ by $\supp(g_i)=v$, we can check that $h_1,h_2,h_3$ defined by $h_1=h'_1$, $h_2=h'_2$, $h_3=h'_3g_i$ satisfy the statement.
\end{proof}

\begin{cor}\label{cor:support for conjugation}
    Suppose that $v,w\in V(\Gamma)$, $a\in G_v\setminus\{1\}$, $b\in G_w\setminus\{1\}$, and $g,h\in \Gamma\G$ satisfy $gag^{-1}=hbh^{-1} \in \Gamma\G$, then $v=w$ and $\supp(h^{-1}g) \subset \st_\Gamma(v)$.
\end{cor}

\begin{proof}
    Without loss of generality, we can assume $h=1$. By $gag^{-1}=b$ and Lemma \ref{lem:conjugation of vertex group}, we have $v \in \supp(gag^{-1})=\supp(b)=\{w\}$. This implies $v=w$. By Lemma \ref{lem:conjugation of vertex group}, there exist $h_1,h_2,h_3 \in \Gamma\G$ such that $g=h_1h_2h_3$, $\supp(h_1)\cup\{v\}=\supp(gag^{-1})$, $\supp(h_2) \subset \lk_\Gamma(v)$, and $\supp(h_3) \subset \{v\}$. By this and $v=w$, we have $\supp(h_1)\subset\supp(gag^{-1})=\{w\}=\{v\}$. Hence, we have $\supp(g) \subset \supp(h_1)\cup\supp(h_2)\cup\supp(h_3) \subset \st_\Gamma(v)$ by Lemma \ref{lem:support of product}.
\end{proof}

Corollary \ref{cor:support for conjugation} has two important consequences, Corollary \ref{cor:Stab(v)} and Corollary \ref{cor:vertex orbits are disjoint}.

\begin{cor}\label{cor:Stab(v)}
    Let $v\in V(\Gamma)$ and $g\in\stab_{\Gamma\G}(G_v)$, then we have $\supp(g) \subset \st_\Gamma(v)$. In particular, for any $v\in V(\Gamma)$, we have $\stab_{\Gamma\G}(G_v)=\la G_w \mid w\in \st_\Gamma(v) \ra$.
\end{cor}

\begin{proof}
    Since $G_v$ is non-trivial, take $a \in G_v\setminus\{1\}$. By $g \in \stab_{\Gamma\G}(G_v)$, we have $gag^{-1} \in G_v\setminus\{1\}$. This implies $\supp(g) \subset \st_\Gamma(v)$ by Corollary \ref{cor:support for conjugation}.
\end{proof}

Corollary \ref{cor:vertex orbits are disjoint} below means that every vertex in $\Gamma^e$ has a unique vertex in $\Gamma$ associated to it, which is formulated in Definition \ref{def:v(x)}.

\begin{cor}\label{cor:vertex orbits are disjoint}
    For any $v,w\in V(\Gamma)$ with $v\neq w$, we have $\Gamma\G.G_v \cap \Gamma\G.G_w = \emptyset \subset V(\Gamma^e)$.
\end{cor}

\begin{proof}
    Let $v,w\in V(\Gamma)$ satisfy $\Gamma\G.G_v \cap \Gamma\G.G_w \neq \emptyset$. There exists $g\in \Gamma\G$ such that $gG_vg^{-1}=G_w$. Since $G_v$ and $G_w$ are non-trivial, we have $v=w$ by Corollary \ref{cor:support for conjugation}.
\end{proof}

Lemma \ref{lem:edges of Gammae come from Gamma} below means that every edge in $\Gamma^e$ comes by translating an edge in $\Gamma$ by the action of the graph product $\Gamma\G$. It also means that we can consider $\Gamma$ as an induced subgraph of $\Gamma^e$, which we elaborate in Convention \ref{conv:Gamma is embedded to Gammae}.

\begin{lem}\label{lem:edges of Gammae come from Gamma}
    Let $v,w\in V(\Gamma)$ and $g,h\in \Gamma\G$. If $(g.G_v, h.G_w)\in E(\Gamma)$, then we have $(v,w)\in E(\Gamma)$ and there exists $k \in \Gamma\G$ such that $(k.G_v, k.G_w) = (g.G_v, h.G_w)$.
\end{lem}

\begin{proof}
     We can assume $h=1$ without loss of generality, because we can consider $(h^{-1}g.G_v, G_w) \in E(\Gamma^e)$ in general case. Since $G_v$ and $G_w$ are non-trivial and we have $(g.G_v, G_w)\in E(\Gamma)$, there exist $a \in G_v\setminus\{1\}$ and $b\in G_w\setminus\{1\}$ such that $[gag^{-1},b]=1$. By $(gag^{-1})b(gag^{-1})^{-1} = b$ and Corollary \ref{cor:support for conjugation}, we have $\supp(gag^{-1}) \subset \st_\Gamma(w)$. This and Lemma \ref{lem:conjugation of vertex group} imply $v\in \supp(gag^{-1}) \subset \st_\Gamma(w)$. Suppose $v = w$ for contradiction, then we have $\supp(gag^{-1}) \subset \st_\Gamma(w) = \st_\Gamma(v)$. This implies $\supp(g) \subset \st_\Gamma(v)$ by Lemma \ref{lem:support of product} and Lemma \ref{lem:conjugation of vertex group}. Hence, $g.G_v = G_v = G_w$. This contradicts $g.G_v \neq G_w$ by $(g.G_v, G_w) \in E(\Gamma^e)$. Hence, $v\neq w$. By this and $v\in \st_\Gamma(w)$, we have $(v,w) \in E(\Gamma)$. By Lemma \ref{lem:conjugation of vertex group}, there exist $h_1,h_2,h_3 \in \Gamma\G$ such that $g=h_1h_2h_3$, $\supp(h_1)\cup\{v\}=\supp(gag^{-1})$, $\supp(h_2) \subset \lk_\Gamma(v)$, and $\supp(h_3) \subset \{v\}$. By $\supp(h_2h_3) \subset \supp(h_2) \cup \supp(h_3) \subset \st_\Gamma(v)$, we have $g.G_v=h_1.G_v$. On the other hand, by $\supp(h_1) \subset \supp(gag^{-1}) \subset \st_\Gamma(w)$, we have $G_w=h_1.G_w$.
\end{proof}

\begin{conv}\label{conv:Gamma is embedded to Gammae}
Define the map $\iota \colon \Gamma \to \Gamma^e$ by $\iota(v)=G_v$ for each $v\in V(\Gamma)$, then by Corollary \ref{lem:edges of Gammae come from Gamma}, $\iota$ is a graph isomorphism from $\Gamma$ to the induced subgraph on $\iota(V(\Gamma))$ in $\Gamma^e$. In what follows, we consider $\Gamma$ as a subgraph of $\Gamma^e$ by this embedding and denote $G_v$ by $v$ for each $v\in V(\Gamma)$. Note also that the graph $g.\Gamma$ is isomorphic to $\Gamma$ for any $g\in \Gamma\G$ since $\Gamma\G$ acts on $\Gamma^e$ as graph automorphism.    
\end{conv}

\begin{defn}\label{def:v(x)}
    For $x \in V(\Gamma^e)$, we define $v(x) \in V(\Gamma)$ to be the unique vertex in $V(\Gamma)$ such that $x \in \Gamma\G.v(x)$.
\end{defn}

\begin{rem}
    Uniqueness of $v(x)$ in Definition \ref{def:v(x)} follows from Corollary \ref{cor:vertex orbits are disjoint}.
\end{rem}

\begin{rem}\label{rem:intersection of stabilizers}
    Let $a,b\in V(\Gamma)$. If $\girth(\Gamma)>4$, then the following hold from Corollary \ref{cor:Stab(v)}.
    \begin{itemize}
        \item[(1)]
        If $d_\Gamma(a,b)=1$, then $\stab_{\Gamma\G}(a)\cap\stab_{\Gamma\G}(b) = G_a \times G_b$.
        \item[(2)]
        If $d_\Gamma(a,b)=2$, then there exists a unique vertex $c\in V(\Gamma)$ satisfying $\st_\Gamma(a)\cap\st_\Gamma(b)=\{c\}$ and we have $\stab_{\Gamma\G}(a)\cap\stab_{\Gamma\G}(b) = G_c$.
        \item[(3)]
        If $d_\Gamma(a,b) \ge 3$, then $\stab_{\Gamma\G}(a)\cap\stab_{\Gamma\G}(b)=\{1\}$. In particular, if $g,h \in \Gamma\G$ satisfy $\diam_{\Gamma^e}(g.\Gamma\cap h.\Gamma) \ge 3$, then $g=h$. Indeed, take $x,y \in V(g.\Gamma)\cap V(h.\Gamma)$ with $d_{\Gamma^e}(x,y) \ge 3$, then we have $x=g.v(x)=h.v(x)$ and $y=g.v(y)=h.v(y)$ by Corollary \ref{cor:vertex orbits are disjoint}, hence $h^{-1}g \in \stab_{\Gamma\G}(v(x))\cap\stab_{\Gamma\G}(v(y))=\{1\}$ by $d_{\Gamma}(v(x), v(y)) \ge d_{\Gamma^e}(x,y) \ge 3$.
    \end{itemize}
\end{rem}

Lemma \ref{lem:Gammae is connected if Gamma is connected} below means that we can handle both $\Gamma$ and $\Gamma^e$ as geometric objects when $\Gamma$ is connected.

\begin{lem}\label{lem:Gammae is connected if Gamma is connected}
    If $\Gamma$ is connected, then $\Gamma^e$ is connected.
\end{lem}

\begin{proof}
    Let $a,b \in V(\Gamma)$ and $g \in \Gamma\G$. Let $g = s_1 \cdots s_n$ be a normal form of $g$. For each $i \in \{1, \cdots, n\}$, define $x_i \in V(\Gamma^e)$ by $x_i = s_1 \cdots s_i.\supp(s_i)$ for brevity. Note $\forall\, i \in \{1, \cdots, n\},\, x_i = s_1 \cdots s_{i-1}.\supp(s_i)$ by $s_i \in \stab_{\Gamma\G} (\supp(s_i))$. This implies $\{a, x_1\} \subset \Gamma$, $\{x_n, g.b\} \subset g.\Gamma$, and $\{x_i, x_{i+1}\} \subset s_1 \cdots s_i. \Gamma$ for any $i \in \{1,\cdots,n-1\}$. For every $h \in \Gamma\G$, any two vertices in $h.\Gamma$ are connected by a path in $h.\Gamma$ since $\Gamma$ is connected. Hence, $a$ and $g.b$ are connected by a path obtained by concatenating paths from $a$ to $x_1$, from $x_i$ to $x_{i+1}$ for each $i$, and from $x_n$ to $b$. This implies that $\Gamma^e$ is connected.
\end{proof}

Next, we study the link of a vertex in the extension graph. Lemma \ref{lem:leaf in Gamma is leaf in extension graph} (1) means that leaves in $\Gamma$ remain being leaves even in $\Gamma^e$.

\begin{lem}\label{lem:leaf in Gamma is leaf in extension graph}
The following hold.
\begin{itemize}
    \item[(1)]
    For any $v,w \in V(\Gamma)$ satisfying $\lk_{\Gamma}(v) = \{w\}$, we have $\lk_{\Gamma^e}(v) = \{w\}$.
    \item[(2)] 
    If distinct vertices $x,y,z \in V(\Gamma)$ satisfy $\{(x,y), (x,z)\} \subset E(\Gamma)$ and $(y,z) \notin E(\Gamma)$, then we have $|\lk_{\Gamma^e}(x)|=\infty$.
\end{itemize}
\end{lem}

\begin{proof}
    (1) Let $x \in V(\Gamma)$ and $g \in \Gamma\G$ satisfy $(v,g.x) \in E(\Gamma^e)$. By Lemma \ref{lem:edges of Gammae come from Gamma} and $\lk_{\Gamma}(v) = \{w\}$, we have $x=w$ and there exists $k \in \Gamma\G$ such that $(k.v, k.w)=(v, g.x)$. This implies $k \in \stab_{\Gamma\G}(v)$. Hence, we have $g.x=k.w=w$ since we have $k \in \stab_{\Gamma\G}(v)=G_v\times G_w$ by Lemma \ref{cor:Stab(v)}. Thus, $\lk_{\Gamma^e}(v) = \{w\}$.

    (2) By $(y,z) \notin E(\Gamma)$, $G_y$ has infinite index in $\la G_y,G_z \ra \,(\cong G_y\ast G_z)$. By Corollary \ref{cor:Stab(v)}, $\la G_y,G_z \ra \cap \stab_{\Gamma\G}(y) = G_y$. Hence, we have $|\lk_{\Gamma^e}(x)|=\infty$ by $|\la 
G_y,G_z \ra .y|=|\la 
G_y,G_z \ra / G_y|=\infty$ and $\la 
G_y,G_z \ra .y \subset \lk_{\Gamma^e}(x)$.
\end{proof}

Finally, we study the intersection of stabilizers of two distinct vertices in the extension graph in Corollary \ref{cor:stabilizer of two far vertices is trivial}. Lemma \ref{lem:conjugation of stabilizer} is an auxiliary lemma for this.

\begin{lem}\label{lem:conjugation of stabilizer}
    Let $\Gamma$ be a simplicial graph with $\girth(\Gamma) > 4$ and $\G=\{G_v\}_{v\in V(\Gamma)}$ be a collection of non-trivial groups. Let $g,h\in \Gamma\G$ and $v\in V(\Gamma)$. If $\emptyset \neq \supp(h) \subset \st_\Gamma(v)$, then we have $\supp(h)\cap\supp(ghg^{-1})\neq \emptyset$ and there exist $k_1,k_2,k_3 \in\Gamma\G$ and $w\in \st_\Gamma(v)$ such that $g=k_1k_2k_3$, $\supp(k_1) \subset \supp(ghg^{-1})$, $\supp(k_2) \subset \st_\Gamma(w)$, and $\supp(k_3) \subset \st_\Gamma(v)$.
\end{lem}

\begin{proof}
    We will show the statement by induction on $\|g\|$. When $\|g\|=0$, we have $g=1$. Hence, $\supp(h)\cap\supp(ghg^{-1})\neq \emptyset$ and $k_1,k_2,k_3 \in \Gamma\G$ defined by $k_1=k_2=k_3=1$ satisfy the statement. Next, given $n\in\NN$, we assume that the statement is true for any $g,h,v$ satisfying $\|g\|<n$ in addition. Let $g\in \Gamma\G$ satisfy $\|g\|=n$. If $|\supp(h)|=1$, then there exists $w \in \st_\Gamma(v)$ such that $\supp(h) =\{w\}$. By Lemma \ref{lem:conjugation of vertex group}, there exist $k'_1,k'_2,k'_3 \in \Gamma\G$ such that $g=k'_1k'_2k'_3$, $\supp(k'_1)\cup\{w\}=\supp(ghg^{-1})$, $\supp(k'_2) \subset \lk_\Gamma(w)$, and $\supp(k'_3) \subset \{w\}$. Hence, $k_1,k_2,k_3 \in \Gamma\G$ defined by $k_1=k'_1$, $k_2=k'_2k'_3$, and $k_3=1$ satisfy the conclusion and we also have $w \in \supp(h)\cap\supp(ghg^{-1})$. 
    
    Thus, we assume $|\supp(h)|\ge 2$ in what follows. Let $g=g_1\cdots g_n$ and $h=h_1\cdots h_m$ be normal forms of $g$ and $h$ respectively. If the decomposition $ghg^{-1}$ is reduced, then by Theorem \ref{thm:normal form theorem}, $ghg^{-1} = (g_1\cdots g_n)(h_1\cdots h_m)(g_n^{-1}\cdots g_1^{-1})$ is a normal form of $ghg^{-1}$. Hence, $k_1,k_2,k_3 \in \Gamma\G$ defined by $k_1=g$ and $k_2=k_3=1$ satisfy the conclusion and we also have $\emptyset \neq \supp(h)\subset\supp(ghg^{-1})$. If the decomposition $ghg^{-1}$ is not reduced, then one of (1)-(3) holds: (1) there exists $i$ such that $\{\supp(g_{i'}) \mid i<i' \}\cup\supp(h) \subset \lk_\Gamma(\supp(g_i))$, (2) there exist $i$ and $j$ such that $\supp(g_i)=\supp(h_j)$ and $\{\supp(g_{i'}) \mid i<i' \}\cup\{\supp(h_{j'}) \mid j'<j \} \subset \lk_\Gamma(\supp(g_i))$ (3) there exist $i$ and $j$ such that $\supp(g_i)=\supp(h_j)$ and $\{\supp(h_{j'}) \mid j<j' \}\cup\{\supp(g_{i'}) \mid i<i' \} \subset \lk_\Gamma(\supp(g_i))$.

    In case (1), since we have $\supp(h)\subset \st_\Gamma(v)$ and $|\supp(h)| \ge 2$ and $\Gamma$ has no triangle nor square, we have $\supp(g_i)=v$ by $\supp(h)\subset \lk_\Gamma(\supp(g_i))$. We have $gg_i^{-1} = g_1\cdots g_{i-1}g_{i+1} \cdots g_n$ and $ghg^{-1}=gg_i^{-1}hg_ig^{-1}$. By $\|gg_i^{-1}\|\le n-1$ and by applying our assumption of induction to $gg_i^{-1}$ and $h$, we have $\supp(h)\cap\supp(ghg^{-1}) = \supp(h)\cap\supp(gg_i^{-1}hg_ig^{-1}) \neq \emptyset$ and there exist $k'_1,k'_2,k'_3 \in \Gamma\G$ and $w \in \st_\Gamma(v)$ such that $gg_i^{-1}=k'_1k'_2k'_3$, $\supp(k'_1)\subset \supp(gg_i^{-1}hg_ig^{-1}) = \supp(ghg^{-1})$, $\supp(k'_2) \subset \st_\Gamma(w)$, and $\supp(k'_3) \subset \st_\Gamma(v)$. By $\supp(g_i)=v$, we have $\supp(k'_3g_i) \subset \st_\Gamma(v)$. Hence, $k_1,k_2,k_3 \in \Gamma\G$ defined by $k_1=k'_1$, $k_2=k'_2$, and $k_3=k'_3g_i$ satisfy the conclusion.

    In case (2), we have $gg_i^{-1} = g_1\cdots g_{i-1}g_{i+1} \cdots g_n$ and $\emptyset \neq \supp(g_i^{-1}hg_i) \subset \st_\Gamma(v)$ by $\supp(g_i)=\supp(h_j) \in \st_\Gamma(v)$ and $\emptyset \neq \supp(h) \subset \st_\Gamma(v)$. By $\|gg_i^{-1}\|\le n-1$ and by applying our assumption of induction to $gg_i^{-1}$ and $g_ihg_i^{-1}$, we have $\supp(h)\cap\supp(ghg^{-1}) = \supp(h)\cap\supp(gg_i^{-1}g_ihg_i^{-1}g_ig^{-1}) \neq \emptyset$ and there exist $k'_1,k'_2,k'_3 \in \Gamma\G$ and $w \in \st_\Gamma(v)$ such that $gg_i^{-1}=k'_1k'_2k'_3$, $\supp(k'_1)\subset \supp(gg_i^{-1}g_ihg_i^{-1}g_ig^{-1}) = \supp(ghg^{-1})$, $\supp(k'_2) \subset \st_\Gamma(w)$, and $\supp(k'_3) \subset \st_\Gamma(v)$. By $\supp(g_i)=\supp(h_j) \in \st_\Gamma(v)$, we have $\supp(k'_3g_i) \subset \st_\Gamma(v)$. Hence, $k_1,k_2,k_3 \in \Gamma\G$ defined by $k_1=k'_1$, $k_2=k'_2$, and $k_3=k'_3g_i$ satisfy the conclusion. In case (3), we get the conclusion in the same way as case (2).
\end{proof}

\begin{cor}\label{cor:stabilizer of two far vertices is trivial}
    Let $\Gamma$ be a simplicial graph with $\girth(\Gamma) > 4$ and $\G=\{G_v\}_{v\in V(\Gamma)}$ be a collection of non-trivial groups. The following hold.
    \begin{itemize}
        \item[(1)]
        For any $a,b\in V(\Gamma^e)$, we have $\stab_{\Gamma\G}(a)\cap\stab_{\Gamma\G}(b) = \la H \mid H \in \st_{\Gamma^e}(a) \cap \st_{\Gamma^e}(b) \ra$. In particular, $\stab_{\Gamma\G}(a)\cap\stab_{\Gamma\G}(b) \neq \{1\}$ if and only if $d_{\Gamma^e}(a,b) \le 2$.
        \item[(2)] 
        For any $v \in V(\Gamma)$ and $g \in \Gamma\G \setminus \stab_{\Gamma\G}(v)$, there exist $w \in \st_\Gamma(v)$  and $g_1 \in \Gamma\G$ such that $\stab_{\Gamma\G}(v) \cap g \stab_{\Gamma\G}(v) g^{-1} \subset g_1 G_w g_1^{-1}$.
        \item[(3)]
        If, in addition, the collection $\G=\{G_v\}_{v\in V(\Gamma)}$ is non-trivial finite groups, then for any $v,w \in V(\Gamma)$ with $v \neq w$ and any $g \in \Gamma\G$, the group $\stab_{\Gamma\G}(v) \cap g \stab_{\Gamma\G}(w) g^{-1}$ is finite.
    \end{itemize}
\end{cor}

\begin{proof}
    (1) We first show that $d_{\Gamma^e}(a,b) > 2$ implies $\stab_{\Gamma\G}(a)\cap\stab_{\Gamma\G}(b) = \{1\}$. To show the contraposition, suppose that $a,b\in V(\Gamma^e)$ satisfy $\stab_{\Gamma\G}(a)\cap\stab_{\Gamma\G}(b) \neq \{1\}$. Without loss of generality, we assume $a\in V(\Gamma)$. Let $g\in\Gamma\G$ satisfy $b=g.v(b)$. By $\stab_{\Gamma\G}(a)\cap\stab_{\Gamma\G}(b) \neq \{1\}$, there exist $k \in \stab_{\Gamma\G}(a)\setminus\{1\}$ and $h \in \stab_{\Gamma\G}(v(b))\setminus\{1\}$ such that $k=ghg^{-1}$. By Corollary \ref{cor:Stab(v)}, we have $\emptyset\neq\supp(k)\subset\st_\Gamma(a)$ and $\emptyset\neq\supp(h) \subset \st_\Gamma(v(b))$. By Lemma \ref{lem:conjugation of stabilizer}, we have $\supp(h)\cap\supp(ghg^{-1}) \neq \emptyset$ and there exist $k_1,k_2,k_3 \in \Gamma\G$ and $w \in \st_\Gamma(v(b))$ such that $g=k_1k_2k_3$, $\supp(k_1) \subset \supp(ghg^{-1}) = \supp(k)\subset \st_\Gamma(a)$, $\supp(k_2) \subset \st_\Gamma(w)$, and $\supp(k_3) \subset \st_\Gamma(v(b))$. Here, we retake $k_2$ and $k_3$ so that $\|k_2\|$ is the minimum among all $\|k_2'\|$, where $k_2',k_3' \in \Gamma\G$ satisfy $k_2k_3 = k_2'k_3'$, $\supp(k_2') \subset \st_\Gamma(w)$, and $\supp(k_3') \subset \st_\Gamma(v(b))$. In particular, $\|k_2(k_3hk_3^{-1})\| = \|k_2\| + \|k_3hk_3^{-1}\|$ and $\|(k_3hk_3^{-1})k_2^{-1}\| = \|k_3hk_3^{-1}\|+\|k_2^{-1}\|$ hold. Also, by $\emptyset \neq \supp(h)\cap\supp(ghg^{-1}) \subset \st_\Gamma(v(b))\cap\st_\Gamma(a)$, we have $d_\Gamma(a,v(b)) \le 2$. We will discuss two cases (i) and (ii).

    (i) When $\|k_2(k_3hk_3^{-1})k_2^{-1}\| < \|k_2\|+ \|k_3hk_3^{-1}\|+\|k_2^{-1}\|$, by $\|k_2(k_3hk_3^{-1})\| = \|k_2\| + \|k_3hk_3^{-1}\|$ and $\|(k_3hk_3^{-1})k_2^{-1}\| = \|k_3hk_3^{-1}\|+\|k_2^{-1}\|$, there exist a syllable $s$ of $k_2$ and a syllable $s'$ of $k_2^{-1}$ that cancel in $k_2(k_3hk_3^{-1})k_2^{-1}$. Hence, we have $\supp(k_3hk_3^{-1}) \subset \lk_\Gamma(\supp(s))$ and also $\supp(s) \notin \st_\Gamma(v(b))$ by the minimality of $\|k_2\|$ (see Remark \ref{rem:notation of supp}). By $\supp(k_3hk_3^{-1}) \subset \lk_\Gamma(\supp(s)) \cap \st_\Gamma(v(b))$, $\supp(s) \in \st_\Gamma(w) \setminus \st_\Gamma(v(b))$, and $\girth(\Gamma)>4$, we have $\supp(k_3hk_3^{-1}) = \{w\}$. By this, $k=(k_1k_2)(k_3hk_3^{-1})(k_1k_2)^{-1}$, and Lemma \ref{lem:conjugation of vertex group}, we have $w \in \supp(k) \subset \st_\Gamma(a)$. Thus, by $b=k_1k_2k_3.v(b)=k_1k_2.v(b)$, $k_1k_2.w=k_1.w$, and $a=k_1.a$, we have $d_{\Gamma^e}(a,b)
        =
        d_{\Gamma^e}(a, k_1k_2.v(b))
        \le
        d_{\Gamma^e}(k_1.a, k_1.w) + d_{\Gamma^e}(k_1k_2.w ,k_1k_2.v(b))
        \le
        1+1=2$.

    (ii) When $\|k_2(k_3hk_3^{-1})k_2^{-1}\| = \|k_2\|+ \|k_3hk_3^{-1}\|+\|k_2^{-1}\|$, we have $\supp(k_2) \cup \supp(k_3hk_3^{-1}) = \supp(k_2(k_3hk_3^{-1})k_2^{-1}) = \supp(k_1^{-1}kk_1) \subset \st_\Gamma(a)$. This implies $a=k_1k_2.a$. By this and $b=k_1k_2k_3.v(b)=k_1k_2.v(b)$, we have $d_{\Gamma^e}(a,b) = d_{\Gamma^e}(k_1k_2.a, k_1k_2.v(b)) = d_{\Gamma^e}(a,v(b)) \le 2$. 
    
    By (i) and (ii), we have shown $\stab_{\Gamma\G}(a)\cap\stab_{\Gamma\G}(b) \neq \{1\}$ implies $d_{\Gamma^e}(a,b) \le 2$.

    Next, suppose $d_{\Gamma^e}(a,b) = 2$ and let $c \in V(\Gamma^e)$ satisfy $\{(a,c),(c,b)\} \subset E(\Gamma^e)$. By Lemma \ref{lem:edges of Gammae come from Gamma}, we may assume $\{a,c\}\subset V(\Gamma)$ without loss of generality. Let $g\in\Gamma\G$ satisfy $(c,b)=g.(v(c),v(b))$ (such $g$ exists by Lemma \ref{lem:edges of Gammae come from Gamma}). By $c \in V(\Gamma)$ and Corollary \ref{cor:vertex orbits are disjoint}, we have $v(c) = c$, hence $g \in \stab_{\Gamma\G}(c)$. 
    
    Let $k \in \stab_{\Gamma\G}(a)\cap\stab_{\Gamma\G}(b)$. By $\stab_{\Gamma\G}(b)=g\stab_{\Gamma\G}(v(b))g^{-1}$, there exists $h \in \stab_{\Gamma\G}(v(b))$ such that $k = ghg^{-1}$. Note $\supp(k) \subset \st_\Gamma(a)$, $\supp(g) \subset \st_\Gamma(c)$, and $\supp(h) \subset \st_\Gamma(v(b))$ by Corollary \ref{cor:Stab(v)}. We will discuss two cases (i') and (ii').
    
    (i') When $a\neq v(b)$, we have $\st_\Gamma(a)\cap \big(\st_\Gamma(c)\cup \st_\Gamma(v(b))\big) = \{a,c\}$ and $\st_\Gamma(c)\cap \st_\Gamma(v(b)) = \{c,v(b)\}$ by $\girth(\Gamma)>4$. Hence, we can see $\supp(k) \subset \{a,c\}$ and $\supp(h) \subset \{c,v(b)\}$ by comparing the support of the syllables on each side of the equality $k = ghg^{-1}$. In particular, $\{k,g,h\} \subset \stab_{\Gamma\G}(c) = G_c \times \big(*_{v \in \lk_\Gamma(c)}G_v\big)$ (note $\girth(\Gamma)>4$). Hence, we have $\supp(h) \subset \{c\}$, which implies $\supp(k) \subset \{c\}$. Indeed, $v(b) \in \supp(h)$ would imply $v(b) \in \supp(k)$, which contradicts $\supp(k) \subset \{a,c\}$. 

    (ii') When $a =v(b)$, by $a\neq b ~(=g.a)$, there exists a syllable $s$ of $g$ such that $\supp(s) \notin \{a,c\}$. Since $s$ must cancel in $ghg^{-1}$ by $\supp(s) \cap \supp(k) \subset \supp(s) \cap \st_\Gamma(a) =\emptyset$ but cannot cancel with any syllable of $h$ by $\supp(s) \cap \supp(h) \subset \supp(s) \cap \st_\Gamma(a) =\emptyset$, we can see $\supp(h) \subset \{a,c\}$. In particular, $\{g,h\} \subset \stab_{\Gamma\G}(c) = G_c \times \big(*_{v \in \lk_\Gamma(c)}G_v\big)$. Hence, by $\supp(s) \notin \{a,c\}$ and $\supp(k) = \supp(ghg^{-1})\subset \st_\Gamma(a)\cap \st_\Gamma(c) = \{a,c\}$, we can see $\supp(h) \subset \{c\}$, which implies $\supp(k) \subset \{c\}$.
    
    By (i') and (ii'), we have shown $\stab_{\Gamma\G}(a)\cap\stab_{\Gamma\G}(b) \subset G_c$. The converse inclusion is straightforward. By a similar argument on the support of syllables, we can also show $\girth(\Gamma^e)>4$ using $\girth(\Gamma)>4$. Hence, we have $\{G_c\}=\st_{\Gamma^e}(a) \cap \st_{\Gamma^e}(b)$.

    Finally, when $d_{\Gamma^e}(a,b) = 1$, we can see $\stab_{\Gamma\G}(a)\cap\stab_{\Gamma\G}(b) = a \times b$ by Lemma \ref{lem:edges of Gammae come from Gamma} and Remark \ref{rem:intersection of stabilizers} (1). Here, note $\st_{\Gamma^e}(a) \cap \st_{\Gamma^e}(b) = \{a,b\}$ by $\girth(\Gamma^e)>4$. The case $d_{\Gamma^e}(a,b) = 0$ follows from Corollary \ref{cor:Stab(v)}. Indeed, assume $a \in V(\Gamma)$ without loss of generality, then we have $\stab_{\Gamma\G}(a) = \la G_v \mid v \in \st_{\Gamma}(a) \ra = \la H \mid H \in \st_{\Gamma^e}(a) \ra$. Here, the second equality holds since for any $g \in \Gamma\G$ and $v \in V(\Gamma)$ satisfying $gG_vg^{-1} \in \lk_{\Gamma^e}(a)$, we have $g \in \stab_{\Gamma\G}(a)$ and $v \in \lk_\Gamma(a)$ by Corollary \ref{cor:vertex orbits are disjoint} and Lemma \ref{lem:edges of Gammae come from Gamma}.
    
    (2) By $g \in \Gamma\G \setminus \stab_{\Gamma\G}(v)$ and Lemma \ref{lem:edges of Gammae come from Gamma}, we have $d_{\Gamma^e}(v,g.v) \ge 2$. By the proof of Corollary \ref{cor:stabilizer of two far vertices is trivial} (1), the group $\stab_{\Gamma\G}(v) \cap \stab_{\Gamma\G}(g.v)$ is $\{1\}$ if $d_{\Gamma^e}(v,g.v)>2$ and $c$ if $d_{\Gamma^e}(v,g.v)=2$ where $\{c\}=\st_{\Gamma^c}(v) \cap \st_{\Gamma^c}(g.v)$.

    (3) Note $g \stab_{\Gamma\G}(w) g^{-1} = \stab_{\Gamma\G}(g.w)$. By $v \neq w$ and Corollary \ref{cor:vertex orbits are disjoint}, we have $v \neq g.w$. Hence, by the proof of Corollary \ref{cor:stabilizer of two far vertices is trivial} (1), the group $\stab_{\Gamma\G}(v) \cap \stab_{\Gamma\G}(g.w)$ is $\{1\}$ if $d_{\Gamma^e}(v,g.w)>2$, $c$ if $d_{\Gamma^e}(v,g.w)=2$ where $\{c\}=\st_{\Gamma^c}(v) \cap \st_{\Gamma^c}(g.w)$, and $v \times g.w$ if $d_{\Gamma^e}(v,g.w)=1$. Here, $v$, $c$, and $g.w$ are finite subgroups of $\Gamma\G$ by the assumption on $\G$.
\end{proof}

\subsection{Connection to the crossing graph and the contact graph}
\label{subsec:Connection to the crossing and contact graphs}

In this section, we show that the extension graph and its coned-off graph are isomorphic to two famous graphs associated to a quasi-median graph, namely the crossing graph and the contact graph, which were studied by Genevois in \cite{Gen17}. See Section \ref{subsec:The crossing graph and the contact graph of a quasi-median graph} for notations related to a quasi-median graph. We first introduce the coned-off graph of the extension graph in Definition \ref{def:coned-off extension graph}, which plays an important role in the proof of Theorem \ref{thm:intro asymptotic dimension}. It is an analogous object to the projection complex in \cite{BBF15}.

\begin{defn}\label{def:coned-off extension graph}
    Let $\Gamma$ be a simplicial graph and $\G=\{G_v\}_{v\in V(\Gamma)}$ be a collection of non-trivial groups. We define the graph $\hGammae$ by
    \begin{align*}
    V(\hGammae) &= V(\Gamma^e), \\
    E(\hGammae) &= \{ (x,y) \in V(\Gamma^e)\times V(\Gamma^e) \mid \text{ $x \neq y$ and $\exists\, g \in \Gamma\G, \{x,y\} \subset g.\Gamma$} \}
    \end{align*}
    (see Definition \ref{def:extension graph} for $\Gamma^e$). We call $\hGammae$ the \emph{coned-off extension graph} of $\Gamma\G$.
\end{defn}

\begin{rem}
    By Lemma \ref{lem:edges of Gammae come from Gamma}, $\Gamma^e$ is a subgraph of $\hGammae$.
\end{rem}

\begin{prop}\label{prop:connection to crossing and contact graphs}
    Let $\Gamma$ be a simplicial graph and $\G=\{G_v\}_{v\in V(\Gamma)}$ be a collection of non-trivial groups. Let $X$ be the Cayley graph of $\Gamma\G$ with respect to $\bigsqcup_{v\in V(\Gamma)}(G_v\setminus\{1\})$. Then, there exists a $\Gamma\G$-equivariant bijection $F \colon V(\Gamma^e) \to \H(X)$ such that the maps $F \colon \Gamma^e \to \Delta X$ and $F \colon \hGammae \to \mathcal{C} X$ are both graph isomorphisms.
\end{prop}

\begin{proof}
    Let $g,h \in \Gamma\G$ and $v,w \in V(\Gamma)$. Suppose $g.v = h.w \in V(\Gamma)$, then $v=w$ and $\supp(g^{-1}h) \subset \st_\Gamma(v)$ by Corollary \ref{cor:vertex orbits are disjoint} and Corollary \ref{cor:Stab(v)}. This implies $hJ_w = g(g^{-1}h)J_v = gJ_v$ by Proposition \ref{prop:results on quasi-median graph} (3). Hence, the map $F \colon V(\Gamma^e) \to \H(X)$ defined by $F(g.v) = gJ_v$ ($g \in \Gamma\G$, $v \in V(\Gamma)$) is well-defined and $\Gamma\G$-equivariant. By Proposition \ref{prop:results on quasi-median graph} (2), the map $F$ is surjective. To show that $F$ is injective, suppose $gJ_v = hJ_w$. Since all edges in $gJ_v$ (resp. $hJ_w$) are labeled by elements in $G_v$ (resp. $G_w$) in the Cayley graph $X$, we have $v = w$. Also, $g^{-1}h \in N'(J_v)$ by $1 \in N'(J_w)$. This implies $\supp(g^{-1}h) \subset \st_\Gamma(v)$ by Proposition \ref{prop:results on quasi-median graph} (3). Hence, $h.w = g(g^{-1}h).v = g.v$ by Corollary \ref{cor:Stab(v)}. Thus, the map $F$ is bijective.

    Suppose $(g.v,h.w) \in E(\Gamma^e)$, then $(v,w) \in E(\Gamma)$ and $(g.v,h.w) = (k.v,k.w)$ for some $k \in \Gamma\G$ by Lemma \ref{lem:edges of Gammae come from Gamma}. This implies $F(g.v) = kJ_v$ and $F(h.w) = kJ_w$. By taking $a \in G_v\setminus\{1\}$ and $b \in G_w\setminus\{1\}$, we get the square induced by the vertices $\{k, ka,kb,kab\}$ in $X$. Hence, $F(g.v)$ and $F(h.w)$ are adjacent in $\Delta X$. Conversely, suppose that $gJ_v$ and $hJ_w$ are adjacent in $\Delta X$. By $N'(gJ_v) \cap N'(hJ_w) \neq \emptyset$ and Proposition \ref{prop:results on quasi-median graph} (3), there exist $s,t \in \Gamma\G$ with $\supp(s) \subset \st_\Gamma(v)$ and $\supp(t) \subset \st_\Gamma(w)$ such that $gs = ht$. This implies $g.v = gs.v$ and $h.w = g(g^{-1}h).w = g(st^{-1}).w = gs.w$. Also, $(v,w) \in E(\Gamma)$ by Proposition \ref{prop:results on quasi-median graph} (4). Hence, $(g.v,h.w) = gs.(v,w) \in E(\Gamma^e)$. Thus, $F$ is an isomorphism from $\Gamma^e$ to $\Delta X$.

    Suppose $(g.v,h.w) \in E(\hGammae)$, then $(g.v,h.w) = (k.v,k.w)$ for some $k \in \Gamma\G$. This implies $k \in N'(gJ_v) \cap N'(hJ_w)$. Hence, $gJ_v$ and $hJ_w$ are adjacent in $\mathcal{C}X$. The converse direction can be shown in the same way as the argument for $\Gamma^e$. Thus, $F \colon \hGammae \to \mathcal{C} X$ is isomorphic.
\end{proof}

\begin{cor}\label{cor:coned off extension graph is quasitree}
    Let $\Gamma$ be a simplicial graph and $\G=\{G_v\}_{v\in V(\Gamma)}$ be a collection of non-trivial groups. Then, $\hGammae$ is quasi-isometric to a tree.
\end{cor}

\begin{proof}
    This follows from Proposition \ref{prop:connection to crossing and contact graphs} and Theorem \ref{thm:contact graph is quasitree}.
\end{proof}

\subsection{Admissible paths}\label{subsec:Admissible paths}

In this section, we introduce the notion of an admissible path. This notion is useful to prove classification of geodesic bigons and triangles in the extension graph in Section \ref{subsec:Classification of geodesic bigons and triangles}. 

In Section \ref{subsec:Admissible paths}, suppose that $\Gamma$ is a connected simplicial graph with $\girth(\Gamma) > 20$ and $\G=\{G_v\}_{v \in V(\Gamma)}$ is a collection of non-trivial groups. Since $\Gamma$ is connected, $\Gamma^e$ is also connected by Lemma \ref{lem:Gammae is connected if Gamma is connected}. We begin with defining admissible paths below.

\begin{defn}\label{def: admissible path}
    Let $a,b\in V(\Gamma^e)$ with $a\neq b$. Let $p=(p_0,\cdots,p_N)$ be a path in $\Gamma^e$ from $a$ to $b$ and $\mathbf{x}=(x_0,\cdots,x_n)$ be a subsequence of $V(p)$. We call $p$ \emph{admissible with respect to} $\mathbf{x}$ (and denote $(p, \mathbf{x})$) if we have $n \in \NN$, $x_0=a$, and $x_n=b$ and the following three conditions (1)-(3) hold.
    \begin{itemize}
    \item [(1)]
        For any $i\in\{1,\cdots,n\}$, the subpath $p_{[x_{i-1},x_i]}$ is geodesic in $\Gamma^e$ and there exists $g_i \in \Gamma\G$ such that $p_{[x_{i-1},x_i]} \subset g_i.\Gamma$.
    \item [(2)]
        For any $i\in\{1,\cdots,n-1\}$, no $g\in \Gamma\G$ satisfies $p_{[x_{i-1},x_{i+1}]} \subset g.\Gamma$.
    \item [(3)]
        For any $i\in\{1,\cdots,n-1\}$, if $\max\{ d_{\Gamma^e}(x_{i-1},x_i), d_{\Gamma^e}(x_i,x_{i+1}) \} \le 4$, then the subpath $p_{[x_{i-1},x_{i+1}]}$ has no backtracking.
    \end{itemize}
    We call a path $p$ in $V(\Gamma^e)$ from $a$ to $b$ \emph{admissible} if there exists a subsequence $\mathbf{x}=(x_0,\cdots,x_n)$ of $V(p)$ such that $p$ is admissible with respect to $\mathbf{x}$. For convenience, we define every path $p=(a)$ of length 0, where $a \in V(\Gamma^e)$, to be admissible with respect to the sequence $(a)$.
\end{defn}

\begin{rem}
    For $v \in V(p_{[x_{i-1},x_i]})$ and $w \in V(p_{[x_{j-1},x_j]})$ with $1 \le i < j \le n$, we will denote by $p_{[v,w]}$ the subpath $p_{[v,x_i]}p_{[x_i,x_{i+1}]}\cdots p_{[x_{j-2},x_{j-1}]} p_{[x_{j-1},w]}$ of $p$ by abuse of notation.
\end{rem}

The following lemma provides a natural example of an admissible path.

\begin{lem}\label{lem: geodesics are admissible}
    Any geodesic path in $\Gamma^e$ is admissible.
\end{lem}

\begin{proof}
    Let $p$ be a geodesic from $a\in V(\Gamma^e)$ to $b\in V(\Gamma^e)$. Since the case $a=b$ is obvious, we assume $a\neq b$. Define $\mathcal{A}$ to be the set of all pairs $(n,\mathbf{x})$, where $n \in \NN$ and $\mathbf{x} = (x_0,\cdots,x_n)$ is a subsequence of $V(p)$ with $x_0 = a$ and $x_n = b$ satisfying Definition \ref{def: admissible path} (1). Take $(N,\mathbf{x}) \in \mathcal{A}$ such that $N = \min\{n' \mid (n',\mathbf{x'}) \in \mathcal{A}\}$. Definition \ref{def: admissible path} (2) is satisfied with $\mathbf{x}$ by minimality of $N$. Definition \ref{def: admissible path} (3) is satisfied since geodesic paths have no backtracking. Thus, $(p,\mathbf{x})$ is admissible.
\end{proof}

Next, we introduce the notion describing how an admissible path travels through copies of a defining graph in Definition \ref{def: admissible decomposition}. The set $\A_1(p,\mathbf{x})$ below can be considered as refinement of $\A_0(p,\mathbf{x})$ and will be used to prove Proposition \ref{prop:Gamma is embedded into Gammae}.

\begin{defn}\label{def: admissible decomposition}
    Let $a\in V(\Gamma)$ and $b\in V(\Gamma^e)$ with $a \neq b$. Let $p$ be an admissible path in $\Gamma^e$ from $a$ to $b$ with respect to a subsequence $\mathbf{x}=(x_0,\cdots,x_n)$ of $V(p)$. Define the sets $\A_0(p,\mathbf{x}), \A_1(p,\mathbf{x}) \subset (\Gamma\G)^n$ by
    \begin{align*}
        \A_0(p,\mathbf{x}) &= \{ (g_1,\cdots,g_n) \in (\Gamma\G)^n \mid \forall\, i\in\{1,\cdots,n\},\, p_{[x_{i-1},x_i]} \subset g_1\cdots g_i. \Gamma\}, \\
        \A_1(p,\mathbf{x}) &= \{ (g_1,\cdots,g_n) \in \A_0(p,\mathbf{x}) \mid\, \|g_1\cdots g_n\|=\|g_1\|+\cdots+\|g_n\| \,\}.
    \end{align*}
\end{defn}

\begin{rem}\label{rem:propoerty of A_0(p,x)}
It's straightforward to see that the following (1)-(3) hold for every $(g_1,\cdots,g_n) \in \A_0(p,\mathbf{x})$. These facts will be used often without mention throughout this paper.
    \begin{itemize}
        \item[(1)]
        By Definition \ref{def: admissible path} (2), we have $g_i\neq1$ for any $i \ge 2$ and $v(x_{i-1}) \neq v(x_i)$ for any $i \ge 1$.
        \item[(2)]
        We have $x_{i-1}=g_1\cdots g_{i-1}.v(x_{i-1})=g_1\cdots g_{i-1}g_i.v(x_{i-1})$ for any $i\ge 1$ by Corollary \ref{cor:vertex orbits are disjoint} (where we define $g_{n+1}=1$ for convenience). Hence, $g_i\in \stab_{\Gamma\G}(v(x_{i-1}))$ for any $i \ge 1$. By this and Corollary \ref{cor:Stab(v)}, $\supp(g_i) \subset \st_\Gamma(v(x_{i-1}))$ for any $i \ge 1$.
        \item[(3)]
        For each $i\ge1$, the path $(g_1\cdots g_i)^{-1}.p_{[x_{i-1},x_i]}$ is a geodesic from $v(x_{i-1})$ to $v(x_i)$ in $\Gamma$ since $p_{[x_{i-1},x_i]}$ is geodesic in $\Gamma^e$.
    \end{itemize}
\end{rem}

We study how an element in $\A_0(p,\mathbf{x})$ can be transformed in Lemma \ref{lem:transformation of admissible decomposition} below.

\begin{lem}\label{lem:transformation of admissible decomposition}
    Let $a\in V(\Gamma)$ and $b\in V(\Gamma^e)$ with $a \neq b$. Let $p$ be an admissible path in $\Gamma^e$ from $a$ to $b$ with respect to a subsequence $\mathbf{x}=(x_0,\cdots,x_n)$ of $V(p)$. Let $(g_1,\cdots,g_n) \in \A_0(p,\mathbf{x})$, then the following hold.
    \begin{itemize}
        \item[(1)]
        If there exists $v \in \st_{\Gamma}(v(x_{i-1}))\cap \st_{\Gamma}(v(x_i))$ with $i \in \{1,\cdots,n\}$, then for any $h \in G_v$, we have $(g_1,\cdots,g_ih,h^{-1}g_{i+1}, \cdots,g_n) \in \A_0(p,\mathbf{x})$, where this means $(g_1,\cdots,g_nh) \in \A_0(p,\mathbf{x})$ when $i=n$.
        \item[(2)]
        If there exists $v \in \st_{\Gamma}(v(x_{i-1}))\cap \st_{\Gamma}(v(x_{i+1}))$ with $i \in \{1,\cdots,n-1\}$ such that $\supp(g_{i+1}) \subset \lk_\Gamma(v)$, then $\{v(x_{i-1}), v(x_{i+1})\} \subset \lk_\Gamma(v(x_i))$ and for any $h \in G_v$, we have $(g_1,\cdots,g_ih, g_{i+1}, h^{-1}g_{i+2}, \cdots,g_n) \in \A_0(p,\mathbf{x})$, where this means $(g_1,\cdots,g_{n-1}h,g_n) \in \A_0(p,\mathbf{x})$ when $i=n-1$.
        \item[(3)]
        If there exists $v \in \st_{\Gamma}(v(x_{i-1}))\cap \st_{\Gamma}(v(x_{j-1}))$ with $1 \le i < j \le n+1$ such that $\bigcup_{i<k<j} \supp(g_k) \subset \lk_\Gamma(v)$, then $j \le i+2$.
    \end{itemize}
\end{lem}

\begin{proof}
    (1) Let $h \in G_v$. We denote $(g'_1,\cdots,g'_n)=(g_1,\cdots,g_ih,h^{-1}g_{i+1}, \cdots,g_n)$, that is, $g'_i=g_ih$, $g'_{i+1}=h^{-1}g_{i+1}$, and $\forall k\notin \{i,i+1\}, g'_k=g_k$. Since we have $g'_1\cdots g'_k = g_1\cdots g_k$ for any $k\neq i$, it's enough to show $p_{[x_{i-1},x_i]} \subset g'_1\cdots g'_i.\Gamma$. For brevity, define the path $q$ in $\Gamma$ by $q=(g_1\cdots g_i)^{-1}.p_{[x_{i-1},x_i]}$. Since $q$ is a geodesic in $\Gamma$ from $v(x_{i-1})$ to $v(x_i)$ and we have $\{v(x_{i-1}), v(x_i)\} \subset \st_\Gamma(v)$ by $v \in \st_{\Gamma}(v(x_{i-1}))\cap \st_{\Gamma}(v(x_i))$, we have $q \subset \st_{\Gamma}(v)$ by $\girth(\Gamma)>4$. This implies $q=h.q \subset h.\Gamma$. Hence, $p_{[x_{i-1},x_i]} = g_1\cdots g_i.q \subset g_1\cdots g_ih. \Gamma = g'_1\cdots g'_i.\Gamma$.

    (2) Let $h \in G_v$. We denote $(g'_1,\cdots,g'_n)=(g_1,\cdots,g_ih, g_{i+1}, h^{-1}g_{i+2}, \cdots,g_n)$, that is, $g'_i=g_ih$, $g'_{i+2}=h^{-1}g_{i+2}$, and $\forall k\notin \{i,i+2\}, g'_k=g_k$. Since we have $g'_1\cdots g'_k = g_1\cdots g_k$ for any $k\notin \{i,i+1\}$, it's enough to show $p_{[x_{k-1},x_k]} \subset g'_1\cdots g'_k.\Gamma$ for $k=i,i+1$. For brevity, define the paths $q_i$ and $q_{i+1}$ in $\Gamma$ by $q_k=(g_1\cdots g_k)^{-1}.p_{[x_{k-1},x_k]}$ for $k=i,i+1$. The paths $q_i$ and $q_{i+1}$ are geodesic in $\Gamma$. Also, we have $\{v(x_{i-1}), v(x_{i+1})\} \subset \st_\Gamma(v)$ and $d_\Gamma(v,v(x_i)) \le 2$ by $v \in \st_{\Gamma}(v(x_{i-1}))\cap \st_{\Gamma}(v(x_{i+1}))$ and $\emptyset \neq \supp(g_{i+1}) \subset \lk_\Gamma(v)$. Hence, we can see $v(x_i)=v$ by Definition \ref{def: admissible path} (3). 
    
    Indeed, if $d_{\Gamma}(v,v(x_i))=1$, then we have $v \in q_i \cap q_{i+1}$ and $\supp(g_{i+1}) \subset \st_\Gamma(v(x_i)) \cap \lk_\Gamma(v) = \{v(x_i)\}$ by $\girth(\Gamma) > 4$. Hence, $g_{i+1}.v=v$. This implies $g_1\cdots g_{i+1}.v \in p_{[x_{i-1},x_i]} \cap p_{[x_i,x_{i+1}]}$, which contradicts that the path $p_{[x_{i-1},x_{i+1}]}$ has no backtracking by $\max\{ d_{\Gamma^e}(x_{i-1},x_i), d_{\Gamma^e}(x_i,x_{i+1}) \} \le 4$. 
    
    Similarly, if $d_{\Gamma}(v,v(x_i))=2$ with $(v,w,v(x_i))$ being geodesic in $\Gamma$, then we have $w \in q_i \cap q_{i+1}$ and $\supp(g_{i+1}) \subset \st_\Gamma(v(x_i)) \cap \lk_\Gamma(v) = \{w\}$ by $\girth(\Gamma) > 6$. Hence, $g_{i+1}.w=w$. This implies $g_1\cdots g_{i+1}.w \in p_{[x_{i-1},x_i]} \cap p_{[x_i,x_{i+1}]}$, which again contradicts that the path $p_{[x_{i-1},x_{i+1}]}$ has no backtracking. 
    
    Thus, $v(x_i)=v$ and $\{v(x_{i-1}), v(x_{i+1})\} \subset \lk_\Gamma(v)$. This implies $q_k=h.q_k \subset h.\Gamma$ for any $k \in \{i,i+1\}$. Hence, $p_{[x_{i-1},x_i]} = g_1\cdots g_i.q_i \subset g_1\cdots g_ih. \Gamma$ and $p_{[x_i,x_{i+1}]} = g_1\cdots g_{i+1}.q_{i+1} \subset g_1\cdots g_{i+1}h. \Gamma = g_1\cdots g_ihg_{i+1}. \Gamma$ by $\supp(g_{i+1}) \subset \lk_\Gamma(v)$.

    (3) Suppose $j > i+2$ for contradiction. We have $\{v(x_{i-1}), v(x_{j-1})\} \subset \st_\Gamma(v)$ and $d_\Gamma(v,v(x_{k-1})) \le 2$ for each $k$ with $i<k<j$ by $v \in \st_{\Gamma}(v(x_{i-1}))\cap \st_{\Gamma}(v(x_{j-1}))$ and $\emptyset \neq \supp(g_k) \subset \lk_\Gamma(v)$ for each $k$ with $i<k<j$. For each $k\in \{i,\cdots,j\}$, define the path $q_k$ in $\Gamma$ by $q_k=(g_1\cdots g_k)^{-1}.p_{[x_{k-1},x_k]}$. The path $q_k$ is a unique geodesic in $\Gamma$ from $v(x_{k-1})$ to $v(x_k)$ by $\girth(\Gamma) > 8$. By $j > i+2$, we can see that there exists $k$ with $i<k<j$ such that $v(x_k) \neq v$ and the paths $q_kq_{k+1}$ has backtracking at $v(x_k)$. Hence, in the same way as the proof of Lemma \ref{lem:transformation of admissible decomposition} (2), we can see that the path $p_{[x_{k-1}, x_{k+1}]}$ has backtracking. Indeed, when $d_{\Gamma}(v,v(x_k))=1$, we have $v \in q_k \cap q_{k+1}$ and $\supp(g_{k+1}) \subset \st_\Gamma(v(x_k)) \cap \lk_\Gamma(v) = \{v(x_k)\}$, hence $g_{k+1}.v=v$. This implies $g_1\cdots g_{k+1}.v \in p_{[x_{k-1},x_k]} \cap p_{[x_k,x_{k+1}]}$. Similarly, if $d_{\Gamma}(v,v(x_k))=2$ with $(v,w,v(x_k))$ being geodesic in $\Gamma$, then we have $w \in q_k \cap q_{k+1}$ and $\supp(g_{k+1}) \subset \st_\Gamma(v(x_k)) \cap \lk_\Gamma(v) = \{w\}$, hence $g_{k+1}.w=w$. This implies $g_1\cdots g_{k+1}.w \in p_{[x_{k-1},x_k]} \cap p_{[x_k,x_{k+1}]}$. Since we have $\max\{ d_{\Gamma^e}(x_{k-1},x_k), d_{\Gamma^e}(x_k,x_{k+1}) \} = \max\{ d_{\Gamma^e}(v(x_{k-1}),v(x_k)), d_{\Gamma^e}(v(x_k),v(x_{k+1})) \} \le 4$, the existence of backtracking of $p_{[x_{k-1}, x_{k+1}]}$ contradicts Definition \ref{def: admissible path} (3).
\end{proof}

Lemma \ref{lem:transformation of admissible decomposition} induces a useful property of $\A_1(p,\mathbf{x})$ in Lemma \ref{lem: properties of A_0 A_1} below.

\begin{lem}\label{lem: properties of A_0 A_1}
    Let $a\in V(\Gamma)$ and $b\in V(\Gamma^e)$ with $a \neq b$. Let $p$ be an admissible path in $\Gamma^e$ from $a$ to $b$ with respect to a subsequence $\mathbf{x}=(x_0,\cdots,x_n)$ of $V(p)$. Then, for any $(g_1,\cdots,g_n) \in \A_0(p,\mathbf{x})$, there exists $(h_1,\cdots,h_n) \in \A_1(p,\mathbf{x})$ such that $g_1\cdots g_n = h_1\cdots h_n$. In particular, $\A_1(p,\mathbf{x}) \neq \emptyset$.
\end{lem}

\begin{proof}
    Let $\mathbf{g}=(g_1,\cdots,g_n) \in \A_0(p,\mathbf{x})$. Let $g_i=h_{i,1}\cdots h_{i, N_i}$ be a normal form of $g_i$ for each $i\in\{1,\cdots,n\}$. Suppose $\|g_1\cdots g_n\| < \sum_{k=1}^n \|g_k\|$, then by Theorem \ref{thm:normal form theorem} and Remark \ref{rem:reduced decomposition}, there exist $i,j$ with $1 \le i < j \le n$ and syllables $h_{i,\ell}$ and $h_{j,m}$ such that $\supp(h_{i,\ell}) = \supp(h_{j,m})$ and $\{\supp(h_{k,k'}) \mid (k=i \wedge \ell < k') \vee (i<k<j) \vee (k=j \wedge k'<m) \} \subset \lk_\Gamma(\supp(h_{i,\ell}))$. By Remark \ref{rem:propoerty of A_0(p,x)} (2), we have $\supp(h_{i,\ell}) \in \supp(g_i)\cap\supp(g_j) \subset \st_\Gamma(v(x_{i-1}))\cap\st_\Gamma(v(x_{j-1}))$. By Lemma \ref{lem:transformation of admissible decomposition} (3), we have $j \le i+2$. 
    
    When $j=i+1$, the sequence $\mathbf{g}_1 \in (\Gamma\G)^n$ defined by $\mathbf{g}_1 = (g^{(1)}_1,\cdots,g^{(1)}_n) = (g_1,\cdots,g_ih_{i,\ell}^{-1},h_{i,\ell}g_{i+1}, \cdots,g_n)$ satisfies $\mathbf{g}_1 \in \A_0(p,\mathbf{x})$ by Lemma \ref{lem:transformation of admissible decomposition} (1). We have $\sum_{k=1}^n\|g^{(1)}_k\|<\sum_{k=1}^n\|g_k\|$ by $g_ih_{i,\ell}^{-1}=h_{i,1}\cdots h_{i,\ell-1}h_{i,\ell+1} \cdots h_{i, N_i}$ and $h_{i,\ell}g_{i+1} = h_{i+1,1}\cdots (h_{i,\ell}h_{i+1,m}) \cdots h_{i+1, N_{i+1}}$. Also, $g^{(1)}_1\cdots g^{(1)}_n = g_1\cdots g_n$ is obvious.
    
    When $j=i+2$, the sequence $\mathbf{g}_1 \in (\Gamma\G)^n$ defined by $\mathbf{g}_1 = (g^{(1)}_1,\cdots,g^{(1)}_n) = (g_1,\cdots,g_ih_{i,\ell}^{-1}, g_{i+1}, h_{i,\ell}g_{i+2}, \cdots,g_n)$ satisfies $\mathbf{g}_1 \in \A_0(p,\mathbf{x})$ by Lemma \ref{lem:transformation of admissible decomposition} (2). We can also see $\sum_{k=1}^n\|g^{(1)}_k\|<\sum_{k=1}^n\|g_k\|$ and $g^{(1)}_1\cdots g^{(1)}_n = g_1\cdots g_n$ by $\supp(g_{i+1}) \subset \lk_\Gamma(\supp(h_{i,\ell}))$. 
    
    Suppose $\|g^{(1)}_1\cdots g^{(1)}_n\| < \sum_{k=1}^n \|g^{(1)}_k\|$, then we can repeat this process and eventually get a sequence $\mathbf{g}, \mathbf{g}_1, \cdots, \mathbf{g}_M  \in \A_0(p,\mathbf{x})$ such that $\sum_{k=1}^n\|g^{(m+1)}_k\|<\sum_{k=1}^n\|g^{(m)}_k\|$, $g^{(m+1)}_1\cdots g^{(m+1)}_n = g^{(m)}_1\cdots g^{(m)}_n$ for any $m \ge 1$ and $\|g^{(M)}_1\cdots g^{(M)}_n\| = \sum_{k=1}^n\|g^{(M)}_k\|$ since we cannot continue this process infinitely by $\forall m \ge 1, \sum_{k=1}^n\|g^{(m)}_k\| \ge 0$. Thus, we have $\mathbf{g}_M \in \A_1(p,\mathbf{x})$ and $g^{(M)}_1\cdots g^{(M)}_n = g_1\cdots g_n$. Hence, $\A_1(p,\mathbf{x})\neq \emptyset$ also follows from $\A_0(p,\mathbf{x})\neq \emptyset$.
\end{proof}

We now prove the key result about admissible paths. Proposition \ref{prop:Gamma is embedded into Gammae} verifies that the notion of an admissible path is an analogue of a path without backtracking in a tree.

\begin{prop}\label{prop:Gamma is embedded into Gammae}
    Let $x,y \in V(\Gamma)$. If $p$ is an admissible path in $\Gamma^e$ from $x$ to $y$, then we have $p \subset \Gamma$ and $p$ is geodesic in $\Gamma^e$.
\end{prop}

\begin{proof}
Since the case $x = y$ is trivial, we assume $x \neq y$. Let $p$ be admissible with respect to a subsequence $\mathbf{x} = (x_0,\cdots,x_n)$ of $V(p)$. By Lemma \ref{lem: properties of A_0 A_1}, there exists $(g_1,\cdots,g_n) \in \A_1(p,\mathbf{x})$. Define $g$ by $g=g_1 \cdots g_n$. Note $v(y) = y$ by $y \in V(\Gamma)$ and Corollary \ref{cor:vertex orbits are disjoint}. This implies $g.y = g.v(y) = y$. Hence, $\supp(g) \subset \st_\Gamma(y)$ by Corollary \ref{cor:Stab(v)}. By this and $\|g\| = \sum_{i = 1}^n \|g_i\|$, we have $\bigcup_{i=1}^n \supp(g_i)=\supp(g) \subset \st_\Gamma(y)$. This implies $d_\Gamma(v(x_i), y) \le 2$ for any $i\ge1$ since we have $\supp(g_{i+1})\subset\st_\Gamma(v(x_i))\cap\st_\Gamma(y)$ and $g_{i+1}\neq1$ by Remark \ref{rem:propoerty of A_0(p,x)} (1) and (2). We claim $n=1$. Suppose $n > 1$ for contradiction, then one of the following three cases holds: (1) $n=2$ and $g_1=1$, (2) $n=2$ and $g_1\neq1$, (3) $n\ge3$.

In case (1), $g_1=1$ implies $v(x_1)\in \Gamma$. By $d_\Gamma(v(x_1),y) \le 2$ and $\supp(g_2)\subset\st_\Gamma(v(x_1))\cap\st_\Gamma(y)$, $g_2$ fixes the unique geodesic in $\Gamma$ from $v(x_1)$ to $y$. This implies $p_{[x_1,y]} \subset \Gamma$. Since we also have $p_{[x,x_1]} \subset \Gamma$ by $g_1=1$, this contradicts $n=2$ and Definition \ref{def: admissible path} (2).

In case (2), we have $d_\Gamma(x,y) \le 2$ by $\emptyset \neq \supp(g_1) \subset \st_\Gamma(x) \cap \st_\Gamma(y)$. By $\max\{d_\Gamma(x,v(x_1)), d_\Gamma(v(x_1),y)\} \le 4$ and $\girth(\Gamma) > 8$, we have $|\geo_\Gamma(x,v(x_1))| = |\geo_\Gamma(v(x_1),y)|=1$. Hence, either (2-a) or (2-b) below holds, (2-a) the unique geodesic $g_1^{-1}.p_{[x,x_1]}$ in $\Gamma$ from $x$ to $v(x_1)$ and the unique geodesic $g^{-1}.p_{[x_1,y]}$ in $\Gamma$ from $v(x_1)$ to $y$ has a backtracking at $v(x_1)$, (2-b) the sequence $(x,v(x_1),y)$ is a geodesic path in $\Gamma$.

In case (2-a), let $c \in V(\Gamma)$ satisfy $(v(x_1),c) \in E(g_1^{-1}.p_{[x,x_1]}) \cap E(g^{-1}.p_{[x_1,y]})$, then we have $g_2.c = c$. Indeed, if $d_\Gamma(v(x_1),y)=1$, then we have $c=y$ and $\supp(g_2)\subset \st_\Gamma(v(x_1))\cap\st_\Gamma(y)=\{v(x_1),y\}$. Also if the sequence $(v(x_1),c',y)$ is geodesic, then we have $c=c'$ and $\supp(g_2)\subset \st_\Gamma(v(x_1))\cap\st_\Gamma(y)=\{c\}$. By $g_2.c = c$, we have $g_1.c = g_1g_2.c \in p_{[x,x_1]} \cap p_{[x_1,y]}$. Hence, the subpaths $p_{[x,x_1]}$ and $p_{[x_1,y]}$ have backtracking at $x_1$. We also have $\max\{d_{\Gamma^e}(x,x_1), d_{\Gamma^e}(x_1,y)\} \le 4$. This contradicts Definition \ref{def: admissible path} (3).

In case (2-b), we have $x_1=g_1.v(x_1)=v(x_1) \in \Gamma$ by $\supp(g_1) \subset \st_\Gamma(x)\cap\st_\Gamma(y)=\{v(x_1)\}$. This implies $p \subset \Gamma$, hence contradicts $n=2$ and Definition \ref{def: admissible path} (2).

In case (3), as in case (2), either (3-a) or (3-b) below holds, (3-a) the unique geodesic in $\Gamma$ from $v(x_{n-2})$ to $v(x_{n-1})$ and the unique geodesic in $\Gamma$ from $v(x_{n-1})$ to $y$ has a backtracking at $v(x_{n-1})$, (3-b) the sequence $(v(x_{n-2}),v(x_{n-1}),y)$ is a geodesic path in $\Gamma$. In case (3-a), we get a contradiction to Definition \ref{def: admissible path} (3) in the same way as case (2-a). In case (3-b), we have $x_{n-1}=g_1\cdots g_{n-1}.v(x_{n-1})=g_1\cdots g_{n-2}.v(x_{n-1})$ by $\supp(g_{n-1}) \subset \st_\Gamma(v(x_{n-2}))\cap\st_\Gamma(y)=\{v(x_{n-1})\}$. By $n\ge3$, this implies $p_{[x_{n-3},x_{n-1}]} \subset g_1\cdots g_{n-2}.\Gamma$. This contradicts Definition \ref{def: admissible path} (2).

Thus, $n = 1$. By this and Definition \ref{def: admissible path} (1), $p$ is geodesic. If $g_1 = 1$, then $p \subset g_1.\Gamma = \Gamma$. Assume $g_1 \neq 1$ in the following, then we have $\emptyset \neq \supp(g_1) \subset \st_\Gamma(x)\cap\st_\Gamma(y)$. This implies $d_\Gamma(x,y) \le 2$. If $d_\Gamma(x,y) = 1$, then $p = (x,y) \subset \Gamma$. If $d_\Gamma(x,y) = 2$, then let the sequence $(x,c,y)$ with $c \in V(\Gamma)$ be the unique geodesic path in $\Gamma$ from $x$ to $y$. By $\supp(g_1) \subset \st_\Gamma(x)\cap\st_\Gamma(y) = \{c\}$, we have $g_1.c = c$. Thus, $p = g_1.(x,c,y) = (x,c,y) \subset \Gamma$
\end{proof}

Proposition \ref{prop:Gamma is embedded into Gammae} has the following immediate corollary. Corollary \ref{cor:Gamma is embedded into Gammae} implies that $\Gamma$ is convex in $\Gamma^e$ and is embedded in $\Gamma^e$ not only as a graph but also as a metric space.

\begin{cor}\label{cor:Gamma is embedded into Gammae}
    If $x,y \in V(\Gamma)$, then for any geodesic $p$ in $\Gamma^e$ from $x$ to $y$, we have $p \subset \Gamma$. In particular, $\Gamma$ is convex in $\Gamma^e$.
\end{cor}

\begin{proof}
    This follows from Lemma \ref{lem: geodesics are admissible} and Proposition \ref{prop:Gamma is embedded into Gammae}.
\end{proof}

\subsection{Classification of geodesic bigons and triangles}\label{subsec:Classification of geodesic bigons and triangles}

In Section \ref{subsec:Classification of geodesic bigons and triangles}, suppose that $\Gamma$ is a connected simplicial graph with $\girth(\Gamma) > 20$ and $\G=\{G_v\}_{v \in V(\Gamma)}$ is a collection of non-trivial groups.

The goal of this section is to prove Proposition \ref{prop:paths must penetrate long} and Proposition \ref{prop: classification of geodesic triangle}. We start with introducing a useful property of an admissible path, which we need in the proofs of Lemma \ref{lem:classification of geodesic bigon}, Proposition \ref{prop:paths must penetrate long}, Proposition \ref{prop: classification of geodesic triangle}, and Lemma \ref{lem:geodesic decomposition for hGammae exists}.

\begin{defn}\label{def:straight}
    We say that an admissible path $p$ in $\Gamma^e$ is \emph{straight} if any subpath $q$ of $p$ such that $q \subset g.\Gamma$ with some $g \in \Gamma\G$ is geodesic.
\end{defn}

\begin{rem}\label{rem:geodesic is a straight admissible path}
    By Lemma \ref{lem: geodesics are admissible}, any geodesic path in $\Gamma^e$ is a straight admissible path.
\end{rem}

\begin{rem}\label{rem:subpath for straightness}
    If $p$ is an admissible path in $\Gamma^e$ with respect to a subsequence $\mathbf{x}=(x_0,\cdots,x_n)$ of $V(p)$ and $q$ is a subpath of $p$ such that $q \subset g.\Gamma$ with some $g \in \Gamma\G$, then by Definition \ref{def: admissible path} (2) there exists $i \in \{0, \cdots, n-2\}$ such that $q$ is a subpath of $p_{[x_i, x_{i+2}]}$.
\end{rem}

The reason straight admissible paths are useful is because it has permanence property when we take its subpaths, which we prove in Lemma \ref{lem:subpath of straight admissible path} below. Note that a subpath of an admissible path is not necessarily admissible in general. In Lemma \ref{lem:subpath of straight admissible path} and Lemma \ref{lem:classification of geodesic bigon}, the point of the condition that a path $p$ has no backtracking is that Definition \ref{def: admissible path} (3) is always satisfied when we consider subpaths of $p$.

\begin{lem}\label{lem:subpath of straight admissible path}
    Let $p$ be a straight admissible path in $\Gamma^e$ without backtracking, then any subpath of $p$ is a straight admissible path.
\end{lem}

\begin{proof}
    Let $p$ be admissible with respect to a subsequence $(x_0,\cdots,x_n)$. Let $q$ be a subpath of $p$ and let $i,j$ with $1 \le i \le j \le n$ satisfy $q_- \in p_{[x_{i-1}, x_i]}$ and $q_+ \in p_{[x_{j-1}, x_j]}$. If there exists $h \in \Gamma\G$ satisfying $q \subset h.\Gamma$, then $q$ is geodesic in $\Gamma^e$ since $p$ is straight, hence $q$ is admissible with respect to $(q_-, q_+)$. In the following, we assume that there exists no $h \in \Gamma\G$ satisfying $q \subset h.\Gamma$. This implies $j \ge i+1$.
    
    When $j = i+1$, $q$ is admissible with respect to $(q_-, x_i, q_+)$.
    
    When $j = i+2$, if there exists $h \in \Gamma\G$ satisfying $q_{[q_-, x_{i+1}]} \subset h.\Gamma$, then $q_{[q_-, x_{i+1}]}$ is geodesic in $\Gamma^e$ since $p$ is straight. Hence, $q$ is admissible with respect to $(q_-, x_{i+1}, q_+)$. Similarly, if there exists $h \in \Gamma\G$ satisfying $q_{[x_i, q_+]} \subset h.\Gamma$, then $q$ is admissible with respect to $(q_-, x_i, q_+)$. If there exists no $h \in \Gamma\G$ satisfying $q_{[q_-, x_{i+1}]} \subset h.\Gamma$ nor $q_{[x_i, q_+]} \subset h.\Gamma$, then $q$ is admissible with respect to $(q_-, x_i, x_{i+1}, q_+)$.

    When $j \ge i+3$, define $i' \in \{i,i+1\}$ by $i' = i+1$ if there exists $h \in \Gamma\G$ satisfying $q_{[q_-, x_{i+1}]} \subset h.\Gamma$ and by $i' = i$ if there exists no $h \in \Gamma\G$ satisfying $q_{[q_-, x_{i+1}]} \subset h.\Gamma$. Since $p$ is straight, $q_{[q_-, x_{i'}]}$ is geodesic in $\Gamma^e$. Similarly, define $j' \in \{j-2,j-1\}$ by $j' = j-2$ if there exists $h \in \Gamma\G$ satisfying $q_{[x_{j-2}, q_+]} \subset h.\Gamma$ and by $j' = j-1$ if there exists no $h \in \Gamma\G$ satisfying $q_{[x_{j-2}, q_+]} \subset h.\Gamma$. Then, $q$ is admissible with respect to $(q_-, x_{i'}, \cdots, x_{j'}, q_+)$.

    It's straightforward to show that $q$ is straight.
\end{proof}

First, we prove Proposition \ref{prop:paths must penetrate long}. Lemma \ref{lem:classification of geodesic bigon} below is an auxiliary lemma for this, which is also used in the proof of Proposition \ref{prop: classification of geodesic triangle}.

\begin{lem}\label{lem:classification of geodesic bigon}
    Let $a,b,c \in V(\Gamma^e)$ and let $p,q$ be straight admissible paths in $\Gamma^e$ without backtracking respectively from $a$ to $b$ and from $a$ to $c$. If $V(p)\cap V(q) = \{a\}$ and there exists $g\in \Gamma\G$ such that $\{b,c\}\subset g.\Gamma$, then there exist $k\in\NN$, $g_1,\cdots,g_k \in \Gamma\G$, a subsequence $(a=)z_0,\cdots,z_k(=b)$ of $V(p)$, and a subsequence $(a=)w_0,\cdots,w_k(=c)$ of $V(q)$ such that
    \begin{align}\label{eq:bigon}
        \forall\, i \in \{1,\cdots,k\},\, p_{[z_{i-1},z_i]} \cup q_{[w_{i-1},w_i]} \subset g_i. \Gamma.
    \end{align}
\end{lem}

\begin{proof}
    Let $p$ and $q$ be admissible with respect to subsequences $\mathbf{x}=(x_0,\cdots,x_n)$ and $\mathbf{y}=(y_0,\cdots,y_m)$ respectively. We'll show the statement by induction on $n+m$. When $n = 0$ (i.e. $a=b$), by $\{a,c\} = \{a,b\} \subset g.\Gamma$ and Proposition \ref{prop:Gamma is embedded into Gammae}, we have $m = 1$ and the statements holds by $p \cup q \subset g.\Gamma$. When $m = 0$, the statement holds similarly by $n = 1$. Define $N \in \NN \cup \{0\}$ by $N = n+m$ and assume that the statement holds for any straight admissible paths $(p', (x'_0, \cdots, x'_{n'}))$ and $(q', (y'_0, \cdots, y'_{m'}))$ with $n'+m' < N$. Let $z \in V(p_{[x_{n-1}, x_n]}) \cap g.\Gamma$ and $w \in V(q_{[y_{m-1}, y_m]}) \cap g.\Gamma$ be the vertices respectively satisfying $d_{\Gamma^e}(x_{n-1},z) = \min\{d_{\Gamma^e}(x_{n-1},z') \mid z' \in V(p_{[x_{n-1}, x_n]}) \cap g.\Gamma\}$ and $d_{\Gamma^e}(y_{m-1},w) = \min\{d_{\Gamma^e}(y_{m-1},w') \mid w' \in V(q_{[y_{m-1}, y_m]}) \cap g.\Gamma\}$. Note $p_{[z,x_n]} \cup q_{[w,y_m]} \subset g.\Gamma$ by Corollary \ref{cor:Gamma is embedded into Gammae}.
    
    If there exists $h \in \Gamma\G$ such that $p_{[x_{n-2}, z]} \subset h.\Gamma$, then $p_{[x_{n-2}, z]}$ is geodesic in $\Gamma^e$ since $p$ is straight. Hence, $p_{[a, z]}$ is a straight admissible path with respect to $(x_0, \cdots, x_{n-2}, z)$. Thus, we can apply our assumption of induction to $(p_{[a, z]}, (x_0, \cdots, x_{n-2},z))$ and $(q, \mathbf{y})$ and see that there exist $g_1,\cdots,g_k \in \Gamma\G$, a subsequence $(a=)z_0,\cdots,z_k(=z)$ of $V(p_{[a, z]})$, and a subsequence $(a=)w_0,\cdots,w_k(=b)$ of $V(q)$ that satisfy the condition \eqref{eq:bigon}. The sequences $(g_1,\cdots,g_k,g)$, $(z_0,\cdots,z_k,x_n)$, and $(w_0,\cdots,w_k,y_m)$ satisfy the condition \eqref{eq:bigon} for $p$ and $q$. If there exists $h \in \Gamma\G$ such that $q_{[y_{m-2}, w]} \subset h.\Gamma$, then similarly we can apply our assumption of induction to $(p, \mathbf{x})$ and $(q_{[a, w]}, (y_0, \cdots, y_{m-2}, w))$ and show the statement for $p$ and $q$.

    Hence, in the following we assume that there exists no $h \in \Gamma\G$ satisfying $p_{[x_{n-2}, z]} \subset h.\Gamma$ nor $q_{[y_{m-2}, w]} \subset h.\Gamma$. In particular, $z \neq x_{n-1}$ and $w \neq y_{m-1}$. Fix $\alpha \in \geo_{\Gamma^e}(z,w)$. Note $\alpha \subset g.\Gamma$ by Corollary \ref{cor:Gamma is embedded into Gammae}. By minimality of $d_{\Gamma^e}(x_{n-1},z)$ and $d_{\Gamma^e}(y_{m-1},w)$, the paths $p_{[x_0, x_{n-1}]}\alpha$ and $q_{[y_0, y_{m-1}]}\alpha^{-1}$ have no backtracking. If there exists no $h \in \Gamma\G$ satisfying $p_{[x_{n-1}, z]} \cup \alpha \subset h.\Gamma$ nor $\alpha \cup q_{[y_{m-1}, w]} \subset h.\Gamma$, then the path $p_{[a,z]}\alpha q^{-1}_{[w,a]}$ becomes admissible with respect to $(x_0, \cdots, x_{n-1}, z, w, y_{m-1}, \cdots, y_0)$. This contradicts Proposition \ref{prop:Gamma is embedded into Gammae}. Hence, there exists $h \in \Gamma\G$ satisfying either $p_{[x_{n-1}, z]} \cup \alpha \subset h.\Gamma$ or $\alpha \cup q_{[y_{m-1}, w]} \subset h.\Gamma$. Assume $p_{[x_{n-1}, z]} \cup \alpha \subset h.\Gamma$ without loss of generality. By $\{x_{n-1}, w\} \subset h.\Gamma$, we can apply our assumption of induction to the straight admissible paths $(p_{[a,x_{n-1}]}, (x_0, \cdots, x_{n-1}))$ and $(q_{[a,w]}, (y_0, \cdots, y_{n-1}, w))$ and see that there exist $g_1,\cdots,g_k \in \Gamma\G$, a subsequence $(a=)z_0,\cdots,z_k(=x_{n-1})$ of $V(p_{[a, x_{n-1}]})$, and a subsequence $(a=)w_0,\cdots,w_k(=w)$ of $V(q_{[a,w]})$ that satisfy the condition \eqref{eq:bigon}. The sequences $(g_1,\cdots,g_k,h,g)$, $(z_0,\cdots,z_k,z,x_n)$, and $(w_0,\cdots,w_k,w, y_m)$ satisfy the condition \eqref{eq:bigon} for $p$ and $q$.
\end{proof}

\begin{figure}[htbp]
\begin{center}
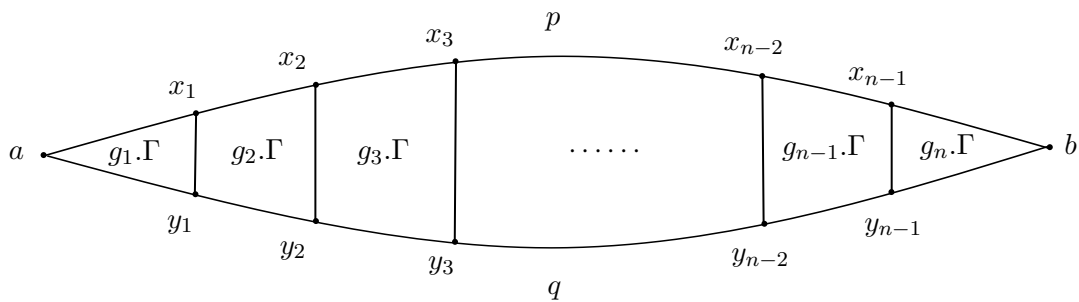
 \caption{Bigon $pq^{-1}$ in Proposition \ref{prop:paths must penetrate long}} 
 \label{fig:bigon}
\end{center} 
\end{figure}

\begin{prop}\label{prop:paths must penetrate long}
    Let $a,b\in V(\Gamma^e)$ be distinct. Suppose that $p$ and $q$ are straight admissible paths in $\Gamma^e$ from $a$ to $b$ without backtracking such that the loop $pq^{-1}$ is a circuit. Then, there exist $n\in\NN$, $g_1,\cdots,g_n \in \Gamma\G$, a subsequence $(x_0,\cdots,x_n)$ of $V(p)$, and a subsequence $(y_0,\cdots,y_n)$ of $V(q)$ with $a = x_0 = y_0$ and $b = x_n = y_n$ that satisfy the following three conditions.
    \begin{itemize}
        \item[(1)]
        $p_{[x_{i-1},x_i]} \cup q_{[y_{i-1},y_i]} \subset g_i. \Gamma$ for any $i \in \{1,\cdots,n\}$.
        \item[(2)] 
        $\min\{|p_{[x_{i-1},x_i]}|,|q_{[y_{i-1},y_i]}|\} \ge 7$ for any $i \in \{1,\cdots,n\}$.
        \item[(3)]
        $g_i \neq g_{i+1}$ for any $i \in \{1,\cdots,n-1\}$.
    \end{itemize}
\end{prop}

\begin{proof}
    Define $\mathcal{A}$ to be the set of all quadruples $(n, \mathbf{g}, \mathbf{x}, \mathbf{y})$, where $n\in\NN$, $\mathbf{g} = (g_1,\cdots,g_n)$ is a sequence in $ \Gamma\G$, $\mathbf{x} = (x_0,\cdots,x_n)$ is a subsequence of $V(p)$, and $\mathbf{y} = (y_0,\cdots,y_n)$ is a subsequence of $V(q)$ with $a = x_0 = y_0$ and $b = x_n = y_n$, that satisfy the condition (1).
    
    We first show $\mathcal{A} \neq \emptyset$. Let $p$ and $q$ be admissible with respect to subsequences $\mathbf{x}=(x_0,\cdots,x_n)$ and $\mathbf{y}=(y_0,\cdots,y_m)$ respectively. Suppose for contradiction that there exists no $g\in \Gamma\G$ satisfying $p_{[x_{n-1},x_n]} \cup q_{[y_{m-1},y_m]} \subset g.\Gamma$, then the path $pq^{-1}$ becomes admissible with respect to $(x_0,\cdots,x_{n-1},b,y_{m-1}\cdots,y_0)$ since $pq^{-1}$ is a circuit. This contradicts Proposition \ref{prop:Gamma is embedded into Gammae} by $p_-=q_-=a$. Hence, there exists $g\in \Gamma\G$ such that $p_{[x_{n-1},x_n]} \cup q_{[y_{m-1},y_m]} \subset g.\Gamma$. By $V(p_{a,x_{n-1}})\cap V(q_{[a,y_{m-1}]}) = \{a\}$ and $\{x_{n-1}, y_{m-1}\} \subset g.\Gamma$, we can apply Lemma \ref{lem:classification of geodesic bigon} to $p_{[a,x_{n-1}]}$ and $q_{[a,y_{m-1}]}$ and see that there exist $k\in\NN$, $h_1,\cdots,h_k \in \Gamma\G$, a subsequence $(a=)z_0,\cdots,z_k(=x_{n-1})$ of $V(p_{[a,x_{n-1}]})$, and a subsequence $(a=)w_0,\cdots,w_k(=y_{m-1})$ of $V(q_{[a,y_{m-1}]})$ satisfying the condition \eqref{eq:bigon} in Lemma \ref{lem:classification of geodesic bigon}. Define the quadruple $(k+1, \mathbf{h}, \mathbf{z}, \mathbf{w})$ by $\mathbf{h} = (h_1,\cdots,h_k,g)$, $\mathbf{z} = (z_0,\cdots,z_k,b)$, and $\mathbf{w} = (w_1,\cdots,w_k,b)$, then we have $(k+1, \mathbf{h}, \mathbf{z}, \mathbf{w}) \in \mathcal{A}$. Thus, $\mathcal{A} \neq \emptyset$.
    
    Let $(n, \mathbf{g}, \mathbf{x}, \mathbf{y}) \in \mathcal{A}$ satisfy $n = \min\{n' \mid (n', \mathbf{g'}, \mathbf{x'}, \mathbf{y'}) \in \mathcal{A}\}$, where $\mathbf{g} = (g_1,\cdots,g_n)$, $\mathbf{x} = (x_0,\cdots,x_n)$, and $\mathbf{y} = (y_0,\cdots,y_n)$. By minimality of $n$, $(n, \mathbf{g}, \mathbf{x}, \mathbf{y})$ satisfy the condition (3). In the following, we'll show that $(n, \mathbf{g}, \mathbf{x}, \mathbf{y})$ satisfy the condition (2). Note that the subpaths $p_{[x_{i-1},x_i]}$ and $q_{[y_{i-1},y_i]}$ are geodesic in $\Gamma^e$ for any $i \ge 1$ since $p$ and $q$ are straight.

    When $n=1$, suppose $|p|\le 6$ for contradiction. Since $p$ has no backtracking, this implies that $p$ is a unique geodesic from $a$ to $b$ by $p \subset g_1.\Gamma$ and $\girth(\Gamma) > 20$. Hence, we have $p=q$ since $q$ is geodesic by $q=q_{[x_0, x_1]} \subset g_1.\Gamma$. This contradicts that $pq^{-1}$ is a circuit.
    
    In what follows, we assume $n \ge 2$. By the condition (3) and Remark \ref{rem:intersection of stabilizers} (3), we have $d_{\Gamma^e}(x_i,y_i) \le 2$ for any $i\in \{0,\cdots,n\}$. In particular, for each $i \in \{1,\cdots,n-1\}$, there exists a unique geodesic in $\Gamma^e$ from $x_i$ to $y_i$, which is contained in $g_i.\Gamma \cap g_{i+1}.\Gamma$ by Corollary \ref{cor:Gamma is embedded into Gammae}. Suppose for contradiction that we have $|p_{[x_{i-1},x_i]}| \le 6$ for some $i$ with $1 \le i \le n$. This implies that $p_{[x_{i-1},x_i]}$ is a unique geodesic in $\Gamma^e$ from $x_{i-1}$ to $x_i$ by $\girth(\Gamma) > 20$.
    
    When $2 \le i \le n-1$, we have $V(p_{[x_{i-1},x_i]}) \cap V(q_{[y_{i-1},y_i]}) = \emptyset$. Since we have $d_{\Gamma^e}(y_{i-1},y_i) \le d_{\Gamma^e}(y_{i-1},x_{i-1}) + d_{\Gamma^e}(x_{i-1},x_i) + d_{\Gamma^e}(x_i,y_i) \le 10$, the subpath $q_{[y_{i-1},y_i]}$ is a unique geodesic in $\Gamma^e$ form $y_{i-1}$ to $y_i$ by $\girth(\Gamma) > 20$. Hence, one of (i) - (iv) below holds, (i) $y_{i-1}=y_i$, (ii) $x_{i-1}=x_i$, (iii) the sequence $(x_{i-1},x_i,y_{i-1},y_i)$ is a geodesic path in $g_i.\Gamma$, (iv) the sequence $(y_{i-1},y_i,x_{i-1},x_i)$ is a geodesic path in $g_i.\Gamma$.
    
    In case (i), we have $x_{i-1} \neq x_i$ by minimality of $n$. Since $p_{[x_{i-1},x_i]}$ is a unique geodesic in $\Gamma^e$ from $x_{i-1}$ to $x_i$ and we have $V(p_{[x_{i-1},x_i]}) \cap V(q_{[y_{i-1},y_i]}) = \emptyset$, one of (i-1) - (i-3) must hold, (i-1) the sequence $(x_{i-1},x_i,y_{i-1})$ is a geodesic path in $g_i.\Gamma$, (i-2) the sequence $(x_i,x_{i-1},y_{i})$ is a geodesic path in $g_i.\Gamma$, (i-3) there exists $z \in V(\Gamma^e)$ such that $\{(x_i,z),(x_{i-1},z),(y_i,z)\} \subset E(\Gamma^e)$, that is, the vertices $x_i,x_{i-1},y_i$ form a tripod whose center is $z$. In case (i-1), we have $(x_{i-1},x_i,y_{i-1}) \subset g_{i-1}.\Gamma$, hence $p_{[x_{i-2},x_i]}\cup q_{[y_{i-2},y_i]} \subset g_{i-1}.\Gamma$. This contradicts minimality of $n$. In case (i-2), we have $(x_i,x_{i-1},y_{i}) \subset g_{i+1}.\Gamma$, hence $p_{[x_{i-1},x_{i+1}]}\cup q_{[y_{i-1},y_{i+1}]} \subset g_{i+1}.\Gamma$. This again contradicts minimality of $n$. In case (i-3), we have $p_{[x_{i-1},x_i]} = (x_{i-1},z,x_i)$ since $p_{[x_{i-1},x_i]}$ is a unique geodesic in $\Gamma^e$ from $x_{i-1}$ to $x_i$. We also have $(x_{i-1},z,y_{i-1}) \subset g_{i-1}.\Gamma$ and $(x_i,z,y_i) \subset g_{i+1}.\Gamma$. This implies $p_{[x_{i-2},z]} \cup q_{[y_{i-2},y_i]} \subset g_{i-1}.\Gamma$ and $p_{[z,x_{i+1}]}\cup q_{[y_i,y_{i+1}]} \subset g_{i+1}.\Gamma$, which contradicts minimality of $n$.
    
    In case (ii), we get the same contradiction as case (i).
    
    In case (iii), we have $(x_{i-1},x_i,y_{i-1}) \subset g_{i-1}.\Gamma$ and $(x_i,y_{i-1},y_i) \subset g_{i+1}.\Gamma$. Note $p_{[x_{i-1},x_i]} = (x_{i-1},x_i)$ and $q_{[y_{i-1},y_i]} = (y_{i-1},y_i)$. Hence, we have $p_{[x_{i-2},x_i]} \cup q_{[y_{i-2},y_{i-1}]} \subset g_{i-1}.\Gamma$ and $p_{[x_i,x_{i+1}]} \cup q_{[y_{i-1},y_{i+1}]} \subset g_{i+1}.\Gamma$, which contradicts minimality of $n$.
    
    In case (iv), we have $(y_{i-1}, y_i, x_{i-1}) \subset g_{i-1}.\Gamma$ and $(y_i, x_{i-1}, x_i) \subset g_{i+1}.\Gamma$. As in case (iii), this implies $p_{[x_{i-2},x_{i-1}]} \cup q_{[y_{i-2},y_i]} \subset g_{i-1}.\Gamma$ and $p_{[x_{i-1},x_{i+1}]} \cup q_{[y_i,y_{i+1}]} \subset g_{i+1}.\Gamma$, which contradicts minimality of $n$.

    When $i=1$, we have $V(p_{[x_0,x_1]}) \cap V(q_{[y_0,y_1]}) = \{a\}$ by $n \ge 2$. Since $p_{[x_0,x_1]}$ and $q_{[y_0,y_1]}$ are both geodesic and contained in $g_1.\Gamma$, one of (i') - (iii') must hold, (i') $y_0=y_1$, (ii') $x_0=x_1$, (iii') the sequence $(x_1,a,y_1)$ is a geodesic path in $g_1.\Gamma$. In case (i'), we have $p_{[x_0,x_1]} \subset g_2.\Gamma$ by $\{x_0, x_1\} = \{y_1, x_1\} \subset g_2.\Gamma$ and Corollary \ref{cor:Gamma is embedded into Gammae}. This implies $p_{[x_0,x_2]} \cup q_{[y_0,y_2]} \subset g_2.\Gamma$, which contradicts minimality of $n$. In case (ii'), we get the same contradiction as case (1'). In case (iii'), we have $(x_1,a,y_1) \subset g_2.\Gamma$. This implies $p_{[x_0,x_2]} \cup q_{[y_0,y_2]} \subset g_2.\Gamma$ by $p_{[x_0,x_1]} = (a,x_1)$ and $q_{[y_0,y_1]} = (a,y_1)$, which contradicts minimality of $n$. When $i=n$, we get the same contradiction as when $i=1$.
\end{proof}

We record the following immediate corollary of Proposition \ref{prop:paths must penetrate long}, which is how Proposition \ref{prop:paths must penetrate long} is used in most of this paper.

\begin{cor}\label{cor:two geodesics with the same endpoints travel through common translations of Gamma}
    Let $a,b\in V(\Gamma^e)$ be distinct and let $p, q \in \geo_{\Gamma^e} (a,b)$. Suppose that the loop $pq^{-1}$ is a circuit. Then, there exist $n\in\NN$, $g_1,\cdots,g_n \in \Gamma\G$, a subsequence $(x_0,\cdots,x_n)$ of $V(p)$, and a subsequence $(y_0,\cdots,y_n)$ of $V(q)$ with $a = x_0 = y_0$ and $b = x_n = y_n$ that satisfy the three conditions (1)-(3) in Proposition \ref{prop:paths must penetrate long}.
\end{cor}

\begin{proof}
    This follows from Remark \ref{rem:geodesic is a straight admissible path} and Proposition \ref{prop:paths must penetrate long}.
\end{proof}

Next, we prove Proposition \ref{prop: classification of geodesic triangle}. Lemma \ref{lem:triangle all of whose sides are on the same planes is on the same plane} is an auxiliary lemma for this.

 \begin{lem}\label{lem:triangle all of whose sides are on the same planes is on the same plane}
    Let $a,b,c\in V(\Gamma^e)$. If there exist $g_1,g_2,g_3\in \Gamma\G$ such that $\{a,b\}\subset g_1.\Gamma$, $\{b,c\}\subset g_2.\Gamma$, and $\{c,a\}\subset g_3.\Gamma$, then there exists $g\in \Gamma\G$ such that $\{a,b,c\}\subset g.\Gamma$.
\end{lem}

\begin{proof}
    Without loss of generality, we assume $g_3=1$. Hence, $\{c,a\}\subset \Gamma$. Define $\mathcal{A}$ to be the set of all pairs $(g_1,g_2) \in \Gamma\G\times \Gamma\G$ such that $\{a,b\} \subset g_1.\Gamma$ and $\{b,c\} \subset g_1g_2.\Gamma$. Note $\mathcal{A} \neq \emptyset$ by the hypothesis. Take $(g_1, g_2) \in \mathcal{A}$ such that $\|g_1\|+\|g_2\| = \min\{\|g'_1\|+\|g'_2\| \mid (g'_1,g'_2) \in \mathcal{A} \}$. If $g_1=1$, then $\{a,b\}\subset\Gamma$, hence $\{a,b,c\}\subset\Gamma$. If $g_2=1$, then $\{b,c\}\subset g_1\Gamma$, hence $\{a,b,c\} \subset g_1\Gamma$. Assume $g_1\neq 1$ and $g_2\neq 1$ in the following. We have $g_1\in \stab_{\Gamma\G}(a)$ and $g_2\in \stab_{\Gamma\G}(v(b))$ by Corollary \ref{cor:vertex orbits are disjoint}. This implies $\supp(g_1) \subset \st_\Gamma(a)$ and $\supp(g_2) \subset \st_\Gamma(b)$ by Corollary \ref{cor:Stab(v)}. For each $i=1,2$, let $g_i=h_{i,1}\cdots h_{i,N_i}$ be a normal form of $g_i$.
    
    Suppose $\|g_1g_2\| < \|g_1\| + \|g_2\|$ for contradiction, then by Theorem \ref{thm:normal form theorem}, there exist syllables $h_{1,i}$ and $h_{2,j}$ such that $\supp(h_{1,i}) = \supp(h_{2,j})$ and $\{\supp(h_{1,i'}) \mid i<i' \}\cup\{\supp(h_{2,j'}) \mid j'<j \} \subset \lk_\Gamma(\supp(h_{1,i}))$. Define $g'_1,g'_2 \in \Gamma\G$ by $g'_1= h_{1,1}\cdots h_{1,i-1}h_{1,i+1} \cdots h_{1,N_1}$ and $g'_2= h_{2,1}\cdots h_{2,j-1}(h_{1,i}h_{2,j})h_{2,j+1} \cdots h_{2,N_2}$. By $\supp(h_{1,i}) \in \supp(g_1) \cap \supp(g_2) \subset \st_\Gamma(a)\cap \st_\Gamma(v(b))$, we have $\{a,b\} = g_1h_{1,i}^{-1}.\{a,v(b)\} \subset g'_1.\Gamma$. We also have $\{b,c\}\subset g_1g_2.\Gamma=g'_1g'_2.\Gamma$ and $\|g'_1\|+\|g'_2\|<\|g_1\|+\|g_2\|$. This contradicts minimality of $\|g_1\|+\|g_2\|$.
    
    Hence, $\|g_1g_2\| = \|g_1\| + \|g_2\|$. This implies $\supp(g_1)\cup\supp(g_2)=\supp(g_1g_2)\subset \st_\Gamma(c)$ since we have $g_1g_2.c=c$ by $c \in g_1g_2.\Gamma \cap \Gamma$ and Corollary \ref{cor:vertex orbits are disjoint}. Hence, $g_2.c = c$. This implies $c=g_1g_2.c=g_1.c \in g_1.\Gamma$. Thus, $\{a,b,c\} \subset g_1.\Gamma$.
\end{proof}

\begin{figure}[tbp]
\begin{center}
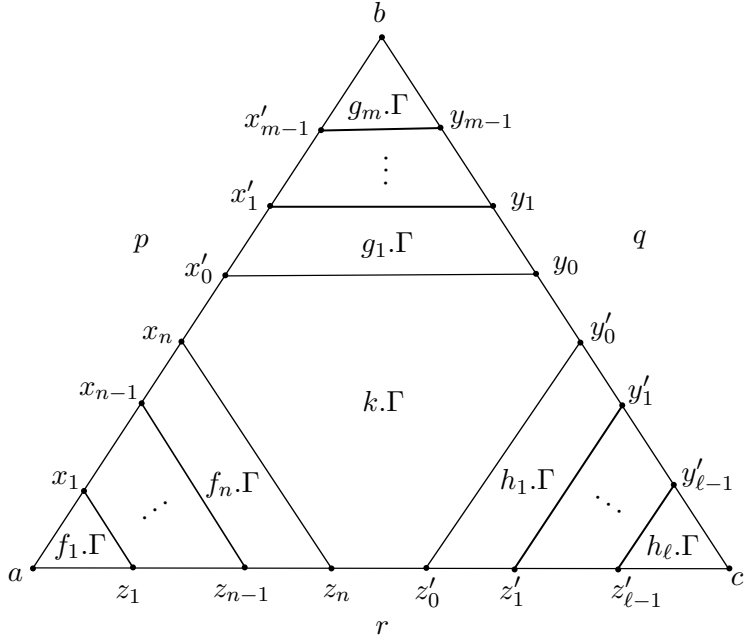
 \caption{Geodesic triangle $pqr^{-1}$ in Proposition \ref{prop: classification of geodesic triangle}} 
 \label{fig:geodesic triangle}
\end{center} 
\end{figure}

\begin{prop}\label{prop: classification of geodesic triangle}
    Let $a,b,c\in V(\Gamma^e)$ be distinct and let $p \in \geo_{\Gamma^e}(a,b)$, $q \in \geo_{\Gamma^e}(b,c)$, and $r \in \geo_{\Gamma^e}(a,c)$ such that the triangle $pqr^{-1}$ is a circuit. Then, there exist a subsequence $(a \! =)\, x_0,\cdots,x_n,x'_0,\cdots,x'_m \, (=\! b)$ of $V(p)$, a subsequence $(b \! =) \, y_m,\cdots,y_0,y'_0,\cdots,y'_\ell \, (= \! c)$ of $V(q)$, a subsequence $(a \! =) \, z_0,\cdots,z_n,z'_0,\cdots,z'_\ell \, (= \! c)$ of $ V(r)$, and $f_1,\cdots,f_n,g_1\cdots,g_m,h_1\cdots,h_\ell,k \in \Gamma\G$ that satisfy the following 8 conditions.
    \begin{itemize}
        \item[(1)]
        $p_{[x_{i-1},x_i]} \cup r_{[z_{i-1},z_i]} \subset f_i. \Gamma$ for any $i \in \{1,\cdots,n\}$.
        \item[(2)] 
        $p_{[x'_{i-1},x'_i]} \cup q^{-1}_{[y_{i-1},y_i]} \subset g_i. \Gamma$ for any $i \in \{1,\cdots,m\}$.
        \item[(3)]
        $q_{[y'_{i-1},y'_i]} \cup r_{[z'_{i-1},z'_i]} \subset h_i. \Gamma$ for any $i \in \{1,\cdots,\ell\}$.
        \item[(4)]
        $p_{[x_n,x'_0]} \cup q_{[y_0,y'_0]} \cup r_{[z_n,z'_0]} \subset k. \Gamma$.
        \item[(5)]
        $\min\{ |p_{[x_{i-1},x_i]}|, |r_{[z_{i-1},z_i]}| \} \ge 7$ for any $i \in \{1,\cdots,n\}$. Also, $f_n \neq k$ and $f_i \neq f_{i+1}$ for any $i \in \{1,\cdots,n-1\}$.
        \item[(6)]
        $\min\{ |p_{[x'_{i-1},x'_i]}|, |q^{-1}_{[y_{i-1},y_i]}| \} \ge 7$ for any $i \in \{1,\cdots,m\}$. Also, $g_1 \neq k$ and $g_i \neq g_{i+1}$ for any $i \in \{1, \cdots, m-1\}$.
        \item[(7)]
        $\min\{ |q_{[y'_{i-1},y'_i]}|, |r_{[z'_{i-1},z'_i]}| \} \ge 7$ for any $i \in \{1,\cdots,\ell\}$. Also, $h_1 \neq k$ and $h_i \neq h_{i+1}$ for any $i \in \{1, \cdots, \ell-1\}$.
        \item[(8)]
        When $n=m=0$ and $d_{\Gamma^e}(a,b)\le 2$, we have $\min\{ |q_{[y_0,y'_0]}|, |r_{[z_n,z'_0]}| \} \ge 7$.
    \end{itemize}
\end{prop}

\begin{proof}
    Define $\mathcal{A}$ to be the set of all tuples $(n, m, \ell, \mathbf{x}, \mathbf{y}, \mathbf{z}, \mathbf{g})$, where $n,m,\ell \in \NN\cup\{0\}$, $\mathbf{x} = (x_0,\cdots,x_n,x'_0,\cdots,x'_m)$ is a subsequence of $V(p)$ with $x_0 = a$ and $x'_m = b$, $\mathbf{y} = (y_m,\cdots,y_0,y'_0,\cdots,y'_\ell)$ is a subsequence of $V(q)$ with $y_m = b$ and $y'_\ell = c$, $\mathbf{z} = (z_0,\cdots,z_n,z'_0,\cdots,z'_\ell)$ is a subsequence of $V(r)$ with $z_0 = a$ and $z'_\ell = c$, and $\mathbf{g} = (f_1,\cdots,f_n,g_1\cdots,g_m,h_1\cdots,h_\ell,k) \subset \Gamma\G$, that satisfy the four conditions (1)-(4).

    First, we show $\mathcal{A} \neq \emptyset$. For this, it's enough to show that there exist $k \in \Gamma\G$, $\xi \in V(p)$, $\eta \in V(q)$, and $\zeta \in V(r)$ such that $\{\xi,\eta,\zeta\} \subset k.\Gamma$. Indeed, by $\{\xi,\eta\} \subset k.\Gamma$ and Remark \ref{rem:geodesic is a straight admissible path}, we can apply Lemma \ref{lem:classification of geodesic bigon} to $p_{[a,\xi]}$ and $r_{[a,\zeta]}$ and see that there exist $f_1,\cdots,f_n \in \Gamma\G$, $(a =) x_0,\cdots,x_n (= \xi) \in V(p_{[a,\xi]})$, and $(a =) z_0,\cdots,z_n (= \zeta) \in V(r_{[a,\zeta]})$ such that $\forall \, i \in \{1,\cdots,n\},\, p_{[x_{i-1},x_i]} \cup r_{[z_{i-1},z_i]} \subset f_i. \Gamma$. We can also apply Lemma \ref{lem:classification of geodesic bigon} to the pair $p^{-1}_{[b, \xi]}$ and $q_{[b,\eta]}$ and the pair $q^{-1}_{[c, \eta]}$ and $r^{-1}_{[c, \zeta]}$, and create an element of $\mathcal{A}$ by combining all the sequences obtained from Lemma \ref{lem:classification of geodesic bigon}. Hence, we'll show the existence of the above $k,\xi,\eta,\zeta$ in the following.

    If there exists no $x \in V(p) \setminus \{b\}$ and $y \in V(q) \setminus \{b\}$ such that $\{x, y\} \subset k.\Gamma$ with some $k \in \Gamma\G$, then the path $pq$ is a straight admissible path. Indeed, let $p$ and $q$ be admissible with respect to subsequences $(u_0,\cdots,u_s)$ of $V(p)$ and $(v_0,\cdots,v_t)$ of $V(q)$ respectively, which is possible by Lemma \ref{lem: geodesics are admissible}. Then, $pq$ is admissible with respect to $(u_0,\cdots,u_{s-1}, b, v_1,\cdots,v_t)$ since there exists no $g \in \Gamma$ satisfying $p_{[u_{s-1}, b]} \cup q_{[b, v_1]} \subset g.\Gamma$ by the non-existence of $x$ and $y$. Let $\gamma$ be a subpath of $pq$ such that $\gamma \subset g.\Gamma$ for some $g \in \Gamma\G$. By $\gamma \subset g.\Gamma$ and the non-existence of $x$ and $y$, we have either $\gamma \subset p$ or $\gamma \subset q$. Hence, $\gamma$ is geodesic. This implies that $pq$ is a straight admissible path without backtracking. Hence, we can see $\mathcal{A} \neq \emptyset$ by applying Proposition \ref{prop:paths must penetrate long} to $pq$ and $r$ (note $m=0$ in this case).

    Hence, in the following we assume that the set $\mathcal{B}$ of all pairs $(x,y)$, where $x \in V(p) \setminus \{b\}$ and $y \in V(q) \setminus \{b\}$ such that $\{x,y\} \subset g.\Gamma$ with some $g \in \Gamma\G$, is non-empty. Take $(x,y) \in \mathcal{B}$ satisfying $d_{\Gamma^e}(x,b) + d_{\Gamma^e}(y,b) = \max\{ d_{\Gamma^e}(x',b) + d_{\Gamma^e}(y',b) \mid (x',y') \in \mathcal{B} \}$. Let $g \in \Gamma\G$ satisfy $\{x,y\} \subset g.\Gamma$. Fix $\alpha \in \geo_{\Gamma^e}(x,y)$. Note $\alpha \subset g.\Gamma$ by Corollary \ref{cor:Gamma is embedded into Gammae}. If $x=a$, then the existence of $k,\xi,\eta,\zeta$ follows by setting $\xi=\zeta=a$, $\eta=y$, and $k=g$. Similarly, if $y=c$, then we can set $\xi=x$, $\eta=\zeta=c$, and $k=g$. Hence, we assume $x \neq a$ and $y \neq c$ in the following. By maximality of $d_{\Gamma^e}(x,b) + d_{\Gamma^e}(y,b)$, the path $p_{[a,x]} \alpha q_{[y,c]}$ has no backtracking.

    We claim that the set $\mathcal{B'}$ of all pairs $(y',z)$, where $y' \in V(q_{[y,c]}) \setminus \{c\}$ and $z \in V(r) \setminus \{c\}$ such that $\{y',z\} \subset h.\Gamma$ with some $h \in \Gamma\G$, is non-empty. Suppose $\mathcal{B'} = \emptyset$ for contradiction. Let $p_{[a,x]}$, $q_{[y,c]}$, and $r$ be admissible with respect to subsequences $(u_0,\cdots,u_s)$ of $V(p_{[a,x]})$, $(v_0,\cdots,v_t)$ of $V(q_{[y,c]})$, and $(w_0,\cdots, w_N)$ of $V(r)$ respectively. By $\mathcal{B'} = \emptyset$ and maximality of $d_{\Gamma^e}(x,b) + d_{\Gamma^e}(y,b)$, the path $p_{[a,x]} \alpha q_{[y,c]} r^{-1}$ is admissible with respect to $(u_0,\cdots,u_s, v_0,\cdots,v_{t-1}, c, w_{N-1},\cdots, w_0)$. This contradicts Proposition \ref{prop:Gamma is embedded into Gammae}. Thus, $\mathcal{B'} \neq \emptyset$.

    Take $(y',z) \in \mathcal{B'}$ satisfying $d_{\Gamma^e}(y',c) + d_{\Gamma^e}(z,c) = \max\{ d_{\Gamma^e}(y'',c) + d_{\Gamma^e}(z',c) \mid (y'',z') \in \mathcal{B'} \}$. Let $h \in \Gamma\G$ satisfy $\{y',z\} \subset h.\Gamma$. Fix $\beta \in \geo_{\Gamma^e}(z,y')$. Note $\beta \subset h.\Gamma$ by Corollary \ref{cor:Gamma is embedded into Gammae}. If $z=a$, then the existence of $k,\xi,\eta,\zeta$ follows by setting $\xi=\zeta=a$, $\eta=y'$, and $k=h$. Hence, we assume $z \neq a$.
    
    We claim $y=y'$. Suppose $y \neq y'$ for contradiction. By maximality of $d_{\Gamma^e}(y',c) + d_{\Gamma^e}(z,c)$, the path $q_{[y,y']} \beta^{-1} r^{-1}_{[z,a]}$ has no backtracking. Let $p_{[a,x]}$, $q_{[y,y']}$, and $r_{[a,z]}$ be admissible with respect to subsequences $(u_0,\cdots,u_s)$ of $V(p_{[a,x]})$, $(v_0,\cdots,v_t)$ of $V(q_{[y,y']})$, and $(w_0,\cdots, w_N)$ of $V(r_{[a,z]})$ respectively. By maximality of $d_{\Gamma^e}(x,b) + d_{\Gamma^e}(y,b)$ and $d_{\Gamma^e}(y',c) + d_{\Gamma^e}(z,c)$, the path $p_{[a,x]} \alpha q_{[y,y']} \beta^{-1} r^{-1}_{[z,a]}$ is admissible with respect to $(u_0,\cdots,u_s, v_0,\cdots,v_t, w_N,\cdots, w_0)$. This contradicts Proposition \ref{prop:Gamma is embedded into Gammae}. Thus, $y=y'$.

    Let $v \in V(\alpha)\cap V(\beta)$ satisfy $\alpha_{[v,y]} = \beta_{[v,y]}$ and $d_{\Gamma^e}(v,y) = \max\{d_{\Gamma^e}(v',y) \mid v' \in V(\alpha)\cap V(\beta), \alpha_{[v',y]} = \beta_{[v',y]}\}$. If $v = x$, then by $x \in \beta \subset h.\Gamma$ the existence of $k,\xi,\eta,\zeta$ follows by setting $\xi=x$, $\eta=y$, $\zeta=z$, and $k=h$. Similarly, if $v = z$, then we can set $\xi=x$, $\eta=y$, $\zeta=z$, and $k=g$. Hence, we assume $v \notin \{x,z\}$ in the following. By maximality of $d_{\Gamma^e}(v,y)$, the path $\alpha_{[x,v]} \beta^{-1}_{[v,y]}$ has no backtracking. Recall that $p_{[a,x]}$ and $r_{[a,z]}$ are admissible with respect to $(u_0,\cdots,u_s)$ and $(w_0,\cdots, w_N)$ respectively.
    
    Define the admissible path $(p',\mathbf{u'})$ as follows depending on the three cases (P1)-(P3). (P1) If there exists no $g' \in \Gamma\G$ satisfying $p_{[u_{s-1}, x]} \cup \alpha_{[x,v]} \subset g'.\Gamma$, then define $p'$ by $p'=p$ and $\mathbf{u'}$ by $\mathbf{u'} = (u_0,\cdots,u_s,v)$. (P2) If there exists $g' \in \Gamma\G$ satisfying $p_{[u_{s-1}, x]} \cup \alpha_{[x,v]} \subset g'.\Gamma$ and the path $p_{[u_{s-1}, x]} \alpha_{[x,v]}$ is geodesic in $\Gamma^e$, then define $p'$ by $p'=p$ and $\mathbf{u'}$ by $\mathbf{u'} = (u_0,\cdots,u_{s-1},v)$. (P3) If there exists $g' \in \Gamma\G$ satisfying $p_{[u_{s-1}, x]} \cup \alpha_{[x,v]} \subset g'.\Gamma$ and the path $p_{[u_{s-1}, x]} \alpha_{[x,v]}$ is not geodesic in $\Gamma^e$, then we have $|p_{[u_{s-1}, x]} \alpha_{[x,v]}| > 10$ by $\girth(\Gamma) > 20$ since the path $p_{[u_{s-1}, x]} \alpha_{[x,v]}$ has no backtracking and is in $g'.\Gamma$. If $|\alpha_{[x,v]}| \ge 3$, then we have $g'=g$ by $\alpha_{[x,v]} \subset g'.\Gamma \cap g.\Gamma$ and Remark \ref{rem:intersection of stabilizers} (3), which implies $\{u_{s-1}, y\} \subset g.\Gamma$ and contradicts maximality of $d_{\Gamma^e}(x,b) + d_{\Gamma^e}(y,b)$. Hence, by $|\alpha_{[x,v]}| \le 2$, we have $|p_{[u_{s-1}, x]}| > 8$. This implies
    \[
    d_{\Gamma^e}(u_{s-1}, v) 
    \ge d_{\Gamma^e}(u_{s-1}, x) - d_{\Gamma^e}(x, v)
    > 8 - 2 
    =6. 
    \]
    Take $\alpha' \in \geo_{\Gamma^e}(u_{s-1}, v)$. Note $\alpha' \subset g'.\Gamma$. Suppose for contradiction that there exists $g'' \in \Gamma\G$ satisfying $p_{[u_{s-2}, u_{s-1}]} \cup \alpha' \subset g''.\Gamma$, then we have $g''=g'$ by $d_{\Gamma^e}(u_{s-1}, v) > 6$ and Remark \ref{rem:intersection of stabilizers} (3). This implies $p_{[u_{s-2}, x]} \subset g'.\Gamma$, hence contradicts Definition \ref{def: admissible path} (2). Thus, define $p'$ by $p' = p_{[a,u_{s-1}]} \alpha'$ and $\mathbf{u'}$ by $\mathbf{u'} = (u_0,\cdots,u_{s-1},v)$, and we can see that $(p',\mathbf{u'})$ is admissible. Here, Definition \ref{def: admissible path} (3) is satisfied by $d_{\Gamma^e}(u_{s-1}, v) > 6$. For brevity, define $s' \in \{0,\cdots,s\}$ by $s' = s$ in case (P1) and by $s' = s-1$ in case (P2) and (P3).

    In the same way, we can define the admissible path $(r',\mathbf{w'})$ and $N' \in \{1,\cdots,N\}$ from $r_{[a,z]}\beta_{[z,v]}$ by using maximality of $d_{\Gamma^e}(y',c) + d_{\Gamma^e}(z,c)$ (recall $y'=y$). If there exists no $f \in \Gamma\G$ satisfying $p'_{[u_{s'}, v]} \cup r'_{[w_{N'}, v]} \subset f.\Gamma$, then we can see that the path $p'r^{\prime -1}$ is admissible with respect to $(u_0,\cdots,u_{s'}, v, w_{N'}, \cdots, w_0)$. This contradicts Proposition \ref{prop:Gamma is embedded into Gammae} by $p'_-=r'_-=a$. Hence, there exists $f \in \Gamma\G$ satisfying $p'_{[u_{s'}, v]} \cup r'_{[w_{N'}, v]} \subset f.\Gamma$. In the cases (P1) and (P2), this implies $x \in V(p'_{[u_{s'}, v]}) \subset f.\Gamma$. In case (P3), since we have $d_{\Gamma^e}(u_{s'}, v) = d_{\Gamma^e}(u_{s-1}, v) > 6$ and $p'_{[u_{s'}, v]} = \alpha' \subset g'.\Gamma \cap f.\Gamma$, we have $g'=f$ by Remark \ref{rem:intersection of stabilizers} (3). This implies $x \in g'.\Gamma = f.\Gamma$. We can argue similarly for $z$ and $r'$, and see $z \in f.\Gamma$.
    
    Thus, we have $\{x,z\} \subset f.\Gamma$. Recall $\{x,y\} \subset g.\Gamma$ and $\{y,z\}=\{y',z\} \subset h.\Gamma$. By applying Lemma \ref{lem:triangle all of whose sides are on the same planes is on the same plane} to $x,y,z$, there exists $k' \in \Gamma\G$ such that $\{x,y,z\} \subset k'.\Gamma$. Hence, the existence of $k,\xi,\eta,\zeta$ follows by setting $\xi=x$, $\eta=y$, $\zeta=z$, and $k=k'$. This finishes the proof of $\mathcal{A} \neq \emptyset$.
   
    Take $(n, m, \ell, \mathbf{x}, \mathbf{y}, \mathbf{z}, \mathbf{g}) \in \mathcal{A}$ satisfying
    \[
    n+m+\ell = \min\{n'+m'+\ell' \mid (n', m', \ell', \mathbf{x'}, \mathbf{y'}, \mathbf{z'}, \mathbf{g'}) \in \mathcal{A}\}.
    \]
    By minimality of $n+m+\ell$, we can show the conditions (5), (6), and (7) in the same way as the proof of Proposition \ref{prop:paths must penetrate long}.
    
    In the following, we'll show the condition (8). We have $z_n=a$ and $y_0=b$ by $n=m=0$. By $\{a,b\}\subset k.\Gamma$ and $d_{\Gamma^e}(a,b)\le 2$, $p$ is a unique geodesic in $\Gamma^e$ from $a$ to $b$.

    When $\ell = 0$, the geodesic triangle $pqr^{-1}$ is contained in $k.\Gamma$ and has no self-intersection. This implies $|p|+|q|+|r| > 20$ by $\girth(\Gamma) > 20$. Since both $q$ and $r$ are geodesic and we have $|p| \le 2$, we have $\min\{|q|, |r|\} > 8$.
    
    When $\ell \ge 1$, we have $d_{\Gamma^e}(y'_0, z'_0) \le 2$ by $\{y'_0, z'_0\} \subset k.\Gamma \cap h_1.\Gamma$ and Remark \ref{rem:intersection of stabilizers} (3). Suppose $|q_{[y_0,y'_0]}| \le 6$ for contradiction, then in the same way as the proof of Proposition \ref{prop:paths must penetrate long}, we can see that one of (1) - (4) must hold, (1) $b=y'_0$, (2) $a=z'_0$, (3) the sequence $(b,y'_0,a,z'_0)$ is a geodesic path in $k.\Gamma$, (4) the sequence $(a,z'_0,b,y'_0)$ is a geodesic path in $k.\Gamma$.
    
    In case (1), in the same way as the proof of Proposition \ref{prop:paths must penetrate long}, we have $a \neq z'_0$ and one of (1-1)-(1-3) must hold: (1-1) the sequence $(b,a,z'_0)$ is a geodesic path in $k.\Gamma$, (1-2) the sequence $(b,z'_0,a)$ is a geodesic path in $k.\Gamma$, (1-3) there exists $w \in V(\Gamma^e)$ such that $\{(a,w),(b,w),(z'_0,w)\} \subset E(\Gamma^e)$, that is, the vertices $a,b,z'_0$ form a tripod whose center is $w$. In case (1-1), by $b=y'_0$, we have $(y'_0,a,z'_0) \subset h_1.\Gamma$. This implies $q_{[b,y'_1]} \cup r_{[a,z'_1]} \subset h_1.\Gamma$, hence contradicts minimality of $N$. In case (1-2), by $p=(b,z'_0,a)$, we have $z'_0 \in V(p) \cap V(r)$. This contradicts that the loop $pqr^{-1}$ is a circuit. In case (1-3), by $p=(a,w,b)$ and $r_{[a,z'_0]} = (a,w,z'_0)$, we have $w \in V(p)\cap V(r)$. This contradicts that the loop $pqr^{-1}$ is a circuit.
    
    In case (2), we get the same contradiction as case (1). In case (3), by $p=(a,y'_0,b)$, we have $y'_0 \in V(p)\cap V(q)$. This contradicts that the loop $pqr^{-1}$ is a circuit. In case (4), by $p=(a,z'_0,b)$, we have $z'_0 \in V(p)\cap V(r)$. This again contradicts that the loop $pqr^{-1}$ is a circuit.
\end{proof}

Lemma \ref{lem:two geodesics form a ladder} below means that two geodesic paths in $\Gamma^e$ whose endpoints are close to each other travel though common copies of $\Gamma$. This is used in the proof of Proposition \ref{prop:extension graph is tight}.

\begin{figure}[htbp]
\begin{center}
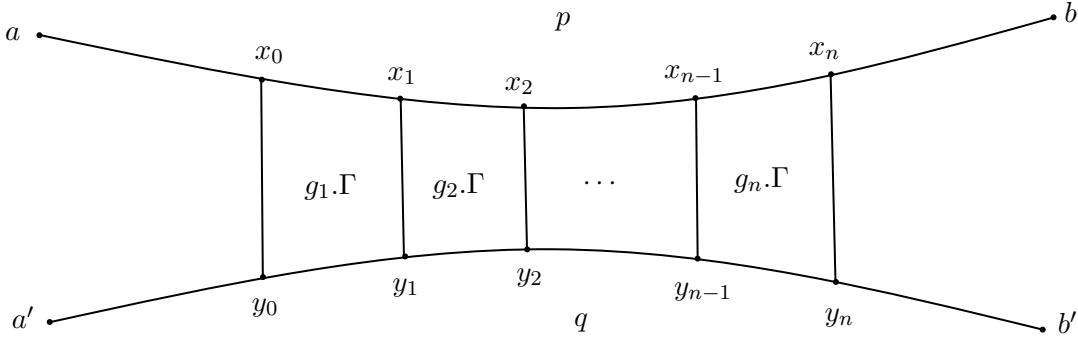
 \caption{Paths $p$ and $q$ in Lemma \ref{lem:two geodesics form a ladder}} 
 \label{fig:Fellow traveling paths}
\end{center} 
\end{figure}

\begin{lem}\label{lem:two geodesics form a ladder}
    Let $a,b \in V(\Gamma^e)$, $r \in \NN$, $a' \in \N_{\Gamma^e}(a,r)$, and $b' \in \N_{\Gamma^e}(b,r)$. Let $p \in \geo_{\Gamma^e}(a,b)$ and $q \in \geo_{\Gamma^e}(a',b')$. Suppose $d_{\Gamma^e}(a,b) \ge 2r + 32$ and $V(p)\cap V(q) = \emptyset$, then there exist $n \in \NN$ a subsequece $(x_0,\cdots,x_n)$ in $V(p)$, a subsequece $(y_0,\cdots,y_n)$ in $V(q)$, and $g_1,\cdots, g_n \in \Gamma\G$ that satisfy the following four conditions.
    \begin{itemize}
        \item[(1)]
        $\max\{d_{\Gamma^e}(a,x_0), d_{\Gamma^e}(a',y_0), d_{\Gamma^e}(b,x_n), d_{\Gamma^e}(b',y_n)\} \le r+22$.
        \item[(2)]
        $p_{[x_{i-1}, x_i]} \cup q_{[y_{i-1}, y_i]} \subset g_i.\Gamma$ for any $i\in \{1,\cdots,n\}$.
        \item[(3)]
        $\min\{|p_{[x_{i-1}, x_i]}|, |q_{[y_{i-1}, y_i]}|\} \ge 5$ for any $i\in \{2,\cdots,n-1\}$.
        \item[(4)]
        $g_i \neq g_{i+1}$ for any $i\in \{1,\cdots,n-1\}$.
    \end{itemize}
\end{lem}

\begin{proof}
    Fix $\alpha \in \geo_{\Gamma^e}(a,a')$, $\beta \in \geo_{\Gamma^e}(b,b')$, and $\gamma \in \geo_{\Gamma^e}(a,b')$. By replacing subpaths of $\gamma$ if necessary, we may assume that there exist $z_0 \in V(r)\cap V(p)$ and $w_0 \in V(\gamma)\cap V(q)$ such that $p_{[a, z_0]}=\gamma_{[a, z_0]}$, $p_{[z_0, b]}\cap \gamma_{[z_0, b']} = \{z_0\}$, $q_{[w_0, b']}=\gamma_{[w_0, b']}$, and $q_{[a', w_0]}\cap \gamma_{[a, w_0]} = \{w_0\}$. Note $d_{\Gamma^e}(a,z_0) < d_{\Gamma^e}(a,w_0)$ by $V(p) \cap V(q) = \emptyset$. We claim that either (i) or (ii) below holds.
    \begin{itemize}
        \item[(i)]
        $\gamma_{[z_0, b']} \subset \N_{\Gamma^e}(b, r+18)$.
        \item[(ii)]
        There exist $n \in \NN$, a subsequence $(z_0 \!=)\,x_0,\cdots ,x_n$ of $V(p_{[z_0, a]})$, a subsequence $z_0,\cdots ,z_n$ of $V(\gamma_{[z_0,b']})$, and $g_1,\cdots, g_n \in \Gamma\G$ such that 
        \begin{itemize}
            \item[-]
            $p_{[x_{i-1}, x_i]} \cup \gamma_{[z_{i-1}, z_i]} \subset g_i.\Gamma$ for any $i \in \{1,\cdots,n\}$,
            \item[-]
            $\min\{|p_{[x_{i-1}, x_i]}|, |\gamma_{[z_{i-1}, z_i]}| \} \ge 7$ and $g_i \neq g_{i+1}$ for any $i \in \{1,\cdots,n-1\}$,
            \item[-]
            $|\gamma_{[z_{n-1},z_n]}| \ge 7$, and $p_{[x_n,b]} \cup \gamma_{[z_n, b']} \subset \N_{\Gamma^e}(b, r+18)$.
        \end{itemize}
    \end{itemize}
\begin{figure}[htbp]
\begin{minipage}[c]{0.5\hsize}
\begin{center}
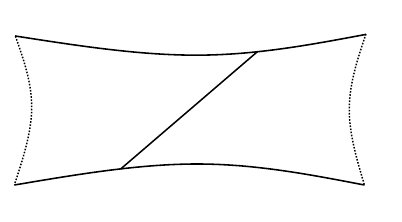
 \caption{Case (ii)} 
 \label{fig:quadrilateral1}
\end{center}
\end{minipage} 
\hfill 
\begin{minipage}[c]{0.5\hsize}
\begin{center}
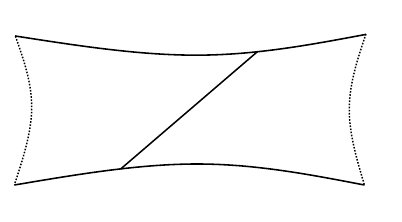
 \caption{Case (ii')}
 \label{fig:quadrilateral2}
\end{center}
\end{minipage} 
\end{figure}
    Indeed, by applying Proposition \ref{prop: classification of geodesic triangle} to the simple geodesic triangle formed by subpaths of $p_{[z_0, b]}$, $\gamma_{[z_0, b']}$, and $\beta$, there exist $N \in \NN\cup\{0\}$, a subsequence $(x_i)_{i=0}^{N+1}$ of $V(p_{[z_0, a]})$ with $x_0=z_0$, a subsequence $(z_i)_{i=0}^{N+1}$ of $V(\gamma_{[z_0,b']})$, $x',z' \in V(\beta)$, and $(g_i)_{i=0}^{N+1}$ in $\Gamma\G$ such that
    \begin{itemize}
        \item[-]
        for any $i\in \{1,\cdots,N\}$, $p_{[x_{i-1}, x_i]} \cup \gamma_{[z_{i-1}, z_i]} \subset g_i.\Gamma$, $\min\{|p_{[x_{i-1}, x_i]}|, |\gamma_{[z_{i-1}, z_i]}| \} \ge 7$, and $g_i \neq g_{i+1}$,
        \item[-] 
        $p_{[x_N, x_{N+1}]} \cup \gamma_{[z_N, z_{N+1}]}\cup \beta_{[x',z']} \subset g_{N+1}.\Gamma$, and $\max\{d_{\Gamma^e}(x_{N+1}, x'), d_{\Gamma^e}(z_{N+1}, z')\} \le 2$.
    \end{itemize}
    By $d_{\Gamma^e}(z_{N+1}, z') \le 2$, we have for any $v \in V(\gamma_{[z_{N+1}, b']})$,
    \begin{align*}
        d_{\Gamma^e}(b,v)
        &\le d_{\Gamma^e}(b,z_{N+1}) + d_{\Gamma^e}(z_{N+1}, v) 
        \le d_{\Gamma^e}(b, z_{N+1}) + d_{\Gamma^e}(z_{N+1}, b') \\
        &\le d_{\Gamma^e}(b, z') + 2 + 2 + d_{\Gamma^e}(z',b')
        =d_{\Gamma^e}(b,b')+4 \le r+4.~~~~~~~(\ast)
    \end{align*}
    Also, for any $v \in V(p_{[x_{N+1}, b']})$, $d_{\Gamma^e}(b,v) \le d_{\Gamma^e}(b, x_{N+1}) \le d_{\Gamma^e}(b, x') + d_{\Gamma^e}(x', x_{N+1}) \le r+2$. Hence, when $|\gamma_{[z_N, z_{N+1}]}| \ge 7$, (ii) holds by defining $n \in \NN$ by $n=N+1$.
    
    When $|\gamma_{[z_N, z_{N+1}]}| < 7$ and $N \ge 1$, we have $d_{\Gamma^e}(z_N,z') \le d_{\Gamma^e}(z_N, z_{N+1}) + d_{\Gamma^e}(z_{N+1}, z') \le 2+7=9$. This implies $\gamma_{[z_N, b']} \subset \N_{\Gamma^e}(b, r+18)$ in the same way as $(\ast)$. Note $d_{\Gamma^e}(z_N,x_N) \le 2$ by $g_N\neq g_{N+1}$ and Remark \ref{rem:intersection of stabilizers} (3). Hence, for any $v \in p_{[x_N, b]}$, we have $d_{\Gamma^e}(b, v) \le d_{\Gamma^e}(b,x_N) \le d_{\Gamma^e}(b,z') + d_{\Gamma^e}(z',z_N) + d_{\Gamma^e}(z_N,x_N) \le r+9+2 = r+11$. Thus, (ii) holds by defining $n$ by $n=N$.
    
    When $|\gamma_{[z_N, z_{N+1}]}| < 7$ and $N = 0$, we have $d_{\Gamma^e}(z_0,z') \le d_{\Gamma^e}(z_0, z_1) + d_{\Gamma^e}(z_1, z') \le 9$, hence (i) holds by the same computation as $(\ast)$. 
    
    Similarly, by applying Proposition \ref{prop: classification of geodesic triangle} to the simple geodesic triangle formed by subpaths of $q^{-1}_{[w_0, a']}$, $\gamma^{-1}_{[w_0, a]}$, and $\alpha$, we can also show that either (i') or (ii') below holds.
    \begin{itemize}
    \item[(i')]
    $d_{\Gamma^e}(a,w_0) \le r+9$.
    \item[(ii')]
    There exist $m \in \NN$, a subsequence $(w_0 \!=)\,y_0,\cdots,y_m$ of $V(q^{-1}_{[w_0, a']})$, a subsequence $w_0,\cdots,w_m$ of $V(\gamma^{-1}_{[w_0,a]})$, and $h_1,\cdots,h_m \in \Gamma\G$ such that
    \begin{itemize}
        \item[-]
        $q^{-1}_{[y_{i-1}, y_i]} \cup \gamma^{-1}_{[w_{i-1}, w_i]} \subset h_i.\Gamma$ for any $i \in \{1,\cdots,m\}$,
        \item[-]
        $\min\{|q^{-1}_{[y_{i-1}, y_i]}|, |\gamma^{-1}_{[w_{i-1}, w_i]}| \} \ge 7$ and $h_i \neq h_{i+1}$ for any $i \in \{1,\cdots,m-1\}$,
        \item[-]
        $|\gamma_{[w_{m-1},w_m]}| \ge 7$, and $\max\{d_{\Gamma^e}(y_m, a), d_{\Gamma^e}(w_m, a)\} \le r+11$.
    \end{itemize}
    \end{itemize}
    Suppose for contradiction that (i) and (i') hold. By $d_{\Gamma^e}(a,z_0) < d_{\Gamma^e}(a,w_0)$, this implies
    $
         d_{\Gamma^e}(a,b)
         \le d_{\Gamma^e}(a,z_0) + d_{\Gamma^e}(z_0,b)
         < d_{\Gamma^e}(a,w_0) + d_{\Gamma^e}(z_0,b)
         \le (r+9)+(r+18)
         = 2r+27
    $, which contradicts the assumption $d_{\Gamma^e}(a,b) \ge 2r + 32$.

    When (i) and (ii') hold, we have $d_{\Gamma^e}(a,w_m) < d_{\Gamma^e}(a,z_0)$. Indeed, if $d_{\Gamma^e}(a,z_0) \le d_{\Gamma^e}(a,w_m)$, then we have $d_{\Gamma^e}(a,b) \le d_{\Gamma^e}(a,z_0) + d_{\Gamma^e}(z_0,b) \le d_{\Gamma^e}(a,w_m) + d_{\Gamma^e}(z_0,b) \le (r+11)+(r+18)$,  which contradicts $d_{\Gamma^e}(a,b) \ge 2r + 32$. By $d_{\Gamma^e}(a,w_m) < d_{\Gamma^e}(a,z_0) < d_{\Gamma^e}(a,w_0)$, there exists $i \in \{1,\cdots,m\}$ such that $z_0 \in \gamma^{-1}_{[w_{i-1}, w_i]}$. By $p_{[a,z_0]}=\gamma_{[a,z_0]}$, we have $w_m, \cdots, w_i, z_0 \in V(p)$ and $p_{[w_i,z_0]}\cup q_{[y_i,y_{i-1}]} \subset h_i.\Gamma$. Since Remark \ref{rem:intersection of stabilizers} (3) implies $d_{\Gamma^e}(y_{i-1}, w_{i-1}) \le 2$, we also have $d_{\Gamma^e}(y_{i-1}, b) \le d_{\Gamma^e}(y_{i-1}, w_{i-1}) + d_{\Gamma^e}(w_{i-1}, b) \le 2+(r+18)$ by $w_{i-1} \in \gamma_{[z_0, b']}$. Thus, the sequences $(w_m,\cdots,w_i,z_0) \subset V(p)$, $(y_m,\cdots,y_{i-1}) \subset V(q)$, and $(h_m,\cdots,h_i) \subset \Gamma\G$ satisfy the statement.

    When (ii) and (i') hold, we have $d_{\Gamma^e}(a,w_0) < d_{\Gamma^e}(a,z_n)$. Indeed, if $d_{\Gamma^e}(a,z_n) \le d_{\Gamma^e}(a,w_0)$, then $d_{\Gamma^e}(a,b) \le d_{\Gamma^e}(a,z_n) + d_{\Gamma^e}(z_n,b) \le d_{\Gamma^e}(a,w_0) + d_{\Gamma^e}(z_n,b) \le (r+9)+(r+18)$, which contradicts $d_{\Gamma^e}(a,b) \ge 2r + 32$. By $d_{\Gamma^e}(a,z_0) < d_{\Gamma^e}(a,w_0) < d_{\Gamma^e}(a,z_n)$, there exists $i \in \{1,\cdots,n\}$ such that $w_0 \in \gamma_{[z_{i-1}, z_i]}$. By $q_{[w_0, b']}=\gamma_{[w_0, b']}$, we have $w_0, z_i, \cdots, z_n \in V(q)$ and $p_{[x_{i-1},x_i]} \cup q_{[w_0,z_i]} \subset g_i.\Gamma$. Since Remark \ref{rem:intersection of stabilizers} (3) implies $d_{\Gamma^e}(x_{i-1}, z_{i-1}) \le 2$, we also have $d_{\Gamma^e}(x_{i-1}, a) \le d_{\Gamma^e}(x_{i-1}, z_{i-1}) + d_{\Gamma^e}(z_{i-1}, a) \le 2+(r+9)$ by $z_{i-1} \in \gamma_{[a, w_0]}$. Thus, the sequences $(x_{i-1},\cdots,x_n) \subset V(p)$, $(w_0,z_i,\cdots,z_n) \subset V(q)$, and $(g_i,\cdots,g_n) \subset \Gamma\G$ satisfy the statement.

    In the following, assume that (ii) and (ii') hold. We'll discuss three cases (A1)-(A3), (A1) $w_m \in \gamma_{[z_{n-1}, b']}$, (A2) $n \ge 2$ and there exists $i \in \{1,\cdots, n-1\}$ such that $w_m \in \gamma_{[z_{i-1}, z_i]}$, (A3) $w_m \in \gamma_{[a, z_0]}$.
    
    In case (A1), suppose $h_m\neq g_n$ for contradiction. This implies $d_{\Gamma^e}(a,z_n)-2 \le d_{\Gamma^e}(a,w_m)$ by $|\gamma_{[w_{m-1},w_m]}| \ge 7$ and Remark \ref{rem:intersection of stabilizers} (3). Hence, $d_{\Gamma^e}(a,b) \le d_{\Gamma^e}(a,z_n) + d_{\Gamma^e}(z_n,b) \le d_{\Gamma^e}(a,w_m) + 2 + d_{\Gamma^e}(z_n,b) \le (r+11)+2+(r+18)$, which contradicts $d_{\Gamma^e}(a,b) \ge 2r + 32$. By $h_m=g_n$, we have $p_{[x_{n-1}, x_n]} \cup q_{[y_m, y_{m-1}]} \subset g_n.\Gamma$. Note $d_{\Gamma^e}(x_{n-1}, z_{n-1}) \le 2$ by Remark \ref{rem:intersection of stabilizers} (3). Hence, we have $d_{\Gamma^e}(a, x_{n-1}) \le d_{\Gamma^e}(a, z_{n-1}) + d_{\Gamma^e}(z_{n-1}, x_{n-1}) \le d_{\Gamma^e}(a, w_m) + d_{\Gamma^e}(z_{n-1}, x_{n-1}) \le (r+11)+2$ by $z_{n-1} \in \gamma_{[a,w_m]}$. Similarly, by $d_{\Gamma^e}(w_{m-1}, y_{m-1}) \le 2$, we also have $d_{\Gamma^e}(b, y_{m-1})\le d_{\Gamma^e}(b, w_{m-1}) + d_{\Gamma^e}(w_{m-1}, y_{m-1}) \le (r+20)+2$. Here, we used $d_{\Gamma^e}(b, w_{m-1}) \le (r+18)+2$, which follows from $d_{\Gamma^e}(a,z_n)-2 \le d_{\Gamma^e}(a,w_m)$ and $\gamma_{[z_n, b']} \subset \N_{\Gamma^e}(b, r+18)$ in (ii). Recall the conditions $d_{\Gamma^e}(b,x_n) \le r+18$ in (ii) and $d_{\Gamma^e}(a,y_m) \le r+11$ in (ii'). Thus, the sequences $(x_{n-1},x_n) \subset V(p)$, $(y_m, y_{m-1}) \subset V(q)$, and $(g_n) \subset \Gamma\G$ satisfy the statement.
    
    In case (A2), if $h_m = g_i$, then we have $w_{m-1}\in \gamma_{[z_{i-1}, z_{i+1}]}$. Indeed, if $w_{m-1} \notin \gamma_{[z_{i-1}, z_{i+1}]}$, then $\gamma_{[z_i, z_{i+1}]} \subset \gamma_{[w_m, w_{m-1}]}$, hence we have $g_{i+1} = h_m \,(= g_i)$ by $|\gamma_{[z_i, z_{i+1}]}| \ge 7$ and Remark \ref{rem:intersection of stabilizers} (3), which contradicts $g_i \neq g_{i+1}$. 
    
    If $h_m \neq g_i$, then we have $w_{m-1} \in \gamma_{[z_i,b']}$ and $|\gamma_{[w_m, z_i]}| \le 2$ by $|\gamma_{[w_m,w_{m-1}]}| \ge 7$ and Remark \ref{rem:intersection of stabilizers} (3). This implies $|\gamma_{[z_i, w_{m-1}]}| = |\gamma_{[w_m, w_{m-1}]}| - |\gamma_{[w_m, z_i]}| \ge 7-2 = 5$. Hence, we have $h_m = g_{i+1}$ by $|\gamma_{[z_i, z_{i+1}]}| \ge 7$ and Remark \ref{rem:intersection of stabilizers} (3). This and $w_{m-1} \in \gamma_{[z_i,b']}$ implies $w_{m-1} \in \gamma_{[z_i,z_{i+2}]}$. Indeed, if $w_{m-1} \notin \gamma_{[z_i,z_{i+2}]}$, then $\gamma_{[z_{i+1}, z_{i+2}]} \subset \gamma_{[w_m, w_{m-1}]}$, hence we have $g_{i+2} = h_m \,(= g_{i+1})$ by $|\gamma_{[z_{i+1}, z_{i+2}]}| \ge 7$ and Remark \ref{rem:intersection of stabilizers} (3), which contradicts $g_{i+1} \neq g_{i+2}$. 
    
    Thus, in either case, there exists $k$ (i.e. $k=i$ or $k=i+1$) such that $h_m = g_k$, $d_{\Gamma^e}(a,z_{k-1}) \le d_{\Gamma^e}(a,w_m) + 2$, and $w_{m-1} \in \gamma_{[z_{k-1}, z_{k+1}]}$. Note $d_{\Gamma^e}(a,x_{k-1}) \le d_{\Gamma^e}(a,z_{k-1}) + d_{\Gamma^e}(z_{k-1},x_{k-1}) \le d_{\Gamma^e}(a,w_m) + 2 + 2 \le (r+11)+4 = r+15$. 
    
    If $w_{m-1} \in \gamma_{[z_{k-1}, z_k]}$, then by $h_{m-1} \neq h_m = g_k$ we can see $h_{m-1} = g_{k+1}$ and $w_{m-2} \in \gamma_{[z_{k}, z_{k+2}]}$ in the same way as we discussed the case $h_m \neq g_i$ above. 
    
    If $w_{m-1} \in \gamma_{[z_k, z_{k+1}]}$, then we have $|\gamma_{[z_k, w_{m-1}]}| \le 2$ by $g_{k+1} \neq g_k = h_m$ and Remark \ref{rem:intersection of stabilizers} (3). This and $|\gamma_{[w_{m-1}, w_{m-2}]}| \ge 7$ imply  $|\gamma_{[z_k, w_{m-2}]}| \ge 5$. Hence, we have $h_{m-1} = g_{k+1}$ by $|\gamma_{[z_k, z_{k+1}]}| \ge 7$ and Remark \ref{rem:intersection of stabilizers} (3). We can also see $w_{m-2} \in \gamma_{[z_{k}, z_{k+2}]}$ in the same way as above. 
    
    Thus, we have $h_{m-1} = g_{k+1}$ and $w_{m-2} \in \gamma_{[z_{k}, z_{k+2}]}$ in either case. We can repeat this argument up to $h_m,\cdots,h_{m - \min\{m-1,n-k\}}$ and see $h_{m-j}=g_{k+j}$ for any $0 \le j \le \min\{m-1,n-k\}$. 
    
    When $m - 1 > n - k$, we can also see $w_{m-(n-k)-1} \in \gamma_{[z_{n-1}, b']}$. If $w_{m-(n-k)-1} \in \gamma_{[z_{n-1}, z_n]}$, then we have $|\gamma_{[w_{m-(n-k)-1}, z_n]}| \le 2$ by $|\gamma_{[w_{m-(n-k)-1}, w_{m-(n-k)-2}]}| \ge 7$ and Remark \ref{rem:intersection of stabilizers} (3). Indeed, $|\gamma_{[w_{m-(n-k)-1}, z_n]}| \ge 3$ would imply $h_{m-(n-k)-1} = g_n (=h_{m-(n-k)})$, which contradicts $h_{m-(n-k)-1} \neq h_{m-(n-k)}$. Hence, we have $d_{\Gamma^e}(w_{m-(n-k)-1}, b) \le (r+18)+2$ by $\gamma_{[z_n, b']} \subset \N_{\Gamma^e}(b,r+18)$ in (ii). This and $d_{\Gamma^e}(y_{m-(n-k)-1},w_{m-(n-k)-1}) \le 2$ imply $d_{\Gamma^e}(y_{m-(n-k)-1},b) \le d_{\Gamma^e}(y_{m-(n-k)-1},w_{m-(n-k)-1}) + d_{\Gamma^e}(y_{m-(n-k)-1},b) \le r+22$. Recall $d_{\Gamma^e}(b,x_n) \le r+18$ in (ii) and $d_{\Gamma^e}(a,y_m) \le r+11$ in (ii'). Thus, the sequences $(x_{k-1},\cdots,x_n) \subset V(p)$, $(y_m, \cdots, y_{m-(n-k)-1}) \subset V(q)$, and $(g_k,\cdots,g_n) \subset \Gamma\G$ satisfy the statement. 
    
    When $m - 1 = n - k$, we can also see $w_0 \in \gamma_{[z_{n-1}, b']}$. If $w_0 \in \gamma_{[z_n, b']}$, then we have $d_{\Gamma^e}(b, y_0) = d_{\Gamma^e}(b, w_0) \le r + 18$ by (ii), hence the sequences $(x_{k-1},\cdots,x_n) \subset V(p)$, $(y_m, \cdots, y_0) \subset V(q)$, and $(g_k,\cdots,g_n) \subset \Gamma\G$ satisfy the statement. If $w_0 \in \gamma_{[z_{n-1}, z_n]}$, then we have $z_n \in \gamma_{[w_0,b']}=q_{[w_0,b']}=q_{[y_0,b']}$, which implies $|q_{[y_1,z_n]}| \ge |q_{[y_1,y_0]}| \ge 7$ and $p_{[x_{n-1},x_n]} \cup q_{[y_1,z_n]} \subset g_n.\Gamma$ by $\{y_1,z_n\} \subset h_1.\Gamma=g_n.\Gamma$ and Corollary \ref{cor:Gamma is embedded into Gammae}. Hence, the sequences $(x_{k-1},\cdots,x_n) \subset V(p)$, $(y_m, \cdots,y_1, z_n) \subset V(q)$, and $(g_k,\cdots,g_n) \subset \Gamma\G$ satisfy the statement. 
    
    When $m - 1 < n - k$, we can also see $w_0 \in \gamma_{[z_{k+m-2},z_{k+m}]}$ and $w_1 \in \gamma_{[a,z_{k+m-1}]}$. If $w_0 \in \gamma_{[z_{k+m-2},z_{k+m-1}]}$, then we have $z_{k+m-1}, \cdots, z_n \in \gamma_{[w_0,b']}=q_{[w_0,b']}=q_{[y_0,b']}$, which implies $|q_{[y_1,z_{k+m-1}]}| \ge |q_{[y_1,y_0]}| \ge 7$ and $p_{[x_{k+m-2},x_{k+m-1}]} \cup q_{[y_1,z_{k+m-1}]} \subset g_{k+m-1}.\Gamma$ by $\{y_1,z_{k+m-1}\} \subset h_1.\Gamma=g_{k+m-1}.\Gamma$ and Corollary \ref{cor:Gamma is embedded into Gammae}. Hence, the sequences $(x_{k-1},\cdots,x_n) \subset V(p)$, $(y_m, \cdots,y_1, z_{k+m-1}, \cdots, z_n) \subset V(q)$, and $(g_k,\cdots,g_n) \subset \Gamma\G$ satisfy the statement. If $w_0 \in \gamma_{[z_{k+m-1},z_{k+m}]}$, then by $w_1 \in \gamma_{[a,z_{k+m-1}]}$, $g_{k+m} \neq g_{k+m-1}=h_1$, and Remark \ref{rem:intersection of stabilizers} (3), we have $|\gamma_{[z_{k+m-1}, y_0]}| = |\gamma_{[z_{k+m-1}, w_0]}| \le 2$. We also have $z_{k+m}, \cdots, z_n \in \gamma_{[w_0,b']}=q_{[w_0,b']}=q_{[y_0,b']}$. Hence, $|q_{[y_0, z_{k+m}]}| = |\gamma_{[y_0, z_{k+m}]}| = |\gamma_{[z_{k+m-1}, z_{k+m}]}| - |\gamma_{[z_{k+m-1}, w_0]}| \ge 7-2 =5$. Hence, the sequences $(x_{k-1},\cdots,x_n) \subset V(p)$, $(y_m, \cdots,y_0, z_{k+m}, \cdots, z_n) \subset V(q)$, and $(g_k,\cdots,g_n) \subset \Gamma\G$ satisfy the statement.
    
    In case (A3), by $d_{\Gamma^e}(a,z_0) < d_{\Gamma^e}(a,w_0)$, there exists $i \in \{1,\cdots,m\}$ such that $x_0=z_0 \in \gamma_{[w_i, w_{i-1}]}$. Note $w_m,\cdots,w_i \in \gamma_{[a,z_0]} = p_{[a,z_0]}$. 
    
    If $h_i=g_1$, then we have $p_{[w_i,x_1]} \cup q_{[y_i,y_{i-1}]} \subset h_i.\Gamma$ by $\{w_i,x_1\}\subset h_i.\Gamma$ and Corollary \ref{cor:Gamma is embedded into Gammae}. Note that we have $|p_{[w_i,x_1]}| \ge |p_{[x_0,x_1]}| \ge 7$ when $n>1$. In the same way as case (A2), we can show $h_{i-(j-1)} = g_j$ for any $1 \le j \le \min\{i,n\}$. Hence, when either $i>n$ or $[i=n ~\wedge~ y_0=w_0 \in \gamma_{[z_n,b']}]$ holds, the sequences $(w_m, \cdots, w_i, x_1,\cdots,x_n) \subset V(p)$, $(y_m, \cdots,y_{i-n}) \subset V(q)$, and $(h_m,\cdots,h_{i-(n-1)}) \subset \Gamma\G$ satisfy the statement. When $i\le n$ and $y_0=w_0 \in \gamma_{[z_{i-1},z_i]})$, the sequences $(w_m, \cdots, w_i, x_1,\cdots,x_n) \subset V(p)$, $(y_m, \cdots,y_1, z_i, \cdots,z_n) \subset V(q)$, and $(h_m,\cdots,h_1,g_{i+1},\cdots,g_n) \subset \Gamma\G$ satisfy the statement. When $i<n$ and $y_0=w_0 \in \gamma_{[z_i,z_{i+1}]}$, the sequences $(w_m, \cdots, w_i, x_1,\cdots,x_n) \subset V(p)$, $(y_m, \cdots,y_0, z_{i+1}, \cdots,z_n) \subset V(q)$, and $(h_m,\cdots,h_1,g_{i+1},\cdots,n) \subset \Gamma\G$ satisfy the statement.

    If $h_i \neq g_1$, then we have $|\gamma_{[z_0,w_{i-1}]}| \le 2$ by Remark \ref{rem:intersection of stabilizers} (3). Hence, $|\gamma_{[w_i,x_0]}|=|\gamma_{[w_i,z_0]}|=|\gamma_{[w_i,w_{i-1}]}|-|\gamma_{[z_0,w_{i-1}]}| \ge 7-2=5$. Note $p_{[w_i,x_0]} \cup q_{[y_i,y_{i-1}]} \subset h_i.\Gamma$. In the same way as case (A2), when $i>1$, we can show $h_{i-j} = g_j$ for any $1 \le j \le \min\{i-1,n\}$. Hence, when either $i-1>n$ or $[i-1=n ~\wedge~ y_0=w_0 \in \gamma_{[z_n,b']}]$ holds, the sequences $(w_m, \cdots, w_i, x_0,\cdots,x_n) \subset V(p)$, $(y_m, \cdots,y_{i-n-1}) \subset V(q)$, and $(h_m,\cdots,h_{i-n}) \subset \Gamma\G$ satisfy the statement. When $i-1 \le n$ and $y_0=w_0 \in \gamma_{[z_{i-2},z_{i-1}]})$, the sequences $(w_m, \cdots, w_i, x_0,\cdots,x_n) \subset V(p)$, $(y_m, \cdots,y_1, z_{i-1}, \cdots,z_n) \subset V(q)$, and $(h_m,\cdots,h_1,g_i,\cdots,g_n) \subset \Gamma\G$ satisfy the statement. When $i-1<n$ and $y_0=w_0 \in \gamma_{[z_{i-1},z_i]}$, the sequences $(w_m, \cdots, w_i, x_0,\cdots,x_n) \subset V(p)$, $(y_m, \cdots,y_0, z_i, \cdots,z_n) \subset V(q)$, and $(h_m,\cdots,h_1,g_i,\cdots,g_n) \subset \Gamma\G$ satisfy the statement.
\end{proof}

\section{Asymptotic dimension of the extension graph}
\label{sec:Asymptotic dimension of the extension graph}

The goal of this section is to prove Theorem \ref{thm:intro asymptotic dimension}, which corresponds to Theorem \ref{thm:asymptotic dimension}. For the rest of Section \ref{sec:Asymptotic dimension of the extension graph}, suppose that $\Gamma$ is a connected simplicial graph with $\girth(\Gamma) > 20$ and $\{G_v\}_{v \in V(\Gamma)}$ is a collection of non-trivial groups.

In order to prove Theorem \ref{thm:asymptotic dimension}, we first study the relation between geodesic paths in $\hGammae$ and those in $\Gamma^e$ in Lemma \ref{lem:geodesic decomposition for hGammae exists} (see Definition \ref{def:coned-off extension graph} for $\hGammae$). It turns out that a geodesic path in $\hGammae$ can be obtained by decomposing a geodesic path in $\Gamma^e$.

\begin{figure}[htbp]
\begin{center}
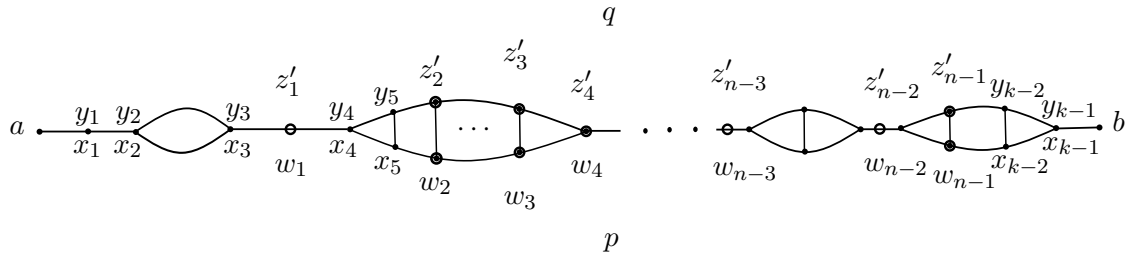
 \caption{Paths $p$ and $q$ in the proof of Lemma \ref{lem:geodesic decomposition for hGammae exists}} 
 \label{fig:}
\end{center} 
\end{figure}

\begin{lem}\label{lem:geodesic decomposition for hGammae exists}
    Let $a,b \in V(\Gamma^e)$. For any $p \in \geo_{\Gamma^e} (a,b)$ and any $\alpha \in \geo_{\hGammae} (a,b)$, there exists a subsequence $\widetilde{p}=(x_0,\cdots,x_n)$ of $V(p)$ with $x_0=a$ and $x_n=b$ such that $\widetilde{p}$ is a geodesic path in $\hGammae$ from $a$ to $b$ (i.e. $n = d_{\hGammae}(a,b)$ and $\forall\, i \ge 1, d_{\hGammae}(x_{i-1},x_i)=1$) and the Hausdorff distance of $\alpha$ and $\widetilde{p}$ in $\hGammae$ is at most 1.
\end{lem}

\begin{proof}
    We assume $a \neq b$ since the case $a = b$ is trivial. Let $\alpha=(y_0,\cdots,y_n) \in \geo_{\hGammae}(a,b)$. Note $n = d_{\hGammae}(a,b)$, $y_0=a$, and $y_n=b$. For each $j\in \{0,\cdots,n\}$, define $\mathcal{A}_j$ by $\mathcal{A}_j=\{g \in \Gamma\G \mid y_j \in g.\Gamma \}$. Define $\mathcal{B}$ to be the set of all tuples $(q, \mathbf{z}, \mathbf{g})$, where $q$ is a path in $\Gamma^e$ from $a$ to $b$, $\mathbf{z} = (z_0,\cdots,z_n)$ is a subsequence of $V(q)$ with $z_0=a$ and $z_n=b$, and $\mathbf{g}=(g_1,\cdots,g_n)$ is a sequence in $\bigcup_{j=0}^n \mathcal{A}_j$ such that $\forall\, i \in \{1,\cdots,n\},\, q_{[z_{i-1},z_i]} \subset g_i.\Gamma$ and $\forall\, j \in\{0,\cdots,n\}, A_j\cap \{g_1,\cdots,g_n\} \neq \emptyset$. The set $\mathcal{B}$ is non-empty. Indeed, for each $i \ge 1$ take $q'_i \in \geo_{\Gamma^e}(y_{i-1}, y_i)$ and $g'_i\in \Gamma\G$ with $q'_i\subset g'_i.\Gamma$, which is possible by $d_{\hGammae}(y_{i-1}, y_i)=1$ and Corollary \ref{cor:Gamma is embedded into Gammae}, then we have $(q'_1\cdots q'_n, (y_0, \cdots, y_n), (g'_1, \cdots, g'_n)) \in \mathcal{B}$. Take $(q, \mathbf{z}, \mathbf{g}) \in \mathcal{B}$ such that $|q| = \min\{|q'| \in \NN\cup\{0\} \mid (q', \mathbf{z'}, \mathbf{g'}) \in \mathcal{B}\}$. By $n=d_{\hGammae}(a,b)$ and minimality of $|q|$, the subpath $q_{[z_{i-1},z_i]}$ is geodesic in $\Gamma^e$ for any $i \ge 1$ and the path $q$ has no self-intersection. By $n = d_{\hGammae}(a,b)$, for any $i\in\{1,\cdots,n-1\}$, no $h \in \Gamma\G$ satisfies $q_{[z_{i-1},z_{i+1}]} \subset h.\Gamma$. Hence, $q$ is admissible with respect to $\mathbf{z}$.

    We claim that $q$ is straight (see Definition \ref{def:straight}). Indeed, let $r$ be a subpath of $q$ such that $r \subset h.\Gamma$ for some $h \in \Gamma\G$. If $|r| \le 6$, then $r$ is geodesic in $\Gamma^e$ by $\girth(\Gamma^e) > 20$ and Corollary \ref{cor:Gamma is embedded into Gammae}. Hence, we assume $|r| > 6$. By Remark \ref{rem:subpath for straightness}, there exists $i \in \{1, \cdots, n-2\}$ such that $r$ is a subpath of $q_{[z_{i-1}, z_{i+2}]}$. Hence, if $h \notin \{g_i, g_{i+1}, g_{i+2}\}$, then we have $|r| \le \sum_{k = i}^{i+2} \diam_{\Gamma^e}(g_{k}.\Gamma \cap h.\Gamma) \le 6$ by Remark \ref{rem:intersection of stabilizers} (3), which contradicts our assumption $|r| > 6$. Hence, there exists $i_0 \in \{i,i+1,i+2\}$ such that $h = g_{i_0}$. By $n = d_{\hGammae}(a,b)$ and Corollary \ref{cor:Gamma is embedded into Gammae}, we can see that $q \cap g_{i_0}.\Gamma$ is a subpath of $q$ satisfying $q_{[z_{i_0 - 1}, z_{i_0}]} \subset q \cap g_{i_0}.\Gamma \subset q_{[z_{i_0 - 2}, z_{i_0 + 1}]}$. Define $w_-, w_+ \in V(q)$ by $w_- = (q \cap g_{i_0}.\Gamma)_-$ and $w_+ = (q \cap g_{i_0}.\Gamma)_+$ for brevity. Since we have $(q, (z_0,\cdots,z_{i_0-2},w_-,w_+,z_{i_0+1},\cdots,z_n), \mathbf{g}) \in \mathcal{B}$, the subpath $q \cap g_{i_0}.\Gamma$ is geodesic in $\Gamma^e$ by minimality of $|q|$. Hence, $r$ is geodesic in $\Gamma^e$ since $r$ is a subpath of $q \cap g_{i_0}.\Gamma$ by $r \subset h.\Gamma = g_{i_0}.\Gamma$. Thus, $q$ is straight.

    By this and Remark \ref{rem:geodesic is a straight admissible path}, both $p$ and $q$ are straight admissible paths without self-intersection. Hence, by Lemma \ref{lem:subpath of straight admissible path} and Proposition \ref{prop:paths must penetrate long}, there exist $k\in\NN$, a subsequence $(a=)x_0,\cdots,x_k(=b)$ of $V(p)$, and a subsequence $(a=)y_0,\cdots,y_k(=b)$ of $V(q)$ such that for any $i \in \{1,\cdots,k\}$, either (A1) or (A2) holds.
    \begin{itemize}
        \item[(A1)]
        $p_{[x_{i-1},x_i]} = q_{[y_{i-1},y_i]}$.
        \item[(A2)]
        $\min\{p_{[x_{i-1},x_i]}, q_{[y_{i-1},y_i]}\} \ge 7$ and $p_{[x_{i-1},x_i]} \cup q_{[y_{i-1},y_i]} \subset h.\Gamma$ with some $h \in \Gamma\G$.
    \end{itemize}
    Note that the subpaths $p_{[x_{i-1},x_i]}$ and $q_{[y_{i-1},y_i]}$ are geodesic in $\Gamma^e$ for any $i \ge 1$ since $p$ and $q$ are straight. We will show that there exists a subsequence $\mathbf{z'} = (z'_0,\cdots,z'_n)$ of $V(q)$ with $z'_0 = a$ and $z'_n = b$ such that $(q, \mathbf{z'}, \mathbf{g}) \in \mathcal{B}$ and for any $j \in \{0,\cdots,n\}$, either (B1) or (B2) holds.
    \begin{itemize}
        \item[(B1)]
        There exists $i \in \{1,\cdots,k\}$ satisfying (A1) and $z'_j \in q_{[y_{i-1},y_i]}$.
        \item[(B2)]
        There exists $i \in \{0,\cdots,k\}$ such that $z'_j = y_i$.
    \end{itemize}
    We define $z'_j$ inductively on $j$. Note that $z'_0$ defined by $z'_0 = a$ satisfies (B2) by $z'_0=a=y_0$. Assume that $z'_0, \cdots, z'_{j-1}$ with $j \in \{1,\cdots,n\}$ were defined so that (B1) or (B2) holds for any $z'_\ell$ with $\ell \in \{0,\cdots,j-1\}$ and $\mathbf{z'_{j-1}} = (z'_0, \cdots, z'_{j-1},z_j, \cdots, z'_n)$ is a subsequence of $V(q)$ satisfying $(q,\mathbf{z'_{j-1}},\mathbf{g}) \in \mathcal{B}$. Note that $z'_0, \cdots, z'_{j-1},z_j, \cdots, z'_n$ are all distinct by $n = d_{\hGammae}(a,b)$. There exists $i \in \{1,\cdots,k\}$ such that $z_j \in q_{[y_{i-1},y_i]}$.
    
    When $i$ satisfies (A1), define $z'_j$ by $z'_j = z_j$. 

    When $i$ doesn't satisfy (A1), $i$ satisfies (A2). Let $h \in \Gamma\G$ satisfy $p_{[x_{i-1},x_i]} \cup q_{[y_{i-1},y_i]} \subset h.\Gamma$. Since $z'_{j-1}$ satisfies one of (B1) or (B2) and $i$ doesn't satisfy (A1), we have $z'_{j-1} \in q_{[a,y_{i-1}]}$. We'll discuss two cases, (i) $z_{j+1} \in q_{[y_i,b]}$ and (ii) $z_{j+1} \in q_{[y_{i-1}, y_i]}$. 
    
    (i) When $z_{j+1} \in q_{[y_i,b]}$, we have $h \in \{g_j, g_{j+1}\}$. Indeed, if $h \notin \{g_j, g_{j+1}\}$, then $|q_{[y_{i-1},y_i]}| = |q_{[y_{i-1}, z_j]}| + |q_{[z_j, y_i]}| \le 2+2$ by Remark \ref{rem:intersection of stabilizers} (3), which contradicts the condition $|q_{[y_{i-1},y_i]}| \ge 7$ in (A2). Here, we used $q_{[y_{i-1}, z_j]} \subset h.\Gamma \cap g_j.\Gamma$ and $q_{[z_j, y_i]} \subset h.\Gamma \cap g_{j+1}.\Gamma$, which follow from $(q,\mathbf{z'_{j-1}},\mathbf{g}) \in \mathcal{B}$. Define $z'_j$ by $z'_j = y_i$ if $h = g_j$, and by $z'_j = y_{i-1}$ if $h = g_{j+1}$. Then, $z'_j$ satisfies (B2) and $\mathbf{z'_j} = (z'_0, \cdots, z'_j, z_{j+1}, \cdots, z'_n)$ is a subsequence of $V(q)$ satisfying $(q,\mathbf{z'_j},\mathbf{g}) \in \mathcal{B}$. 
    
    (ii) When $z_{j+1} \in q_{[y_{i-1}, y_i]}$, we have $h \notin \{g_j,g_{j+2}\}$. Indeed, if $h = g_j$, then by $q_{[z_j, z_{j+1}]} \subset q_{[y_{i-1},y_i]} \subset h.\Gamma = g_j.\Gamma$, we have $q_{[z'_{j-1}, z_{j+1}]} \subset g_j.\Gamma$, which contradicts $d_{\hGammae}(z'_{j-1}, z_{j+1}) = 2$. Similarly, if $h = g_{j+2}$, then $q_{[z_j, z_{j+2}]} \subset g_{j+2}.\Gamma$, which contradicts $d_{\hGammae}(z_j, z_{j+2}) = 2$. Also, we have $z_{j+2} \in q_{[y_i, b]}$. Indeed, if $z_{j+2} \in q_{[y_{i-1}, y_i]}$, then $q_{[z_j,z_{j+2}]} \subset h.\Gamma$, which contradicts $d_{\hGammae}(z_j, z_{j+2}) = 2$. Hence, we have $|q_{[y_{i-1}, z_j]}| \le \diam_{\Gamma^e}(h.\Gamma \cap g_i.\Gamma) \le 2$ and $|q_{[z_{j+1}, y_i]}| \le \diam_{\Gamma^e}(h.\Gamma \cap g_{i+2}.\Gamma) \le 2$ by Remark \ref{rem:intersection of stabilizers} (3). This and $|q_{[y_{i-1},y_i]}| \ge 7$ imply $|q_{[z_j, z_{j+1}]}| \ge 3$. Hence, $h = g_{j+1}$. Define $z'_j$ by $z'_j = y_{i-1}$, then $z'_j$ satisfies (B2) and $\mathbf{z'_j} = (z'_0, \cdots, z'_j, z_{j+1}, \cdots, z'_n)$ is a subsequence of $V(q)$ satisfying $(q,\mathbf{z'_j},\mathbf{g}) \in \mathcal{B}$.

    Thus, we've shown the existence of $\mathbf{z'}$ above. For each $j \in \{0,\cdots,n\}$, define $w_j \in V(p)$ by $w_j = z'_j$ if $z'_j$ satisfies (B1) and by $w_j = y_i$ if $z'_j$ satisfies (B2) with $z'_j = x_i$, where $i \in \{0,\cdots,k\}$. This is well-defined, that is, $w_j$ becomes the same vertex when $z'_j$ satisfies both (B1) and (B2). This is because if $i$ satisfies (A1), then we have $x_{i-1} = y_{i-1}$ and $x_i = y_i$. We can see that $\widetilde{p} = (w_0, \cdots, w_n)$ is a subsequence of $V(p)$.

    Let $j \in \{1,\cdots,n\}$. We'll show $p_{[w_{j-1}, w_j]} \subset g_j.\Gamma$. When $i \in \{1,\cdots,k\}$ satisfies (A1) and $p_{[x_{i-1}, x_i]} \subset p_{[w_{j-1}, w_j]}$, we have $p_{[x_{i-1}, x_i]} = q_{[y_{i-1}, y_i]} \subset q_{[z'_{j-1}, z'_j]} \subset g_j.\Gamma$. Similarly, when $w_{j-1} \in p_{[x_{i-1}, x_i]}$ with some $i$ satisfying (A1), we have $p_{[x_{i-1}, x_i]} \cap p_{[w_{j-1}, w_j]} = q_{[y_{i-1}, y_i]} \cap q_{[z'_{j-1}, z'_j]} \subset g_j.\Gamma$. The same argument holds when $w_j \in p_{[x_{i-1}, x_i]}$ with some $i$ satisfying (A1). When $i \in \{1,\cdots,k\}$ satisfies (A2) and $p_{[x_{i-1}, x_i]} \subset p_{[w_{j-1}, w_j]}$, we have $q_{[y_{i-1}, y_i]} \subset q_{[z'_{j-1}, z'_j]} \subset g_j.\Gamma$ by $p_{[x_{i-1}, x_i]} \subset p_{[w_{j-1}, w_j]}$. By (A2), there exists $h \in \Gamma\G$ such that $p_{[x_{i-1},x_i]} \cup q_{[y_{i-1},y_i]} \subset h.\Gamma$. This and Remark \ref{rem:intersection of stabilizers} (3) imply $h = g_j$ since we have $|q_{[y_{i-1},y_i]}| \ge 7$ by (A2). Hence, $p_{[x_{i-1},x_i]} \subset h.\Gamma = g_j.\Gamma$. Thus, $p_{[w_{j-1}, w_j]} \subset g_j.\Gamma$ for any $j \in \{1,\cdots,n\}$.

    This implies $(p, \widetilde{p}, \mathbf{g}) \in \mathcal{B}$. Hence, $\widetilde{p}$ is a geodesic path in $\hGammae$ from $a$ to $b$ by $n = d_{\hGammae}(a,b)$ and the Hausdorff distance of $\alpha$ and $\widetilde{p}$ in $\hGammae$ is at most 1 by the definition of $\mathcal{B}$.
\end{proof}

We are now ready to prove Theorem \ref{thm:asymptotic dimension}. Before this, we introduce the notion of a geodesic spanning tree and Lemma \ref{lem:geodesic spanning tree exists}.

\begin{defn}
    Let $X$ be a connected graph. A subgraph $T$ of $X$ is called a \emph{geodesic spanning tree of} $X$ \emph{rooted at} $x \in V(X)$ if $T$ is a tree with $V(T)=V(X)$ and satisfies $d_X(x,y)=d_T(x,y)$ for any $y \in V(X)$.
\end{defn}

\begin{lem}\label{lem:geodesic spanning tree exists}
    For any connected graph $X$ and $x \in V(X)$, a geodesic spanning tree of $X$ rooted at $x$ exists.
\end{lem}

\begin{proof}
    Define $\mathcal{T}$ to be the set of all subgraphs $T$ of $X$ such that $T$ is a tree with $x \in V(T)$ and satisfies $d_X(x,y)=d_T(x,y)$ for any $y \in V(T)$. The set $\mathcal{T}$ is nonempty by $\{x\} \in \mathcal{T}$. Define the order on $\mathcal{T}$ by inclusion. Since every chain in $\mathcal{T}$ has an upper bound by taking their union, there exists a maximal element $T_0 \in \mathcal{T}$ by Zorn's lemma. Suppose for contradiction that there exists $y \in V(X)\setminus V(T_0)$. Take a geodesic $p$ in $X$ from $x$ to $y$ and take $z \in V(p)$ satisfying $z \in V(T_0)$ and $d_X(x,z)=\max\{d_X(x,z') \mid z' \in V(T_0)\cap V(p)\}$. We can show that the subgraph $T_1$ of $X$ defined by $V(T_1) = V(T_0) \cup V(p_{[z,y]})$ and $E(T_1) = E(T_0) \cup E(p_{[z,y]})$ satisfies $T_1 \in \mathcal{T}$ and $T_0 \subsetneq T_1$. This contradicts maximality of $T_0$. Hence, $V(T_0)=V(X)$. Thus, $T_0$ is a geodesic spanning tree rooted at $x$.
\end{proof}

We now prove Theorem \ref{thm:asymptotic dimension} using Corollary \ref{cor:coned off extension graph is quasitree} and the natural contraction $\Gamma^e \to \hGammae$. See Section \ref{subsec:Asymptotic dimension of metric spaces} for relevant notions. Lemma \ref{lem:uniformly asdim} is an auxiliary lemma for Theorem \ref{thm:asymptotic dimension}.

\begin{lem}\label{lem:uniformly asdim}
    Let $(S_\alpha, d_\alpha)_{\alpha \in \A}$ be a family of metric spaces and $(A_\alpha)_{\alpha \in \A}$ be a family of subsets of $(S_\alpha)_{\alpha \in \A}$ (i.e. $\forall\,\alpha \in \A,\,A_\alpha \subset S_\alpha$) satisfying $\asdim \le n$ uniformly with $n \in \NN\cup\{0\}$. Then, the following hold.
    \begin{itemize}
        \item[(1)]
        For any $t \ge 0$, the family $\big(\N_{S_\alpha}(A_\alpha,t)\big)_{\alpha \in \A}$ satisfies $\asdim \le n$ uniformly.
        
        \item[(2)] 
        Any family $(B_\alpha)_{\alpha \in \A}$ of subsets of $(A_\alpha)_{\alpha \in \A}$ satisfies $\asdim \le n$ uniformly.
    \end{itemize}
\end{lem}

\begin{proof}
    (1) Let $r > 0$. Since $(A_\alpha)_{\alpha \in \A}$ satisfies $\asdim \le n$ uniformly, there exist $\mathcal{U}^0_\alpha, \cdots, \mathcal{U}^n_\alpha \subset 2^{A_\alpha}$ for each $\alpha \in \A$ satisfying the three conditions of Definition \ref{def:uniformly asdim} for $r+2t$. Note $\inf\{d_{\alpha}(U,V) \mid U,V \in \mathcal{U}^i_\alpha, U\neq V \}>r+2t$ for any $i\in \{0,\cdots,n\}$ and $\alpha \in \A$. The family $\mathcal{V}^0_\alpha, \cdots, \mathcal{V}^n_\alpha \subset 2^{\N_{S_\alpha}(A_\alpha,t)}$ defined by $\mathcal{V}^i_\alpha = \big\{\N_{S_\alpha}(U,t) \in 2^{\N_{S_\alpha}(A_\alpha,t)} \mid U \in\mathcal{U}^i_\alpha \big\}$ for each $i$ and $\alpha$ satisfies the three conditions of Definition \ref{def:uniformly asdim} for $r$. Hence,$\big(\N_{S_\alpha}(A_\alpha,t)\big)_{\alpha \in \A}$ satisfies $\asdim \le n$ uniformly.

    (2) Let $r > 0$. Since $(A_\alpha)_{\alpha \in \A}$ satisfies $\asdim \le n$ uniformly, there exist $\mathcal{U}^0_\alpha, \cdots, \mathcal{U}^n_\alpha \subset 2^{A_\alpha}$ for each $\alpha \in \A$ satisfying the three conditions of Definition \ref{def:uniformly asdim} for $r$. The family $\mathcal{V}^0_\alpha, \cdots, \mathcal{V}^n_\alpha \subset 2^{B_\alpha}$ defined by $\mathcal{V}^i_\alpha = \big\{U \cap B_\alpha \in 2^{B_\alpha} \mid U \in\mathcal{U}^i_\alpha \big\}$ for each $i$ and $\alpha$ satisfies the three conditions of Definition \ref{def:uniformly asdim} for $r$. Hence, $(B_\alpha)_{\alpha \in \A}$ satisfies $\asdim \le n$ uniformly.
\end{proof}

\begin{thm}\label{thm:asymptotic dimension}
    Suppose that $\Gamma$ is a connected simplicial graph with $\girth(\Gamma) > 20$ and that $\{G_v\}_{v \in V(\Gamma)}$ is a collection of non-trivial groups. If $\asdim(\Gamma) \le n$ with $n \in \NN \cup \{0\}$, then $\asdim(\Gamma^e) \le n+1$.
\end{thm}

\begin{proof}
    Define the graph homomorphism $f \colon \Gamma^e \to \hGammae$ by $f(x)=x$ for any $x \in V(\Gamma^e)$. We claim that for any $R \in \NN\cup\{0\}$, the family $\big(\, f^{-1}(\N_{\hGammae}(o,R)) \,\big)_{o \in V(\Gamma^e)}$ satisfies $\asdim \le n$ uniformly. We show this claim by induction on $R$.
    
    When $R=0$, we have $f^{-1}(\N_{\hGammae}(o,0))=\{o\}$ for any $o \in V(\Gamma^e)$ and the family $(\{o\})_{o \in V(\Gamma^e)}$ satisfies $\asdim \le 0 \,(\le n)$ uniformly. Hence, the claim holds for $R=0$. 
    
    Next, assume that the family $\big( \, f^{-1}(\N_{\hGammae}(o,R)) \,\big)_{o \in V(\Gamma^e)}$ satisfies $\asdim \le n$ uniformly for $R \in \NN\cup\{0\}$ and we'll show the claim for $R+1$. For each $o \in V(\Gamma)$, take a geodesic spanning tree $T_o$ of $\Gamma^e$ rooted at $o$, which exists by Lemma \ref{lem:geodesic spanning tree exists}. For each $(o,x) \in V(\Gamma^e)^2$, let $T(o,x)$ be the unique geodesic in $T_o$ from $o$ to $x$. Note $T(o,x) \in \geo_{\Gamma^e}(o,x)$. Also, fix a subsequence $p(o,x)=(p_0(o,x),\cdots,p_m(o,x))$ of $V(T(o,x))$ with $p_0(o,x)=o$ and $p_m(o,x)=x$ such that $p(o,x)$ is a geodesic in $\hGammae$ from $o$ to $x$ (i.e. $m = d_{\hGammae}(o,x)$ and $\forall\, i \ge 1, d_{\hGammae}(p_{i-1}(o,x),p_i(o,x))=1$), which exists by Lemma \ref{lem:geodesic decomposition for hGammae exists}. Since $p(o,x)$ is geodesic in $\hGammae$, the geodesic path $T(o,x)$ in $\Gamma^e$ is admissible with respect to $p(o,x)$. Fix $g(o,x)=(g_1(o,x),\cdots,g_m(o,x)) \in \A_0(T(o,x), p(o,x))$ for each $(o,x) \in V(\Gamma^e)^2$ (see Definition \ref{def: admissible decomposition}).
    
    For each $(k,g) \in (\NN\cup\{0\})\times \Gamma\G$ and $o \in V(\Gamma^e)$, define $Q_o(k,g) \subset V(\Gamma^e)$ by
    \[
    Q_o(k,g)=\{ x \in V(\Gamma^e) \mid \text{$d_{\hGammae}(o,x)=k$ and $g_1(o,x)\cdots g_k(o,x)=g$} \}.
    \]
    Define $X_o \subset V(\Gamma^e)$ and $A_o \subset 2^{X_o}$ by 
    \begin{align*}
     X_o &= f^{-1}(\N_{\hGammae}(o,R+1)), \\
     A_o &= \{Q_o(k,g) \in 2^{X_o} \mid 0 \le k \le R+1, g \in \Gamma\G \}.
    \end{align*}
    We can see $X_o =\bigcup_{0 \le k \le R+1, g \in \Gamma\G} Q_o(k,g) \,(=\! \bigcup_{U \in A_o}U)$. For any $o \in V(\Gamma^e)$ and $(k,g) \in (\NN\cup\{0\})\times \Gamma\G$, we have $Q_o(k,g) \subset g.\Gamma$ and $g.\Gamma$ is isometric to $\Gamma$. Hence, the family $\bigcup_{o \in V(\Gamma^e)}A_o$ satisfies $\asdim \le n$ uniformly by $\asdim(\Gamma) \le n$.

    Let $r \in \NN$. For each $o \in V(\Gamma^e)$, define $Y_{r,o}$ by $Y_{r,o} = \N_{\Gamma^e}(f^{-1}(\N_{\hGammae}(o,R)), 3r+12) \cap X_o$. Since the family $\big( \, f^{-1}(\N_{\hGammae}(o,R)) \, \big)_{o \in V(\Gamma^e)}$ satisfies $\asdim \le n$ uniformly by our assumption of induction, the family $(Y_{r,o})_{o\in V(\Gamma^e)}$ also satisfies $\asdim \le n$ uniformly by Lemma \ref{lem:uniformly asdim}. (Here, we first applied Lemma \ref{lem:uniformly asdim} (1) with $\A=V(\Gamma^e)$, $S_\alpha \equiv \Gamma^e$, and $t=3r+12$, and then applied Lemma \ref{lem:uniformly asdim} (2) by using $Y_{r,o} \subset \N_{\Gamma^e}(f^{-1}(\N_{\hGammae}(o,R)), 3r+12)$.) 
    
    Let $o \in V(\Gamma^e)$. Suppose that $(k,g), (\ell,h) \in  \{0, \cdots, R+1\} \times \Gamma\G$ satisfy $Q_o(k,g) \neq Q_o(\ell,h)$ and that we have $a \in Q_o(k,g) \setminus Y_{r,o}$ and $b \in Q_o(\ell,h) \setminus Y_{r,o}$. By $\{a,b\} \subset X_o \setminus Y_{r,o}$, we have $d_{\hGammae}(o,a)=d_{\hGammae}(o,b)=R+1$, hence $k=\ell=R+1$. By this and $Q_o(k,g) \neq Q_o(\ell,h)$, we have $g \neq h$. By $a,b\notin \N_{\Gamma^e}(f^{-1}(\N_{\hGammae}(o,R)), 3r+12)$, we also have $\min\{d_{\Gamma^e}(a,p_R(o,a)), d_{\Gamma^e}(b,p_R(o,b))\} > 3r+12$.
    
    Suppose $d_{\Gamma^e}(a,b) \le r$ for contradiction. Since $T$ is a geodesic spanning tree, $T(o,a)$ and $T(o,b)$ form a tripod, that is, there exists $y \in V(T(o,a)) \cap V(T(o,b))$ such that $T(o,a)_{[o,y]}=T(o,b)_{[o,y]}$ and $V(T(o,a)_{[y,a]})\cap V(T(o,b)_{[y,b]}) = \{y\}$. Fix $\alpha \in \geo_{\Gamma^e}(a,b)$. By Proposition \ref{prop: classification of geodesic triangle}, there exist $N \in \NN\cup\{0\}$, a subsequence $(y =)\, z_0,\cdots z_N,z_{N+1}$ of $V(T(o,a)_{[y,a]})$, a subsequence $(y =)\, w_0,\cdots w_N,w_{N+1}$ of $V(T(o,b)_{[y,b]})$, a sequence $h_1,\cdots,h_{N+1}$ in $\Gamma\G$, and $a',b' \in V(\alpha)$ that satisfy the four conditions (i)-(iv) below.
        \begin{itemize}
            \item[(i)]
            $T(o,a)_{[z_{i-1},z_i]} \cup T(o,b)_{[w_{i-1},w_i]} \subset h_i.\Gamma$ and $h_i \neq h_{i+1}$ for any $i \in \{1,\cdots,N\}$.
            \item[(ii)]
            $\alpha_{[a',b']} \cup T(o,a)_{[z_N,z_{N+1}]} \cup T(o,b)_{[w_N,w_{N+1}]} \subset h_{N+1}.\Gamma$. 
            \item[(iii)]
            $\min\{|T(o,a)_{[z_{i-1},z_i]}|, |T(o,b)_{[w_{i-1},w_i]}|\} \ge 7$ for any $i \in \{1,\cdots,N\}$.
            \item[(iv)]
            $\max\{d_{\Gamma^e}(z_{N+1}, a'), d_{\Gamma^e}(w_{N+1}, b')\} \le 2$
        \end{itemize}
        By $d_{\Gamma^e}(z_{N+1}, a') \le 2$ and the assumption $d_{\Gamma^e}(a,b) \le r$, we have $d_{\Gamma^e}(a,z_{N+1}) \le r+2$. Similarly, we also have $d_{\Gamma^e}(b,w_{N+1}) \le r+2$. We'll discuss three cases, (A1) when $|T(o,a)_{[z_N, z_{N+1}]}| > 2r + 7$, (A2) when $|T(o,b)_{[w_N, w_{N+1}]}| > 2r + 7$, (A3) when $\max\{ |T(o,a)_{[z_N, z_{N+1}]}|, |T(o,b)_{[w_N, w_{N+1}]}| \} \le 2r + 7$.
        
        In case (A1), the path $T(o,a)_{[p_R(o,a), a]}$ contains a subpath $q$ of $T(o,a)_{[z_N, z_{N+1}]}$ with $|q|\ge 3$ by $d_{\Gamma^e}(a,p_R(o,a)) > 3r+12$ and $d_{\Gamma^e}(a,z_{N+1}) \le r+2$. Hence, we have $h_{N+1}=g_1(o,a)\cdots g_{R+1}(o,a)=g$ by $a \in Q_o(R+1,g)$ and Remark \ref{rem:intersection of stabilizers} (3). On the other hand, by $d_{\Gamma^e}(a,z_N) \ge |T(o,a)_{[z_N, z_{N+1}]}| > 2r + 7$, we have
        \[
        d_{\Gamma^e}(b,w_N) \ge d_{\Gamma^e}(a,z_N) - d_{\Gamma^e}(z_N,w_N) - d_{\Gamma^e}(a,b) > (2r+7)-2-r = r+5.
        \]
        This and $d_{\Gamma^e}(b,w_{N+1}) \le r+2$ imply $d_{\Gamma^e}(w_N,w_{N+1}) > (r+5)-(r+2)=3$. Hence, we can see $h_{N+1}=g_1(o,b)\cdots g_{R+1}(o,b)=h$ in the same way as $T(o,a)$ by using $d_{\Gamma^e}(b,p_R(o,b)) > 3r+12$ and $d_{\Gamma^e}(b,w_{N+1}) \le r+2$. Hence, we have $g=h=h_{N+1}$, which contradicts $g \neq h$. 
        
        In case (A2), we can see $g=h=h_{N+1}$ in the sane way as case (A1), hence get contradiction.

        In case (A3), by $d_{\Gamma^e}(a,p_R(o,a)) > 3r+12$ and $d_{\Gamma^e}(a,z_{N+1}) \le r+2$, we have $d_{\Gamma^e}(a,z_N) < d_{\Gamma^e}(a,p_R(o,a))$ and $d_{\Gamma^e}(z_N, p_R(o,a)) > (3r+12) - (r+2) - (2r+7) = 3$. Similarly, we can also see $d_{\Gamma^e}(b,w_N) < d_{\Gamma^e}(b,p_R(o,b))$ and $d_{\Gamma^e}(w_N, p_R(o,b)) > 3$ by $d_{\Gamma^e}(b,p_R(o,b)) > 3r+12$ and $d_{\Gamma^e}(b,w_{N+1}) \le r+2$.
        
        When $N \ge 1$, by $d_{\Gamma^e}(z_N, p_R(o,a)) >3$ and $|T(o,a)_{[z_{N-1},z_N]}| \ge 7$, the path $T(o,a)_{[p_R(o,a), a]}$ contains a subpath $q$ of $T(o,a)_{[z_{N-1}, z_N]}$ with $|q|\ge 3$. Hence, we have $h_N=g_1(o,a)\cdots g_{R+1}(o,a)=g$ by $a \in Q_o(R+1,g)$ and Remark \ref{rem:intersection of stabilizers} (3). Similarly, we also have $h_N=g_1(o,b)\cdots g_{R+1}(o,b)=h$. Hence, we have $g=h=h_N$, which contradicts $g \neq h$.
        
        When $N=0$, we have $T(o,a)_{[p_R(o,a), z_N]}=T(o,a)_{[p_R(o,a), y]}$ and $T(o,b)_{[p_R(o,b), w_N]}=T(o,b)_{[p_R(o,b), y]}$. By this and $T(o,a)_{[o,y]}=T(o,b)_{[o,y]}$, a subpath $q$ of $T(o,a)_{[o,y]} \,(=T(o,b)_{[o,y]})$ with $|q|\ge 3$ is contained in both $T(o,a)_{[p_R(o,a), a]}$ and $T(o,a)_{[p_R(o,a), a]}$. Hence, we have $g=h$ by Remark \ref{rem:intersection of stabilizers} (3), which contradicts $g \neq h$.

        Thus, we've shown $\inf\{d_{\Gamma^e}(Q_o(k,g) \setminus Y_{r,o}, Q_o(\ell,h) \setminus Y_{r,o}) \mid (k,g), (\ell,h) \in \{0,\cdots,R+1\}\times \Gamma\G, Q_o(k,g)\neq Q_o(\ell,h) \} \ge r$ for any $o \in V(\Gamma^e)$. By Theorem \ref{thm:union theorem}, the family $(X_o)_{o \in V(\Gamma^e)}=\big (f^{-1}(\N_{\hGammae}(o,R+1)) \big)_{o \in V(\Gamma^e)}$ satisfies $\asdim \le n$ uniformly. Hence, the claim holds for any $R \in \NN\cup\{0\}$ by induction. Note $\asdim(\hGammae) \le 1$ by Corollary \ref{cor:coned off extension graph is quasitree}. Hence, we have $\asdim(\Gamma^e) \le n+\asdim(\hGammae) = n+1$ by Theorem \ref{thm:Hurewicz Theorem}.
\end{proof}

\section{Hyperbolicity, tightness, and fineness of the extension graph}
\label{sec:Hyperbolicity, tightness, and fineness of the extension graph}

The goal of this section is to prove Theorem \ref{thm:intro extension graph is tight} (1), (2), (3), which correspond to Proposition \ref{prop:extension graph is hyperbolic}, Proposition \ref{prop:extension graph is tight}, and Proposition \ref{prop:extension graph is fine} respectively. Throughout Section \ref{sec:Hyperbolicity, tightness, and fineness of the extension graph}, suppose that $\Gamma$ is a connected simplicial graph with $\girth(\Gamma)>20$ and $\G = \{G_v\}_{v \in V(\Gamma)}$ is a collection of non-trivial groups.

\subsection{Hyperbolicity}

\begin{prop}\label{prop:extension graph is hyperbolic}
    $\Gamma$ is hyperbolic if and only if $\Gamma^e$ is hyperbolic.
\end{prop}

\begin{proof}
    If $\Gamma^e$ is hyperbolic, then $\Gamma$ is hyperbolic by Corollary \ref{cor:Gamma is embedded into Gammae}. In the following, we assume that $\Gamma$ is $\delta$-hyperbolic with $\delta \in \NN$ and show hyperbolicity of $\Gamma^e$. Let $a,b,c \in V(\Gamma^e)$ and let $p \in \geo_{\Gamma^e}(a,b)$, $q \in \geo_{\Gamma^e}(b,c)$, and $r \in \geo_{\Gamma^e}(a,c)$. By Corollary \ref{cor:two geodesics with the same endpoints travel through common translations of Gamma} and Proposition \ref{prop: classification of geodesic triangle}, for any $v\in V(r)$, one of (1)-(4) holds, (1) $v \in V(p) \cup V(q)$, (2) there exist $g\in \Gamma\G$, a subpath $r'$ of $r$, and a subpath $p'$ of $p$ such that $v\in V(r')$, $r' \cup p' \subset g.\Gamma$, and $\max\{d_{\Gamma^e}(r'_-,p'_-), d_{\Gamma^e}(r'_+,p'_+)\} \le 2$, (3) there exist $g\in \Gamma\G$, a subpath $r'$ of $r$, and a subpath $q'$ of $q$ such that $v\in V(r')$, $r' \cup q' \subset g.\Gamma$, and $\max\{d_{\Gamma^e}(r'_-,q'_-), d_{\Gamma^e}(r'_+,q'_+)\} \le 2$, (4) there exist $g\in \Gamma\G$, a subpath $r'$ of $r$, a subpath $p'$ of $p$, and a subpath $q'$ of $q$ such that $v\in V(r')$, $r' \cup p' \cup q' \subset g.\Gamma$, and $\max\{d_{\Gamma^e}(r'_-,p'_-), d_{\Gamma^e}(p'_+,q'_-), d_{\Gamma^e}(r'_+,q'_+)\} \le 2$.
    
    In case (2), since $\Gamma$ is $\delta$-hyperbolic, there exists $w\in V(p')$ such that $d_{\Gamma^e}(v,w) \le 2\delta+2$. In case (3), similarly there exists $w\in V(q')$ such that $d_{\Gamma^e}(v,w) \le 2\delta+2$. In case (4), since $\Gamma$ is $\delta$-hyperbolic, there exists $w\in V(p')\cup V(q')$ such that $d_{\Gamma^e}(v,w) \le 3\delta+2$.
\end{proof}

\subsection{Tightness in the sense of Bowditch}

Recall that $\Gamma$ is a connected simplicial graph with $\girth(\Gamma)>20$ and $\G = \{G_v\}_{v \in V(\Gamma)}$ is a collection of non-trivial groups as assumed at the beginning of Section \ref{sec:Hyperbolicity, tightness, and fineness of the extension graph}. We first introduce the notion describing how a geodesic in $\Gamma^e$ travels though copies of $\Gamma$ in Definition \ref{def:S(p;n)} and study its property in Lemma \ref{lem:bound the number of planes}. This notion plays an important role in the proof of Proposition \ref{prop:extension graph is tight}.

\begin{defn}\label{def:S(p;n)}
     For $g \in \Gamma\G$ and a geodesic path $p$ in $\Gamma^e$ with $g.\Gamma \cap V(p) \neq \emptyset$, the subgraph $g.\Gamma \cap p$ is a subpath of $p$ by Corollary \ref{cor:Gamma is embedded into Gammae}, hence we denote the initial vertex of $g.\Gamma \cap p$ by $p_{in}(g.\Gamma) \in V(p)$ and the terminal vertex of $g.\Gamma \cap p$ by $p_{out}(g.\Gamma)$. For a geodesic path $p$ in $\Gamma^e$, we define $S(p\,;n) \subset \Gamma\G$ by
    \[
    S(p\,;n) = \{g \in \Gamma\G \mid \text{$g.\Gamma \cap V(p) \neq \emptyset$ and $d_{\Gamma^e}(p_{in}(g.\Gamma),p_{out}(g.\Gamma)) \ge n$ }\}.
    \]
\end{defn}

\begin{rem}
    Note $d_{\Gamma^e}(p_-, p_{in}(g.\Gamma)) \le d_{\Gamma^e}(p_-, p_{out}(g.\Gamma))$.
\end{rem}

\begin{lem}\label{lem:bound the number of planes}
    Let $p$ be a geodesic path in $\Gamma^e$. Then, for any $k\in \NN$ and $c \in V(p)$, we have
    \[
    |\{g \in S(p\,;3) \mid d_{\Gamma^e}(c,\,g.\Gamma \cap p) \le k\}|
    \le
    2(k+1)+4.
    \]
\end{lem}

\begin{proof}
    Define $A=\{g \in S(p\,;3) \mid d_{\Gamma^e}(c,\,g.\Gamma \cap p) \le k\}$ and also define $A_1,A_2,A_3 \subset \Gamma\G$ by
    \begin{align*}
        A_1 &= \{g \in A \mid d_{\Gamma^e}(p_-,p_{out}(g.\Gamma)) \le d_{\Gamma^e}(p_-,c)\}, \\
        A_2 &= \{g \in A \mid d_{\Gamma^e}(p_-,c) \le d_{\Gamma^e}(p_-,p_{in}(g.\Gamma))\}, \\
        A_3 &= \{g \in A \mid c\in V(g.\Gamma \cap p) \}.
    \end{align*}
    We have $A=A_1\cup A_2 \cup A_3$. Let $g,h \in S(p\,;3)$. If $p_{in}(g.\Gamma)=p_{in}(h.\Gamma)$ or $p_{out}(g.\Gamma)=p_{out}(h.\Gamma)$, then we have $g=h$ by $\diam_{\Gamma^e}(g.\Gamma \cap h.\Gamma) \ge 3$ and Remark \ref{rem:intersection of stabilizers} (3). Hence, the map $\phi_1 \colon A_1 \to \{0,\cdots,k\}$ defined by $\phi_1(g)=d_{\Gamma^e}(c,p_{out}(g.\Gamma))$ is injective. This implies $|A_1|\le k+1$. Similarly, we also get $|A_2|\le k+1$ since the map $\phi_2 \colon A_2 \to \{0,\cdots,k\}$ defined by $\phi_2(g)=d_{\Gamma^e}(c,p_{in}(g.\Gamma))$ is injective. Finally, we claim $|A_3| \le 4$. For each $i=0,1,2$, define $A_3^i$ by $A_3^i=\{g \in A_3 \mid d_{\Gamma^e}(c,p_{out}(g.\Gamma))=i\}$ and define $A_3^3$ by $A_3^3 = \{g \in A_3 \mid d_{\Gamma^e}(c,p_{out}(g.\Gamma))\ge 3\}$. We have $A_3 = \bigsqcup_{i=0}^3 A_3^i$. By the same argument as above, we have $|A_3^i| \le 1$ for any $i \in \{0,1,2\}$. When $A_3^3\neq \emptyset$, let $c'\in V(p)$ satisfy $d_{\Gamma^e}(p_-,c) < d_{\Gamma^e}(p_-,c')$ and $|p_{[c,c']}|=3$. For any $g, h \in A_3^3$, we have $p_{[c,c']} \subset g.\Gamma \cap h.\Gamma$. This implies $g=h$ by Remark \ref{rem:intersection of stabilizers} (3). Hence, $|A_3^3|\le 1$. Thus, we get $|A_3| \le 4$ and eventually $|A| \le |A_1|+|A_2|+|A_3| \le 2(k+1)+4$.
\end{proof}

We are now ready to prove Proposition \ref{prop:extension graph is tight}, which corresponds to Theorem \ref{thm:intro extension graph is tight} (2). In Proposition \ref{prop:extension graph is tight}, given $a,b\in V(\Gamma^e)$ and $r\in \NN$, we define $V(a,b), V(a,b\,;r) \subset V(\Gamma^e)$ by
\begin{align*}
V(a,b) &= \bigcup\{V(p) \mid p \in \geo_{\Gamma^e}(a,b) \}, \\
V(a,b\,;r) &= \bigcup\{ V(a',b') \mid a'\in \N_{\Gamma^e}(a,r),\, b'\in \N_{\Gamma^e}(b,r) \}.    
\end{align*}

\begin{prop}\label{prop:extension graph is tight}
    If $\Gamma$ is uniformly fine and $\delta$-hyperbolic with $\delta\in\NN$, then $\Gamma^e$ satisfies (1) and (2) below. In particular, $\Gamma^e$ is tight in the sense of Bowditch.
    \begin{itemize}
        \item [(1)]
        $\forall\, k\in\NN, \exists\, P_0 \in\NN, \forall\, a,b \in V(\Gamma^e), \forall\, c \in V(a,b),\, |V(a,b) \cap \N_{\Gamma^e}(c,k)| \le P_0$.
        \item [(2)] 
        $\forall\, k\in\NN, \exists\, P_1,k_1 \in\NN, \forall\, r \in\NN, \forall\, a,b \in V(\Gamma^e), \text{for all $c \in V(a,b)$ with $d_{\Gamma^e}(c,\{a,b\}) \ge r+k_1$}, 
        \newline
        |V(a,b\,;r) \cap \N_{\Gamma^e}(c,k)| \le P_1$.
    \end{itemize}
\end{prop}

\begin{proof}
    For $e\in E(\Gamma^e)$, $n \in \NN$ and $g \in \Gamma\G$, define $C(e,n,g)\subset V(\Gamma^e)$ by
    \[
    C(e,n,g)=\bigcup\{ V(\gamma) \mid \gamma \in \C_{\Gamma^e}(e,n), \gamma \subset g.\Gamma \}
    \]
    (see Definition \ref{def:concepts in graph theory} for $\C_{\Gamma^e}(e,n)$). Since $\Gamma$ is uniformly fine, there exists a map $f \colon \NN\cup\{0\} \to \NN$ such that $\sup_{e \, \in E(\Gamma^e),\, g \, \in \Gamma\G} |C(e,n,g)| \le f(n)$ for any $n\in\NN$.
    
    (1) Let $k \in \NN$, $a,b \in V(\Gamma^e)$, and $c\in V(a,b)$. We claim
    \begin{equation}\label{eq:P_0 for (1)}
        P_0=2k+1 +2(k+2\delta+3) \cdot (2(k+2\delta+3)+4) \cdot f(24\delta+32).
    \end{equation}
    Let $v\in V(a,b) \cap \N_{\Gamma^e}(c,k)$ and let $p,q \in \geo_{\Gamma^e}(a,b)$ satisfy $c \in V(p)$ and $v \in V(q)$. By Corollary \ref{cor:two geodesics with the same endpoints travel through common translations of Gamma}, either (i) or (ii) holds, (i) $v \in V(p)\cap V(q)$, (ii) there exist $g_0\in \Gamma\G$, a subpath $p'$ of $p$, a subpath $q'$ of $q$, $\alpha \in \geo_{\Gamma^e}(p_-, q_-)$, and $\beta \in \geo_{\Gamma^e}(p_+, q_+)$ such that the loop $p'\beta q^{\prime-1}\alpha^{-1}$ is a circuit in $g_0.\Gamma$ and satisfies $v\in V(q')$, $\min\{|p'|,|q'|\} \ge 5$, and $\max\{|\alpha|,|\beta|\} \le 2$. Define $A \subset E(\Gamma^e)$, $B \subset \Gamma\G$, and $C \subset V(\Gamma^e)$ by
    \begin{align*}
    A &=\{e \in E(p) \mid d_{\Gamma^e}(c,e) \le k+2\delta+2 \}, \\
    B &=\{g \in S(p\,;5) \mid d_{\Gamma^e}(c,g.\Gamma \cap p) \le k+2\delta+2\}, \\
    C &=\{v \in V(\Gamma^e) \mid \exists\,e\in A, \exists\,g \in B {\rm~s.t.~} v\in C(e,24\delta+32,g)\}.
    \end{align*}
    
    In case (i), we have $v \in \N_{\Gamma^e}(c,k)\cap V(p)$. 
    
    In case (ii), there exists $w \in V(p')$ such that $d_{\Gamma^e}(v,w) \le 2\delta+2$ since $g_0.\Gamma$ is $\delta$-hyperbolic. By $|p'|\ge 5$ and $d_{\Gamma^e}(c,w) \le d_{\Gamma^e}(c,v) + d_{\Gamma^e}(v,w) \le k+2\delta+2$, we have $g_0 \in B$. 
    
    If $d_{\Gamma^e}(p'_-, w) > 6\delta+6$, then we can take $w_0 \in V(p'_{[p'_-, w]})$ satisfying $d_{\Gamma^e}(w_0,w) = 4\delta + 3$. Note $d_{\Gamma^e}(p'_-, w_0) > 2\delta+3$. By $\delta$-hyperbolicity of $g_0.\Gamma$, there exists $v_0 \in q'_{[q'_-, v]}$ such that $d_{\Gamma^e}(v_0,w_0) \le 2\delta$. Take $\alpha_0 \in \geo_{\Gamma^e}(w_0, v_0)$. If there exists $z \in V(\alpha_0)\cap V(q'_{[v,q'_+]})$, then the path $\alpha_{0[w_0,z]}q^{\prime -1}_{[z,v_0]}$ becomes a geodesic in $\Gamma^e$ from $w_0$ to $v_0$, hence $d_{\Gamma^e}(w_0, w) \le d_{\Gamma^e}(w_0, v)+d_{\Gamma^e}(v,w) \le d_{\Gamma^e}(w_0, v_0)+d_{\Gamma^e}(v,w) \le 2\delta+ (2\delta+2)$, which contradicts $d_{\Gamma^e}(w_0, w)=4\delta+3$. Hence, $V(\alpha_0)\cap V(q'_{[v,q'_+]})=\emptyset$. We can take a subpath $\alpha_1$ of $\alpha_0$ from $w_1 \in V(p'_{[w_0,w]})$ to $v_1 \in V(q'_{[v_0,v]})$ such that the path $p^{\prime -1}_{[p'_+, w_1]} \alpha_1 q'_{[v_1,q'_+]}$ has no self-intersection. Note $|\alpha_1| \le |\alpha| \le 2\delta$. 
    
    If $d_{\Gamma^e}(p'_-, w) \le 6\delta+6$, then define $w_1,v_1,\alpha_1$ by $w_1=p'_-$, $v_1=q'_-$, and $\alpha_1=\alpha$. 
    
    Similarly, we can take $w_2 \in V(p'_{[w,p'_+]})$, $v_2 \in V(q'_{[v,q'_+]})$, and a geodesic $\beta_1$ in $\Gamma^e$ from $w_2$ to $v_2$ such that the loop $\gamma$ defined by $\gamma=p'_{[w_1,w_2]} \beta_1 q^{\prime -1}_{[v_2,v_1]}\alpha_1^{-1}$ is a circuit and satisfies $|\gamma| \le 4(6\delta+8)$. Take $e \in E(p'_{[w_1,w_2]})$ one of whose endpoints is $w$, then we have $e \in A$ and $v \in C(e,24\delta+32,g_0)$ since we have $\gamma \subset g_0.\Gamma$ by Corollary \ref{cor:Gamma is embedded into Gammae}. 
    
    Hence, $v\in C$ in case (ii). This implies $V(a,b) \cap \N_{\Gamma^e}(c,k) \subset (\N_{\Gamma^e}(c,k)\cap V(p)) \cup C$. By Lemma \ref{lem:bound the number of planes}, we have $|A| \le 2(k+2\delta+3)$ and $|B| \le 2(k+2\delta+3)+4$. Hence, we get $|C| \le |A|\cdot|B|\cdot f(24\delta+32) \le 2(k+2\delta+3)(2(k+2\delta+3)+4)f(24\delta+32)$. Thus, the constant $P_0$ in \eqref{eq:P_0 for (1)} satisfies the condition.

    (2) Let $k\in \NN$. We claim
    \begin{align*}
        P_1 &= 2k+1 + 2(k+4\delta+3)\cdot(2(k+4\delta+3)+4)\cdot f(32\delta+32), \\
        k_1 &= k+12\delta+23.
    \end{align*}
    Let $r \in \NN$ and $a,b \in V(\Gamma^e)$. Let $c \in V(a,b)$ satisfy $d_{\Gamma^e}(c,\{a,b\}) \ge r+k_1$. Take $p \in \geo_{\Gamma^e}(a,b)$ satisfying $c \in V(p)$. Note that the existence of $c$ implies $d_{\Gamma^e}(a,b) = d_{\Gamma^e}(a,c) + d_{\Gamma^e}(c,b) \ge 2(r+k_1) \ge 2r + 32$. Define $A \subset E(\Gamma^e)$, $B \subset \Gamma\G$, and $C \subset V(\Gamma^e)$ by
    \begin{align*}
    A &=\{e \in E(p) \mid d_{\Gamma^e}(c,e) \le k+4\delta+2 \}, \\
    B &=\{g \in S(p\,;3) \mid d_{\Gamma^e}(c,g.\Gamma \cap p) \le k+4\delta+2\}, \\
    C &=\{v \in V(\Gamma^e) \mid \exists\,e\in A, \exists\,g \in B {\rm~s.t.~} v \in C(e, 32\delta+32 ,g) \}.
    \end{align*}
    Let $v\in V(a,b\,;r) \cap \N_{\Gamma^e}(c,k)$ and let $a'\in \N_{\Gamma^e}(a,r)$, $b'\in \N_{\Gamma^e}(b,r)$, and $q \in \geo_{\Gamma^e}(a',b')$ satisfy $v \in V(q)$. By applying Corollary \ref{cor:two geodesics with the same endpoints travel through common translations of Gamma} and Proposition \ref{prop: classification of geodesic triangle} when $V(p)\cap V(q) \neq \emptyset$ and by applying Lemma \ref{lem:two geodesics form a ladder} when $V(p)\cap V(q) = \emptyset$, we can see that either (i) or (ii) holds.
    \begin{itemize}
        \item[(i)]
        $v \in V(p)\cap V(q)$.
        \item[(ii)]
        There exist $g_0\in \Gamma\G$, a subpath $p'$ of $p$, a subpath $q'$ of $q$, $\alpha \in \geo_{\Gamma^e}(p_-,q_-)$, and $\beta \in \geo_{\Gamma^e}(p_+,q_+)$ satisfying $p'\cup q' \subset g_0.\Gamma$ and $v \in V(q')$ such that one of (ii-1)-(ii-4) holds.
        \begin{itemize}
            \item[(ii-1)]
            The loop $p'\beta q^{\prime-1}\alpha^{-1}$ is a circuit and satisfies $\min\{|p'|,|q'|\} \ge 3$ and $\min\{|\alpha|,|\beta|\} \le 2$.
            \item[(ii-2)]
            The path $p'\beta q^{\prime -1}$ has no self-intersection and satisfies $|\beta| \le 2$ and $\max\{d_{\Gamma^e}(a,p'_-), d_{\Gamma^e}(a,q'_-)\}\le r+22$.
            \item[(ii-3)]
            The path $p^{\prime -1}\alpha q'$ has no self-intersection and satisfies $|\alpha| \le 2$ and $\max\{d_{\Gamma^e}(b,p'_+), d_{\Gamma^e}(b,q'_+)\} \le r+22$.
            \item[(ii-4)]
            $V(p')\cap V(q')=\emptyset$ and $\max\{d_{\Gamma^e}(a,p'_-), d_{\Gamma^e}(a,q'_-),d_{\Gamma^e}(b,p'_+), d_{\Gamma^e}(b,q'_+)\}\le r+22$.
        \end{itemize}
    \end{itemize}
    
    In case (i), we have $v \in \N_{\Gamma^e}(c,k)\cap V(p)$. 
    
    In case (ii-1), we can see $v \in C$ in the same way as the proof of Proposition \ref{prop:extension graph is tight} (1). 
    
    In case (ii-2), take $\epsilon \in \geo_{\Gamma^e}(a, q'_-)$, then we have $\alpha \subset \N_{\Gamma^e}(\epsilon \cup p_{[a,p'_-]}, \delta)$. This and $\max\{d_{\Gamma^e}(a,p'_-), d_{\Gamma^e}(a,q'_-)\}\le r+22$ imply $\alpha \subset \N_{\Gamma^e}(a,r+22+\delta)$. Suppose for contradiction that there exists $z \in V(\alpha)$ such that $d_{\Gamma^e}(z,v) \le 11\delta$, then
    \[
    d_{\Gamma^e}(a,c) \le d_{\Gamma^e}(a,z) + d_{\Gamma^e}(z,v) + d_{\Gamma^e}(v,c) \le (r+22+\delta)+11\delta+k,
    \]
    which contradicts $d_{\Gamma^e}(a,c)\ge r+k_1$. Hence, $\forall\, z \in V(\alpha),\, d_{\Gamma^e}(z,v) > 11\delta$. In particular, we can take $v_0 \in V(q'_{[q'_-, v]})$ with $d_{\Gamma^e}(v_0, v) = 9\delta$ and there exists $w_0 \in V(p')$ such that $d_{\Gamma^e}(v_0,w_0)\le 2\delta$ by $\delta$-hyperbolicity of $g_0.\Gamma$. Take $\alpha_0 \in \geo_{\Gamma^e}(w_0, v_0)$. Note $V(\alpha_0) \cap V(q'_{[v,q'_+]}\beta^{-1})=\emptyset$ by $d_{\Gamma^e}(v_0, v) = 9\delta$, $|\alpha_0|\le 2\delta$, and $|\beta| \le 2 \le 2\delta$. Hence, we can take a subpath $\alpha_1$ of $\alpha_0$ from $w_1 \in V(p'_{[w_0,p'_+]})$ to $v_1 \in V(q'_{[v_0,v]})$ such that the loop $p'_{[w_1,p'_+]} \beta (q'_{[v_1,q'_+]})^{-1}\alpha_1^{-1}$ is a circuit. If $d_{\Gamma^e}(v,q'_+)>5\delta+2$, then we can take $v'_0 \in V(q'_{[v,q'_+]})$ with $d_{\Gamma^e}(v,v'_0) = 3\delta$ and there exists $w'_0 \in V(p'_{[w_1,p'_+]})$ such that $d_{\Gamma^e}(v'_0,w'_0) \le 2\delta$ by $\delta$-hyperbolicity of $g_0.\Gamma$. Take $\beta_0 \in \geo_{\Gamma^e}(w'_0, v'_0)$. We can take a subpath $\beta_1$ of $\beta_0$ from $w_2 \in V(p'_{[w_1,w'_0]})$ to $v_2 \in V(q'_{[v,v'_0]})$ such that the loop $p'_{[w_1,w_2]} \beta_1 (q'_{[v_1,v_2]})^{-1}\alpha_1^{-1}$ is a circuit. If $d_{\Gamma^e}(v,q'_+) \le 5\delta+2$, then we define $w_2,v_2,\beta_1$ by $w_2=p'_+$, $v_2=q'_+$, and $\beta_1=\beta$. In either case, the loop $\gamma$ defined by $\gamma=p'_{[w_1,w_2]} \beta_1 (q'_{[v_1,v_2]})^{-1}\alpha_1^{-1}$ is a circuit satisfying $|\gamma| \le 2(11\delta+5\delta+4)$ and $v \in V(\gamma)$ and there exists $w \in V(p'_{[w_1, w_2]})$ such that $d_{\Gamma^e}(w,v) \le 4\delta$. Note $d_{\Gamma^e}(w,c) \le d_{\Gamma^e}(w,v) + d_{\Gamma^e}(v,c) \le 4\delta+k$. We also have $g_0 \in S(p\,;3)$ by $d_{\Gamma^e}(w_1,w_2) \ge 9\delta-4\delta > 3$. Hence, by taking $e \in E(p'_{[w_1, w_2]})$ one of whose endpoints is $w$, we have $e\in A$, $g_0 \in B$, and $v \in C(e,32\delta+8,g_0)$. Hence, $v \in C$.

    In case (ii-3), we can show $v \in C$ in the same way as case (ii-2). 
    
    In case (ii-4), in the same way as case (ii-2), we can show $\forall\, z \in V(\alpha)\cup V(\beta),\, d_{\Gamma^e}(z,v) > 11\delta$ by $\max\{d_{\Gamma^e}(a,p'_-), d_{\Gamma^e}(a,q'_-),d_{\Gamma^e}(b,p'_+), d_{\Gamma^e}(b,q'_+)\}\le r+22$. In particular, we can take $v_0 \in V(q'_{[q'_-, v]})$ and $v'_0 \in V(q'_{[v,q'_+]})$ satisfying $d_{\Gamma^e}(v_0,v) = d_{\Gamma^e}(v,v'_0) = 5\delta$ and there exist $w_0 \in V(p')$ and $w'_0 \in V(p'_{[w_0, p'_+]})$ such that $\max\{d_{\Gamma^e}(w_0,v_0), d_{\Gamma^e}(w'_0,v'_0)\} \le 2\delta$ by $\delta$-hyperbolicity of $g_0.\Gamma$. Take $\alpha_0 \in \geo_{\Gamma^e}(w_0, v_0)$ and $\beta_0 \in \geo_{\Gamma^e}(w'_0, v'_0)$. By $d_{\Gamma^e}(v_0, v'_0)=10\delta$ and $\max\{|\alpha_0|,|\beta_0|\} \le 2\delta$, we can take a subpath $\alpha_1$ of $\alpha_0$ from $w_1 \in V(p'_{[w_0,w'_0]})$ to $v_1 \in V(q'_{[v_0,v]})$ and a subpath $\beta_1$ of $\beta_0$ from $w_2 \in V(q'_{[w_1,w'_0]})$ to $v_2 \in V(q'_{[v, v'_0]})$ such that the loop $\gamma$ defined by $\gamma=p'_{[w_1,w_2]} \beta_1 (q'_{[v_1,v_2]})^{-1}\alpha_1^{-1}$ is a circuit. Note $|\gamma| \le 4\cdot 7\delta = 28\delta$ and $v \in V(\gamma)$. There exists $w \in V(p'_{[w_1, w_2]})$ such that $d_{\Gamma^e}(w,v) \le 4\delta$. Note $d_{\Gamma^e}(w,c) \le d_{\Gamma^e}(w,v) + d_{\Gamma^e}(v,c) \le 4\delta+k$. We also have $g_0 \in S(p\,;3)$ by $d_{\Gamma^e}(w_1,w_2) \ge 10\delta-4\delta > 3$. Hence, by taking $e \in E(p'_{[w_1, w_2]})$ one of whose endpoints is $w$, we have $e\in A$, $g_0 \in B$, and $v \in C(e,28\delta,g_0)$. Hence, $v \in C$.

    Thus, we have $V(a,b) \cap \N_{\Gamma^e}(c,k) \subset (\N_{\Gamma^e}(c,k)\cap V(p)) \cup C$. By Lemma \ref{lem:bound the number of planes}, we have $|A| \le 2(k+4\delta+3)$ and $|B| \le 2(k+4\delta+3)+4$. By this and $|C| \le |A|\cdot|B|\cdot f(32\delta+32)$, the constants $P_1$ and $k_1$ satisfy the condition.
   \end{proof}

In Corollary \ref{cor:action on extension graph is acylindrical} below, we record one application of Proposition \ref{prop:extension graph is tight}. Proposition \ref{prop:extension graph is tight} has another application for proving an analytic property of graph product in the upcoming paper.

\begin{cor}\label{cor:action on extension graph is acylindrical}
   If $\Gamma$ is uniformly fine and hyperbolic, then the action $\Gamma\G \act \Gamma^e$ is acylindrical.
\end{cor}

\begin{proof}
    This follows from Corollary \ref{cor:stabilizer of two far vertices is trivial} (1), Proposition \ref{prop:extension graph is tight}, and \cite[Lemma 3.3]{Bow08}.
\end{proof}

When $\Gamma$ is finite, the statement of Corollary \ref{cor:action on extension graph is acylindrical} essentially follows from \cite[Corollary C]{Val21} since the contact graph and crossing graph are quasi-isometric (see \cite[Proposition 3.3]{Gen25} and Proposition \ref{prop:connection to crossing and contact graphs}).

\subsection{Fineness}

Recall that $\Gamma$ is a connected simplicial graph with $\girth(\Gamma)>20$ and $\G = \{G_v\}_{v \in V(\Gamma)}$ is a collection of non-trivial groups as assumed at the beginning of Section \ref{sec:Hyperbolicity, tightness, and fineness of the extension graph}. We first prove Lemma \ref{lem:Greenlinger} below, which enables us to run induction on the length of a circuit in the proof of Proposition \ref{prop:extension graph is fine}. Lemma \ref{lem:Greenlinger} is similar to Greenlinger's lemma for small cancellation groups.

\begin{lem}\label{lem:Greenlinger}
    If $p$ is a circuit in $\Gamma^e$, then there exist $a,b \in V(p)$ and $g \in\Gamma\G$ such that $p_{[a,b]} \subset g.\Gamma$, $d_{\Gamma^e}(a,b) \le 4$, and letting $q$ be a unique geodesic in $\Gamma^e$ from $a$ to $b$, the loop $p_{[a,b]}q^{-1}$ is a circuit in $g.\Gamma$.
\end{lem}

\begin{proof}
    Note that we can set an initial (and terminal) vertex of the circuit $p$ at any vertex in $V(p)$. Define $\mathcal{I}$ to be the set of all triples $(v,n,\mathbf{x})$, where $v \in V(p)$, $n \in \NN$, and $\mathbf{x}=(x_0,\cdots, x_n)$ with $x_0 = x_n = v$ is a subsequence of $V(p)$, which is considered as a circuit with $p_-=p_+=v$, such that for any $i \in \{1,\cdots,n\}$, there exists $g \in \Gamma\G$ satisfying $p_{[x_{i-1}, x_i]} \subset g.\Gamma$. Take $(v,n,\mathbf{x}) \in \mathcal{I}$ satisfying $n = \min\{n' \in \NN \mid (v',n',\mathbf{x'}) \in \mathcal{I} \}$. When $n=1$, $p$ is a circuit in $g.\Gamma$ for some $g \in \Gamma\G$. Hence, the statement follows by setting $a=v$ and $b=v$. 
    
    In the following, assume $n\ge 2$. For each $i \in \{1,\cdots,n\}$, let $g_i \in \Gamma\G$ satisfy $p_{[x_{i-1}, x_i]} \subset g_i.\Gamma$. By minimality of $n$, the path $p$ with the subsequence $\mathbf{x}$ satisfies Definition \ref{def: admissible path} (2). Define $I_1,I_2$ by $I_1=\{i \in \{1,\cdots,n\} \mid \text{$p_{[x_{i-1},x_i]}$ is geodesic in $\Gamma^e$}\}$ and $I_2 = \{i \in \{1,\cdots,n\} \mid d_{\Gamma^e}(x_{i-1}, x_i) \ge 5 \}$. Suppose $\{1,\cdots,n\} = I_1 \cup I_2$ for contradiction. Define a loop $p'$ passing through all $x_0,\cdots,x_n$ by concatenating each subpath $p'_{[x_{i-1}, x_i]}$ defined as follows. If $i\in I_1$, then define $p'_{[x_{i-1}, x_i]}=p_{[x_{i-1}, x_i]}$. If $i\in \{1,\cdots,n\} \setminus I_1$, then take $q_i \in \geo_{\Gamma^e}(x_{i_1},x_i)$ and define $p'_{[x_{i-1}, x_i]}=q_i$. Note $q_i \subset g_i.\Gamma$ by Corollary \ref{cor:Gamma is embedded into Gammae}. 
    
    We claim that $p'$ is admissible. Definition \ref{def: admissible path} (1) follows from the definition of $p'$. Suppose for contradiction that there exist $g \in \Gamma\G$ and $i \in \{1,\cdots,n-1\}$ such that $p'_{[x_{i-1},x_{i+1}]} \subset g.\Gamma$. If $i\in I_1$, then we have $p_{[x_{i-1},x_i]} = p'_{[x_{i-1},x_i]} \subset g.\Gamma$. If $i \in I_2$, then we have $g_i=g$ by Remark \ref{rem:intersection of stabilizers} (3) since we have $d_{\Gamma^e}(x_{i-1}, x_i) \ge 5$ and $\{x_{i-1}, x_i\} \subset g_i.\Gamma \cap g.\Gamma$. This implies $p_{[x_{i-1},x_i]} \subset g_i.\Gamma = g.\Gamma$. By arguing similarly for $p_{[x_i,x_{i+1}]}$ as well, we get $p_{[x_{i-1},x_{i+1}]} \subset g.\Gamma$, which contradicts that the path $p$ with $\mathbf{x}$ satisfies Definition \ref{def: admissible path} (2). Hence, $p'$ also satisfies Definition \ref{def: admissible path} (2). Let $i\in\{1,\cdots,n-1\}$. If $\{i,i+1\} \subset I_1$, then the subpath $p'_{[x_{i-1},x_{i+1}]}$ has no backtracking by $p'_{[x_{i-1},x_{i+1}]}=p_{[x_{i-1},x_{i+1}]}$ since $p$ has no backtracking. If $i\in I_2$ or $i_{i+1} \in I_2$, then we have $\max\{ d_{\Gamma^e}(x_{i-1}, x_i), d_{\Gamma^e}(x_i, x_{i+1}) \} \ge 5$. Hence, $p'$ satisfies Definition \ref{def: admissible path} (3). Thus, $p'$ is admissible. This contradicts Proposition \ref{prop:Gamma is embedded into Gammae} by $x_0=x_n=v$ and $n \ge 2$.

    Hence, there exists $i \in \{1, \cdots, n\}$ such that $p_{[x_{i-1}, x_i]}$ is not geodesic in $\Gamma^e$ and $d_{\Gamma^e}(x_{i-1}, x_i) \le 4$. Take $q \in \geo_{\Gamma^e}(x_{i-1}, x_i)$. By $\{x_{i-1}, x_i\} \subset g_i. \Gamma$ and Corollary \ref{cor:Gamma is embedded into Gammae}, we have $q \subset g_i.\Gamma$. Also, $q$ is a unique geodesic in $\Gamma^e$ from $x_{i-1}$ to $x_i$ by $\girth(\Gamma) > 20$. For each $v \in V(q)$, there exists at most one vertex $w \in V(p_{[x_{i-1}, x_i]})$ such that $v=w$ since $p_{[x_{i-1}, x_i]}$ has no self-intersection by $n \ge 2$. Since $p_{[x_{i-1}, x_i]}$ is not geodesic in $\Gamma^e$, we can see that there exist a subpath $p''$ of $p_{[x_{i-1}, x_i]}$ and a subpath $q''$ of $q$ such that we have $p''_-=q''_-$ and $p''_+=q''_+$ and the loop $p''q^{\prime\prime -1}$ is a circuit in $g_i.\Gamma$. By $|q''|\le |q| \le 4$, the subpath $p''$ satisfies the statement by setting $a = p''_-$ and $b = p''_+$.
\end{proof}

Before proving Proposition \ref{prop:extension graph is fine}, we present an immediate corollary of Lemma \ref{lem:Greenlinger} below.

\begin{cor}\label{cor:girth of exetnsion graph is the same}
    We have $\girth(\Gamma^e) = \girth(\Gamma)$. Moreover, if $p$ is a circuit in $\Gamma^e$ of length $\girth(\Gamma)$ i.e. $|p|=\girth(\Gamma)$, then there exists $g \in \Gamma\G$ such that $p \subset g.\Gamma$.
\end{cor}

\begin{proof}
   Since $\Gamma$ is an induced subgraph of $\Gamma^e$ by Convention \ref{conv:Gamma is embedded to Gammae}, we have $\girth(\Gamma^e) \le \girth(\Gamma)$. To show $\girth(\Gamma) \le \girth(\Gamma^e)$, let $p$ be a circuit in $\Gamma^e$. By Lemma \ref{lem:Greenlinger}, there exist $a,b \in V(p)$ and $g \in\Gamma\G$ such that $p_{[a,b]} \subset g.\Gamma$, $d_{\Gamma^e}(a,b) \le 4$, and letting $q$ be a unique geodesic in $\Gamma^e$ from $a$ to $b$, the loop $p_{[a,b]}q^{-1}$ is a circuit in $g.\Gamma$. Let $p_0$ be the subpath of $p$ from $a$ to $b$ that complements $p_{[a,b]}$ to form $p$ i.e. $p=p_{[a,b]}p_0^{-1}$. Since $p_{[a,b]}q^{-1}$ is a circuit in $g.\Gamma$, we have $|p_{[a,b]}|+|q| \ge \girth(\Gamma)$. We also have $|p_0| \ge |q|$ since $q$ is a geodesic in $\Gamma^e$ from $a$ to $b$. Hence, $|p|=|p_{[a,b]}|+|p_0| \ge |p_{[a,b]}|+|q| \ge \girth(\Gamma)$. This implies $\girth(\Gamma^e) \ge \girth(\Gamma)$ since $p$ is arbitrary. In the above argument, if $p$ satisfies $|p|=\girth(\Gamma)$ in addition, then we have $\girth(\Gamma)=|p| \ge |p_{[a,b]}|+|p_0| \ge |p_{[a,b]}|+|q| \ge \girth(\Gamma)$. This implies $|p_0|=|q|$, hence, $p_0$ is geodesic. Hence, we have $p_0 = q$ since $q$ is a unique geodesic in $\Gamma^e$ from $a$ to $b$. This implies $p=p_{[a,b]}p_0^{-1}=p_{[a,b]}q^{-1} \subset g.\Gamma$.
\end{proof}

We are now ready to prove Proposition \ref{prop:extension graph is fine}, which corresponds to Theorem \ref{thm:intro extension graph is tight} (3). In Proposition \ref{prop:extension graph is fine} below, the assumption that $\{G_v\}_{v\in V(\Gamma)}$ is a collection of finite groups is essential. Indeed, if a vertex group $G_v$ is infinite and there exists a circuit of length $n$ in $\Gamma$ containing an edge $e\in E(\Gamma)$ with $e_- = v$, then there exist infinitely many circuits of length $n$ in $\Gamma^e$ containing $e$.

\begin{prop}\label{prop:extension graph is fine}
    If $\Gamma$ is fine and $\G=\{G_v\}_{v\in V(\Gamma)}$ is a collection of non-trivial finite groups, then $\Gamma^e$ is fine.
\end{prop}

\begin{proof}
    For $e\in E(\Gamma^e)$ and $n \in \NN$, define $L(e,n), M(e,n)\subset E(\Gamma^e)$ and $P(e)\subset \Gamma\G$ by
    \begin{align*}
        L(e,n)&=\bigcup\{ E(\gamma) \mid \gamma \in \C_{\Gamma^e}(e,n) \}, \\
        P(e)&=\{g \in \Gamma\G \mid e \subset g.\Gamma\}, \\
        M(e,n)
        &=\bigcup\{E(\gamma) \mid \text{$\gamma \in \C_{\Gamma^e}(e,n)$ s.t. $\exists\,g\in P(e)$ with $\gamma \subset g.\Gamma$} \}.
    \end{align*}
    Let $e' \in E(\Gamma)$ satisfy $e \in \Gamma\G.e'$. We have $|P(e)| \le |\stab_{\Gamma\G}(e'_-)\cap\stab_{\Gamma\G}(e'_+)| = |G_{e'_-}\times G_{e'_+}|<\infty$ by Remark \ref{rem:intersection of stabilizers} (1), because vertex groups are finite. This implies $|M(e,n)| < \infty$ for any $e \in E(\Gamma^e)$ and $n\in\NN$ since $\Gamma$ is fine. In the following, we will show that for any $n\in\NN$,
    \begin{align}\tag{$\ast$}\label{eq:fine graph}
        \forall\, e \in E(\Gamma^e),\, |\C_{\Gamma^e}(e,n)| < \infty
    \end{align}
    holds by induction on $n$. By Corollary \ref{cor:girth of exetnsion graph is the same}, the statement \eqref{eq:fine graph} is true for any $n \in \NN$ with $n \le \girth(\Gamma)$. Given $N\ge \girth(\Gamma)$, assume that the statement \eqref{eq:fine graph} is true for any $n$ with $n < N$. Let $e\in E(\Gamma^e)$ and $p \in \C_{\Gamma^e}(e,N)$. By Lemma \ref{lem:Greenlinger}, there exist $a,b \in V(p)$ and $g \in\Gamma\G$ such that $p_{[a,b]} \subset g.\Gamma$, $d_{\Gamma^e}(a,b) \le 4$, and letting $q$ be a unique geodesic in $\Gamma^e$ from $a$ to $b$, the loop $p_{[a,b]}q^{-1}$ is a circuit in $g.\Gamma$. Let $p_0$ be the subpath of $p$ from $a$ to $b$ that complements $p_{[a,b]}$ to form $p$ i.e. $p=p_{[a,b]}p_0^{-1}$. Note $|p_{[a,b]}| + |q^{-1}| \le |p_{[a,b]}| + |p_0| \le N$. 
    
    If $|p_0|=0$ i.e. $a=b$, then we have $p \subset g.\Gamma$, hence $E(p) \subset M(e,N)$. 
    
    If $|p_0|>0$, then $p_0$ has no self-intersection. Hence, there exists a subsequence $(a \! =)\,x_0, \cdots, x_k \,(= \! b)$ of $V(p_0)$ with $\forall i \ge 1, x_{i-1} \neq x_i$ such that $\{x_0,\cdots,x_k\}\subset q$ and for every $i \in \{1,\cdots,k\}$, letting $q_i$ be the subpath of $q$ or $q^{-1}$ from $x_{i-1}$ to $x_i$, either (A1) or (A2) holds, (A1) $p_{0[x_{i-1},x_i]}=q_i$, (A2) the loop $p_{[x_{i-1},x_i]}q_i^{-1}$ is a circuit. In case (A2), by $d_{\Gamma^e}(a,b) \le 4$ and $|p_{[a,b]}q^{-1}|>20$, we have $|p_{[x_{i-1},x_i]}q_i^{-1}| < |p|-16+4\le N-12$. Note $k\le 4$ by $|q| \le 4$ since $p_0$ has no self-intersection. We'll discuss three cases (B1)-(B3), (B1) when $e \in E(p_{[a,b]})$, (B2) when $e \in p_{0[x_{i_0-1},x_{i_0}]}$ for some $i_0\in\{1,\cdots,k\}$ and case (A1) holds for $p_{0[x_{i_0-1},x_{i_0}]}$, (B3) when $e \in p_{0[x_{i_0-1},x_{i_0}]}$ for some $i_0\in\{1,\cdots,k\}$ and case (A2) holds for $p_{0[x_{i_0-1},x_{i_0}]}$.

    In case (B1), we have $E(p_{[a,b]}q^{-1}) \subset M(e,N)$ by $|p_{[a,b]}q^{-1}| \le |p_{[a,b]}| + |p_0| \le N$. For each $i \in \{1,\cdots,k\}$, in case (A1), we have $E(p_{0[x_{i-1},x_i]}) \subset M(e,N)$ and in case (A2), we have $p_{0[x_{i-1},x_i]}q_i^{-1} \in \bigcup_{e_1 \in M(e,N)} \C_{\Gamma^e}(e_1,N-12)$. Thus,
    \[
    E(p) \subset M(e,N)\cup \bigcup\nolimits_{e_1 \in M(e,N)}L(e_1,N-12).
    \]
    
    In case (B2), we have $E(p_{[a,b]}q^{-1}) \subset M(e,N)$. Hence, in the same way as case (A1), we can see $E(p) \subset M(e,N)\cup \bigcup_{e_1 \in M(e,N)}L(e_1,N-12)$. 
    
    In case (B3), we have $p_{[x_{i_0-1},x_{i_0}]}q_{i_0}^{-1} \in \C_{\Gamma^e}(e,N-12)$. This implies $E(p_{[a,b]}q^{-1}) \subset \bigcup_{e_1 \in L(e,N-12)}M(e_1,N)$. Hence, for each $i \in \{1,\cdots,k\}$, we have
    \begin{align*}
        &\text{$E(p_{0[x_{i-1},x_i]}) \subset \bigcup\nolimits_{e_1 \in L(e,N-12)}M(e_1,N)$ in case (A1)     and} \\
        &\text{$p_{0[x_{i-1},x_i]}q_i^{-1} \in \bigcup \big\{ \C_{\Gamma^e}(e_2,N-12) \,\big|\, e_2 \in \bigcup\nolimits_{e_1 \in L(e,N-12)}M(e_1,N)\big\}$ in case (A2).}
    \end{align*}
    Hence, $E(p) \subset \bigcup_{e_1 \in L(e,N-12)}M(e_1,N) \cup \bigcup \{L(e_2,N-12) \mid e_2 \in \bigcup_{e_1 \in L(e,N-12)}M(e_1,N)\}$.

    Thus, in any case, we have $E(p) \subset M(e,N)\cup \bigcup_{e_1 \in M(e,N)} L(e_1,N-12) \cup \bigcup_{e_1 \in L(e,N-12)}M(e_1,N) \cup \bigcup \{L(e_2,N-12) \mid e_2 \in \bigcup_{e_1 \in L(e,N-12)}M(e_1,N)\}$. This implies $|L(e,N)|<\infty$ since for any $e' \in E(\Gamma^e)$, we have $|M(e',N)|<\infty$ and $|\C_{\Gamma^e}(e',N-12)|<\infty$, where the latter follows from the assumption of induction. Thus, the statement \eqref{eq:fine graph} is true for $n=N$.
\end{proof}

Proposition \ref{prop:extension graph becomes uniformly fine} below is a variant of Proposition \ref{prop:extension graph is fine}.

\begin{prop}\label{prop:extension graph becomes uniformly fine}
    If $\Gamma$ is uniformly fine and $\G=\{G_v\}_{v\in V(\Gamma)}$ is a collection of non-trivial finite groups with $\sup_{v \in V(\Gamma)}|G_v| < \infty$, then $\Gamma^e$ is uniformly fine.
\end{prop}

\begin{proof}
    Since $\Gamma$ is uniformly fine, there exists $f \colon \NN \to \NN$ such that for any $n \in \NN$ and $e \in E(\Gamma)$, we have $|\C_\Gamma(e,n)| \le f(n)$. Let $P(e)$ and $M(e,n)$ be as in the proof of Proposition \ref{prop:extension graph is fine}, then we have $\sup_{e \in E(\Gamma^e)}|P(e)| \le (\sup_{v \in V(\Gamma)}|G_v|)^2$ and $\sup_{(e,n) \in E(\Gamma^e)\times \NN}|M(e,n)| \le f(n)\cdot (\sup_{v \in V(\Gamma)}|G_v|)^2$. Hence, we can show that $\Gamma^e$ is uniformly fine by induction on the length of circuits in the same way as the proof of Proposition \ref{prop:extension graph is fine}.
\end{proof}

In Lemma \ref{lem:doubling is not true} below, we verify that \cite[Lemma 22]{KK13}, which is a key lemma in \cite{KK13}, is not true anymore in general. The \emph{double of} a graph $X$ \emph{along} $L \subset V(X)$ is the graph obtained by taking two copies of $X$ and gluing them along the
 copies of the induced subgraph on $L$.

\begin{lem}\label{lem:doubling is not true}
    Suppose that $\Gamma$ is a circuit of length more than 20 and there exists $v \in V(\Gamma)$ such that $|G_v| \in 2\NN+1$. Then, there exists a finite induced subgraph $\Lambda$ of $\Gamma^e$ such that for any $\ell > 0$, any sequence $v_1,\cdots,v_\ell$ in $V(\Gamma^e)$, and any sequence of finite induced subgraphs $\Gamma=\Gamma_0 \subset \Gamma_1 \subset \cdots \subset \Gamma_\ell$ of $\Gamma^e$, where $\Gamma_i$ is obtained by taking the double of $\Gamma_{i-1}$ along $\st_{\Gamma_{i-1}} (v_i)$ for each $i = 1,\cdots,\ell$, we have $\Lambda \not\subset \Gamma_\ell$. 
\end{lem}

\begin{proof}
    Note that $\st_{\Gamma}(v)$ is a path of length 2 without backtracking (i.e. $\st_{\Gamma}(v)=\{u,v,w\}$ and $\{(u,v), (v,w)\}\subset E(\Gamma)$) since $\Gamma$ is a circuit. For an induced subgraph $X$ of $\Gamma^e$ and a path $q$ of length 2 in $X$ without backtracking, we define $C(q,X) \subset 2^{V(\Gamma^e)}$ by
    \begin{align*}
        C(q,X)
        =
        \{p \in 2^{V(\Gamma^e)} \mid \text{$p$ is a circuit in $X$ such that $|p|=\girth(\Gamma)$ and $q \subset V(p)$} \}.
    \end{align*}
    (For example, $C(\st_\Gamma(v),\Gamma)=\{\Gamma\}$.) Let $p \in C(\st_\Gamma(v),\Gamma^e)$. By $|p|=\girth(\Gamma)$ and Corollary \ref{cor:girth of exetnsion graph is the same}, there exists $g \in \Gamma\G$ such that $p \subset g.\Gamma$. This implies $\st_\Gamma(v) \subset p \subset g.\Gamma$. Hence, $g \in G_v$ by Corollary \ref{cor:Stab(v)}. We also have $p=g.\Gamma$ since $g.\Gamma$ is the only circuit in $g.\Gamma$. Thus, we have $C(\st_\Gamma(v),\Gamma^e)=\{g.\Gamma \mid g\in G_v\}$. Define the induced subgraph $\Lambda$ of $\Gamma^e$ by $\Lambda = \bigcup_{g \in G_v} g.\Gamma$. $\Lambda$ is finite by $|\Gamma|<\infty$ and $|G_v|<\infty$. We also have $C(\st_\Gamma(v),\Lambda)=\{g.\Gamma \mid g\in G_v\}=C(\st_\Gamma(v),\Gamma^e)$. This implies $|C(\st_\Gamma(v),\Lambda)|=|C(\st_\Gamma(v),\Gamma^e)|=|G_v| \in 2\NN+1$. Suppose for contradiction that there exist $\ell > 0$, a sequence $v_1,\cdots,v_\ell$ in $\Gamma^e$, and a sequence of finite induced subgraphs $\Gamma=\Gamma_0 \subset \Gamma_1 \subset \cdots \subset \Gamma_\ell$ of $\Gamma^e$, where $\Gamma_i$ is obtained by taking the double of $\Gamma_{i-1}$ along $\st_{\Gamma_{i-1}} (v_i)$ for each $i = 1,\cdots,\ell$, such that $\Lambda \subset \Gamma_\ell$. By induction on $i$, it's not difficult to see that for any $i \in \{0,\cdots,\ell\}$ and any path $q$ of length 2 without backtracking in $\Gamma_i$, we have $|C(q,\Gamma_i)| \in \{0,1\}\cup 2\NN$. In particular, $|C(\st_\Gamma(v),\Gamma_\ell)|\in \{0,1\}\cup 2\NN$. On the other hand, by $\Lambda \subset \Gamma_\ell \subset \Gamma^e$, we have $|G_v|=|C(\st_\Gamma(v),\Lambda)| \le |C(\st_\Gamma(v),\Gamma_\ell)| \le |C(\st_\Gamma(v),\Gamma^e)|=|G_v|$. This contradicts $|G_v| \in 2\NN+1$.
\end{proof}

\section{Relative hyperbolicity of graph wreath product}
\label{sec:Relative hyperbolicity of semi-direct product}

In this section, we present one application of the extension graph of graph product of groups as discussed in Section \ref{sec:Introduction}. More applications to analytic properties of graph product will be presented in the forthcoming paper. The goal of this section is to prove Theorem \ref{thm:intro semidirect product is relatively hyperbolic} and Corollary \ref{cor:intro semidirect product becomes hyperbolic}, which corresponds to Theorem \ref{thm:semidirect product is relatively hyperbolic} and Corollary \ref{cor:semidirect product becomes hyperbolic} respectively.
\begin{defn}
    Given a group $G$ acting on a simplicial graph $\Gamma$ and another group $H$, we can assign vertex groups $\G = \{G_v\}_{v \in V(\Gamma)}$ by setting $G_v = H$ for every $v \in V(\Gamma)$. For each $g \in G$, the identity map $G_v=H \ni h \mapsto h \in H=G_{gv}$ defined on each $v \in V(\Gamma)$ extends to the group automorphism $\alpha_g \colon \Gamma\G \to \Gamma\G$. The map
\begin{align}\label{eq:Aut(GammaG)}
    \alpha \colon G \ni g \mapsto \alpha_g \in \mathrm{Aut}(\Gamma\G)
\end{align}
is a group homomorphism. Hence, $\alpha$ defines the semi-direct product $\Gamma\G \rtimes G$, which is called \emph{graph wreath product}.
\end{defn}
This construction interpolates between wreath product and free product as discussed in Section \ref{sec:Introduction}. Note that in Theorem \ref{thm:semidirect product is relatively hyperbolic} and Corollary \ref{cor:semidirect product becomes hyperbolic}, the group action on a graph can invert edges.
\begin{thm}\label{thm:semidirect product is relatively hyperbolic}
    Suppose that $\Gamma$ is a fine hyperbolic graph with $\girth(\Gamma)>20$ and a finitely generated group $G$ acts on $\Gamma$ satisfying the following conditions.
    \begin{itemize}
        \item[(1)]
        $E(\Gamma) / G$ is finite and for any $e \in E(\Gamma)$, $\stab_G(e) \,(=\stab_G(e_-) \cap \stab_G(e_+))$ is finite.
        \item[(2)]
        For any $v \in V(\Gamma)$, $\stab_G(v)$ is finitely generated.
    \end{itemize}
    Let $H$ be a finite group and define $\G=\{G_v\}_{v \in V(\Gamma)}$ by $G_v=H$ for any $v \in V(\Gamma)$. Then, there exists a finite set $F \subset V(\Gamma)$ such that $\Gamma\G \rtimes G$ is hyperbolic relative to the collection $\{\, \la \stab_G(v), G_w \mid w \in \st_\Gamma(v) \ra \,\}_{v \in F}$.
\end{thm}

\begin{proof}
    By Proposition \ref{prop:extension graph is hyperbolic} and Proposition \ref{prop:extension graph is fine}, the extension graph $\Gamma^e$ is a fine hyperbolic graph. For any $g \in G$, $h \in \Gamma\G$, and $v \in V(\Gamma)$, we have $\alpha_g(h G_v h^{-1}) = \alpha_g(h) \cdot G_{gv} \cdot \alpha_g(h)^{-1} \in V(\Gamma^e)$, where $\alpha_g$ is as in \eqref{eq:Aut(GammaG)}. Hence, $\alpha_g$ induces the graph automorphism $\widetilde{\alpha}_g \colon \Gamma^e \to \Gamma^e$ such that $\widetilde{\alpha}_g (h.v) = \alpha_g(h).gv$ for any $h \in \Gamma\G$ and $v \in V(\Gamma)$. This defines the action $\widetilde{\alpha} \colon G \act \Gamma^e$. For brevity, we'll denote $\widetilde{\alpha}_g (x)$ by $gx$ for $g \in G$ and $x \in V(\Gamma^e)$. Since we can see $\forall\, g \in G, \forall\, h \in \Gamma\G, \forall\, x \in V(\Gamma^e),\, g(h.x) = \alpha_g(h).gx$, the actions of $G$ and $\Gamma\G$ on $\Gamma^e$ extend to the action $\Gamma\G \rtimes G \act \Gamma^e$. 
    
    In the following, we'll show that the action $\Gamma\G \rtimes G \act \Gamma^e$ satisfies the conditions in Theorem \ref{thm:relatively hyperbolic groups}. By $|E(\Gamma) / G| < \infty$, there exists a finite set $A \subset E(\Gamma)$ such that $E(\Gamma) = \bigcup_{g \in G}gA$. By this and $\Gamma^e = \bigcup_{g \in \Gamma\G} g.\Gamma$, we have $E(\Gamma^e) = \bigcup_{g \in \Gamma\G \rtimes G} gA$. Hence, the edge orbit set $E(\Gamma^e) \,/\,  (\Gamma\G \rtimes G)$ is finite. Let $e \in E(\Gamma)$. For any $fg \in \Gamma\G \rtimes G$, where $f \in\Gamma\G$ and $g \in G$, satisfying $fge=e$, we have $ge=e$ by Corollary \ref{cor:vertex orbits are disjoint}, which also implies $f.e=fge=e$. Hence, we have $\stab_{\Gamma\G \rtimes G}(e) = \la \stab_{\Gamma\G}(e), \stab_G(e) \ra$. Note $\stab_{\Gamma\G}(e) = G_{e_-} \times G_{e_+}$ by Remark \ref{rem:intersection of stabilizers} (1). For any $g \in \stab_G(e)$, we have $gG_{e_-}g^{-1}=G_{ge_-}=G_{e_-}$ in $\Gamma\G \rtimes G$ and this conjugation is the identity map from $G_{e_-}$ to $G_{e_-}$. The same is true for $G_{e_+}$. This implies $\la \stab_{\Gamma\G}(e), \stab_G(e) \ra \cong G_{e_-} \times G_{e_+} \times \stab_G(e)$. Hence, $\stab_{\Gamma\G \rtimes G}(e)$ is finite for any $e \in E(\Gamma)$. By $E(\Gamma^e) = \bigcup_{g \in \Gamma\G \rtimes G} gA$, the edge stabilizer of every edge in $E(\Gamma^e)$ is finite.

    Let $v \in V(\Gamma)$. For any $fg \in \Gamma\G \rtimes G$, where $f \in\Gamma\G$ and $g \in G$, satisfying $fgv=v$, we have $gv=v$ by Corollary \ref{cor:vertex orbits are disjoint}, which also implies $f.v=fgv=v$. Hence, we have $\stab_{\Gamma\G \rtimes G}(v) = \la \stab_{\Gamma\G}(v), \stab_G(v) \ra$. By $|E(\Gamma) / G| < \infty$, there exists a finite set $F_v \subset \lk_\Gamma(v)$ such that $\lk_\Gamma(v) = \bigcup_{g \in \stab_G(v)} gF_v$. Hence, $\la \stab_{\Gamma\G}(v), \stab_G(v) \ra = \la \stab_G(v), G_v, G_w \mid w \in F_v \ra$. Since $\stab_G(v)$ is finitely generated, this and $|H|<\infty$ imply that $\stab_{\Gamma\G \rtimes G}(v)$ is finitely generated for any $v \in V(\Gamma)$. 
    
    By $|E(\Gamma) / G| < \infty$, there exists a finite set $F_0 \subset V(\Gamma)$ such that $V(\Gamma) = \bigcup_{g \in G} gF_0$. Define $F$ by $F=\{v \in F_0 \mid |\stab_{\Gamma\G \rtimes G}(v)|=\infty\}$. Since $V(\Gamma) = \bigcup_{g \in G} gF_0$ implies $V(\Gamma^e) = \bigcup_{g \in \Gamma\G \rtimes G} gF_0$ as above, for any $x\in V(\Gamma^e)$, the group $\stab_{\Gamma\G \rtimes G}(x)$ is either finite or conjugate to $\stab_{\Gamma\G \rtimes G}(v)$ for some $v \in F$. Also, $\Gamma\G \rtimes G$ is finitely generated since we have $V(\Gamma) = \bigcup_{g \in G} gF_0$ and $G$ is finitely generated. By Theorem \ref{thm:relatively hyperbolic groups}, $\Gamma\G \rtimes G$ is hyperbolic relative to the collection $\{\, \la \stab_G(v), G_w \mid w \in \st_\Gamma(v) \ra \,\}_{v \in F} \,(= \{\stab_{\Gamma\G \rtimes G} (v)\}_{v \in F})$.
\end{proof}

Before deducing Corollary \ref{cor:semidirect product becomes hyperbolic} from Theorem \ref{thm:semidirect product is relatively hyperbolic}, we have to prepare Lemma \ref{lem:free product semidirect product finite group is hyperbolic} below.

\begin{lem}\label{lem:free product semidirect product finite group is hyperbolic}
    Suppose that a finite group $G$ acts on a graph $\Gamma$ satisfying $|V(\Gamma)| < \infty$ and $E(\Gamma) = \emptyset$. Let $H$ be a finite group and define $\G=\{G_v\}_{v \in V(\Gamma)}$ by $G_v=H$ for any $v \in V(\Gamma)$. Then, $\Gamma\G \rtimes G$ is hyperbolic.
\end{lem}

\begin{proof}
    Note $\Gamma\G = \ast_{v \in V(\Gamma)} G_v$ by $E(\Gamma) = \emptyset$. Define $K$ by $K = \Gamma\G \rtimes G$ for brevity. The free product $\ast_{v \in V(\Gamma)} G_v$ is hyperbolic by $\max\{|V(\Gamma)|,|H|\} < \infty$ and the set $S \subset \Gamma\G$ defined by $S = \bigcup_{v \in V(\Gamma)} G_v$ is a finite generating set of $\Gamma\G$. The set $T \subset K$ defined by $T = S\cup G$ is a finite generating set of $K$ and the embedding $\iota \colon (\Gamma\G,d_S) \to (K,d_T)$ defined by $\forall\, x\in \Gamma\G,\iota(x)=x$ is isometric (i.e. $\forall\, x\in \Gamma\G, \|x\|_T=\|x\|_S$) since we have $gSg^{-1}=S$ for any $g \in G$. By this and $|K/\Gamma\G|<\infty$, the map $\iota$ is quasi-isometric. Hence, $K$ is hyperbolic as well.
\end{proof}

\begin{cor}\label{cor:semidirect product becomes hyperbolic}
    Suppose that $\Gamma$ is a locally finite hyperbolic graph with $\girth(\Gamma)>20$ and a group $G$ acts on $\Gamma$ properly and cocompactly. Let $H$ be a finite group and define $\G=\{G_v\}_{v \in V(\Gamma)}$ by $G_v=H$ for any $v \in V(\Gamma)$. Then, $\Gamma\G \rtimes G$ is hyperbolic.
\end{cor}

\begin{proof}
    Note that $G$ is finitely generated since $G$ acts on $\Gamma$ properly and cocompactly. Also, we have $|E(\Gamma) / G|<\infty$ and $|\stab_G(v)|<\infty$ for any $v \in V(\Gamma)$. By Theorem \ref{thm:semidirect product is relatively hyperbolic}, there exists a finite set $F \subset V(\Gamma)$ such that $\Gamma\G \rtimes G$ is hyperbolic relative to the collection $\{\, \la \stab_G(v), G_w \mid w \in \st_\Gamma(v) \ra \,\}_{v \in F}$. For any $v\in F$, we have
    \[
    \stab_{\Gamma\G \rtimes G}(v)=\la \stab_G(v), G_w \mid w \in \st_\Gamma(v) \ra = G_v \times \big( (\ast_{w \in \lk_\Gamma(v)}G_w)\rtimes \stab_G(v) \big).
    \]
    Hence, $\stab_{\Gamma\G \rtimes G}(v)$ is hyperbolic by $\max\{ |\lk_\Gamma(v)|, |H|, |\stab_G(v)| \}<\infty$ and Lemma \ref{lem:free product semidirect product finite group is hyperbolic}. By \cite[Corollary 2.41]{Osi06}, $\Gamma\G \rtimes G$ is hyperbolic.
\end{proof}

In fact, Theorem \ref{thm:semidirect product is relatively hyperbolic}, Lemma \ref{lem:free product semidirect product finite group is hyperbolic}, and Corollary \ref{cor:semidirect product becomes hyperbolic} can be generalized respectively to Theorem \ref{thm:semidirect product is relatively hyperbolic more general}, Lemma \ref{lem:free product semidirect product finite group is hyperbolic more general}, and Corollary \ref{cor:semidirect product becomes hyperbolic more general} by varying the finite group $H$ from one orbit of the action $G \act \Gamma$ to another. We omit their proofs since they can be proven in the same way. Indeed, in the proofs of Theorem \ref{thm:semidirect product is relatively hyperbolic}, Lemma \ref{lem:free product semidirect product finite group is hyperbolic}, and Corollary \ref{cor:semidirect product becomes hyperbolic}, we used the fact that vertex groups are all the same just to make sure that the semi-direct product $\Gamma\G \rtimes G$ is well-defined. The semi-direct product $\Gamma\G \rtimes G$ in Theorem \ref{thm:semidirect product is relatively hyperbolic more general}, Lemma \ref{lem:free product semidirect product finite group is hyperbolic more general}, and Corollary \ref{cor:semidirect product becomes hyperbolic more general} can be defined similarly by permuting vertex groups, because vertex groups in the same orbit are the same.

\begin{thm}\label{thm:semidirect product is relatively hyperbolic more general}
    Suppose that $\Gamma$ is a fine hyperbolic graph with $\girth(\Gamma)>20$ and a finitely generated group $G$ acts on $\Gamma$ satisfying the conditions (1) and (2) in Theorem \ref{thm:semidirect product is relatively hyperbolic}. Let $\G=\{G_v\}_{v \in V(\Gamma)}$ be a collection of finite groups such that $G_{gv}=G_v$ for any $g \in G$ and $v \in V(\Gamma)$. Then, there exists a finite set $F \subset V(\Gamma)$ such that $\Gamma\G \rtimes G$ is hyperbolic relative to the collection $\{\, \la \stab_G(v), G_w \mid w \in \st_\Gamma(v) \ra \,\}_{v \in F}$.
\end{thm}

\begin{lem}\label{lem:free product semidirect product finite group is hyperbolic more general}
    Suppose that a finite group $G$ acts on a graph $\Gamma$ satisfying $|V(\Gamma)| < \infty$ and $E(\Gamma) = \emptyset$. Let $\G=\{G_v\}_{v \in V(\Gamma)}$ be a collection of finite groups such that $G_{gv}=G_v$ for any $g \in G$ and $v \in V(\Gamma)$. Then, $\Gamma\G \rtimes G$ is hyperbolic.
\end{lem}

\begin{cor}\label{cor:semidirect product becomes hyperbolic more general}
    Suppose that $\Gamma$ is a locally finite hyperbolic graph with $\girth(\Gamma)>20$ and a group $G$ acts on $\Gamma$ properly and cocompactly. Let $\G=\{G_v\}_{v \in V(\Gamma)}$ be a collection of finite groups such that $G_{gv}=G_v$ for any $g \in G$ and $v \in V(\Gamma)$. Then, $\Gamma\G \rtimes G$ is hyperbolic.
\end{cor}



\providecommand{\bysame}{\leavevmode\hbox to3em{\hrulefill}\thinspace}
\providecommand{\MR}{\relax\ifhmode\unskip\space\fi MR }
\providecommand{\MRhref}[2]{%
  \href{http://www.ams.org/mathscinet-getitem?mr=#1}{#2}
}
\providecommand{\href}[2]{#2}

\vspace{5mm}

\noindent  Department of Mathematics, Vanderbilt University, Nashville 37240, U.S.A.

\noindent E-mail: \emph{koichi.oyakawa@vanderbilt.edu}

\end{document}